\documentclass[11pt,letterpaper]{amsart}

\usepackage{amsxtra}
\usepackage[utf8]{inputenc}
\usepackage{amsthm,amsmath,amssymb,oldgerm,latexsym, amscd, calc, microtype, booktabs}
\usepackage[square]{natbib}
\usepackage{graphicx}
\usepackage[dvipsnames]{xcolor}
\usepackage{caption}
\usepackage{tcolorbox}
\usepackage[all]{xy}

\usepackage{hyperref}
\hypersetup{
    colorlinks,
    citecolor=blue,
    filecolor=black,
    linkcolor=red,
    urlcolor=black
}
\bibpunct{[}{]}{,}{n}{}{;}

\numberwithin{equation}{section}
\newcommand{\zz}{\mathbb{Z}}

\newcommand{\qq}{\mathbb{Q}}
\newcommand{\ff}{\mathbb{F}}
\newcommand{\pp}{\mathbb{P}}

\newcommand{\cX}{\mathcal{X}}
\newcommand{\cO}{\mathcal{O}}

\newcommand{\tcX}{\widetilde{\cX}}

\newcommand{\hh}{\mathbb{H}}

\DeclareMathOperator{\can}{\mu_{can}}
\DeclareMathOperator{\hyp}{\mu_{hyp}}

\DeclareMathOperator{\spec}{Spec}

\newcommand{\Z}{\mathbb{Z}}
\newcommand{\N}{\mathbb{N}}

\newcommand{\Ga}{\Gamma}

\renewcommand{\Im}{\operatorname{Im}}

\begin{document}
\theoremstyle{plain}
\newtheorem{thm}{Theorem}[section]
\newtheorem{lem}[thm]{Lemma}
\newtheorem{prop}[thm]{Proposition}
\newtheorem{cor}[thm]{Corollary}
\theoremstyle{definition}
\newtheorem{defn}[thm]{Definition} 
\theoremstyle{definition}
\newtheorem{rem}[thm]{Remark}
\newtheorem{notn}[thm]{Notation}
\newtheorem{notns}[thm]{Notations}
\newtheorem{example}[thm]{Example}

\makeatletter
\def\imod#1{\allowbreak\mkern10mu({\operator@font mod}\,\,#1)}
\makeatother
\title{The intersection matrices of $X_0(p^r)$ and some applications}

\author{Debargha Banerjee}
\author{Priyanka Majumder}

\address{Debargha Banerjee. IISER PUNE, INDIA}
\email{debargha.banerjee@gmail.com}

\address{Priyanka Majumder. IISER PUNE, INDIA}
\email{pmpriyanka57@gmail.com}

\author{Chitrabhanu Chaudhuri}

\address{Chitrabhanu Chaudhuri. NISER BHUBANESWAR, INDIA}
\email{chitrabhanu@gmail.com}
\thanks{The first named author is partially supported by the SERB grant MTR/2017/000357 and CRG/2020/000223. The paper is dedicated to the loving memory of  Professor Bas Edixhoven. The first named author started the project after a stimulating conference at the Leiden University in 2013. 
        The second named author was partially supported by the  CRG/2020/000223 and IISER, Pune post-doctoral fellowship.}
\begin{abstract} 
We compute intersection matrices for modular curves of the form $X_0(p^r)$ with $r \in \{3,4\}$ and as an application, we compute an asymptotic expression for the  Arakelov self-intersection number of the relative
  dualizing sheaf of Edixhoven's  minimal regular model  for the modular  curve $X_0(p^r)$ over 
  $\qq$ with $r$ as above. This computation will be useful to understand an effective version 
  of the Bogolomov conjecture for the stable models of modular curves $X_0(p^r)$ with $r \in \{3,4\}$ and obtain 
  a bound on the stable Faltings height for those curves.

\end{abstract}

\maketitle

\setcounter{tocdepth}{1}
\tableofcontents{}

\section{Introduction}
There is a considerable interest to understand nice integral models of the modular curves starting with the classical work of Deligne--Rapoport and Katz. In all the these works, the special fibers at the primes dividing the levels are problematic and 
if the higher powers of  primes divide levels then the special fibers become very hard to understand. For the modular curves of the form $X_0(N)$,  Edixhoven constructed the regular integral models of these curves \cite{MR1056773}. 
These models may or may not be {\it minimal}. Soon after, Coleman considered these curves in the category of rigid analytic spaces and constructed eigencurves using these. 
Coleman invented several nice properties about these curves (for instance \cite[Theorem 1.2]{MR2114670}). Stable models of the curves 
in the category of rigid analytic spaces are investigated by Coleman--McMurdy \cite{MR2290590}, \cite{MR2661537}. Recently, semi-stable models of modular curves of arbitrary levels are studied by 
Weinstein~\cite{MR3529120}. 
 
 To study the theory of perfectoid spaces, Scholze considers these schemes in the category of adic spaces (see \cite{MR3418533}, \cite{MR4446467}, and \cite{MR4433445}).  This in turn helps us to understand modular curves with
{\it powerful levels}. However, if we wish to understand explicit arithmetic aspects (for instance Zhang's proof of Bogomolov's conjecture), we need information about intersection matrices of the components of special fibres of the minimal regular models of these curves considered as schemes.  

In the article, we compute intersection matrices of Edixhoven's minimal regular models of modular curves of the form $X_0(p^r)$ with $r \in \{3,4\}$. Our method can be generalised for higher values of $r$ ({\it powerful levels}) but since the intersection matrices become unmanageable, we restrict ourselves to these particular values of $r$. 

As an application, we compute an asymptotic expression for the Arakelov self-intersection numbers \cite{MR0466150} for the minimal regular models of the modular curves of the form as above. We hope to prove an effective version of the Bogomolov's conjecture for these particular modular curve $X_0(p^r)$ using our result in a subsequent work. 

Sufficiently good upper bounds for the self-intersection of the relative dualizing sheaf play a crucial role in the work of Edixhoven and his co-authors, when estimating the running time of his algorithm regarding fast computation of Fourier coefficients of modular forms and 
 for determining Galois representations \cite{MR2857099}. 

An asymptotic expression for the Arakelov self-intersection number of the relative dualizing sheaf of the minimal regular model over $\zz$ for the modular curve $X_0(N)$ is obtained from Abbes--Ullmo \cite{AbbesUllmo} and Michel--Ullmo~\cite{MichelUllmo} with  certain assumption on $N$ (basically squarefree). Following the strategy of Abbes--Ullmo ~\cite{AbbesUllmo}, Mayer \cite{MR3232770} computed these asymptotic expressions for the case of
modular curves $X_1(N)$ with some mild {\it squarefee} restriction on  $N$. 

In a similar spirit, Grados--von Pippich \cite{https://doi.org/10.48550/arxiv.2205.11437} computed this asymptotic expression for the case of modular curves $X(N)$
 with some restriction on $N$.  Recently, Banerjee--Borah--Chaudhuri \cite{BanerjeeBorahChaudhuri}
proved this asymptotic expression for curves $X_0(p^2)$ with a prime number $p$ by following mostly the lines
of proof in \cite{AbbesUllmo}. Banerjee--Chaudhuri \cite{debarghachitra} proved an effective Bogomolov conjecture and found an asymptotic expression 
for the stable Faltings heights for the modular curves of the form $X_0(p^2)$ with a prime number $p$.  

Recall that 
~\cite{MR0466150} the Arakelov self-intersection of the relative dualizing sheaf on modular curves is the sum of two parts: 
\begin{itemize}
\item 
{\it Analytic} part is given in terms of the canonical Green’s functions evaluated at the cusps.
\item 
{\it Geometric} part is given by the intersection of vertical divisors and divisors coming from the cusps.
\end{itemize}

Till now for all modular curves (cf.  \cite{AbbesUllmo} and \cite{MichelUllmo} for $X_0(N)$, \cite{MR3232770} for $X_1(N)$, and \cite{BanerjeeBorahChaudhuri} for $X_0(p^2)$), 
the leading term in the asymptotics for the Arakelov self-intersection number of the relative dualizing sheaf of the minimal regular model over $\zz$ for the modular curve $X_0(N)$ is $3g_{N}\log N$. In all  these instances of modular curves, $g_{N}\log N$ comes from the geometric part, and $2g_{N}\log N$ comes from the analytic part.

Recently, Majumder--von Pippich \cite{vonPippichMajumder} (also see \cite{Majumder:Thesis}), proved that the leading term in the analytic part of the Arakelov self-intersection of the relative dualizing sheaf on modular curves $X_0(N)$ of genus $g_N$ is $2g_{N}\log N$ for {\it any} $N \in \N$. 
Note that this can also be deduced by suitably modifying  \cite[\S 4]{BanerjeeBorahChaudhuri} for prime power level.
It is natural to study the algebraic part of the Arakelov self-intersection of the relative dualizing sheaf and compute the asymptotics.

In this article, we derive the asymptotic expressions for the geometric part of the Arakelov self-intersection number of the relative dualizing sheaf of the minimal regular model over $\zz$ for modular curves $X_0(p^3)$ and $X_0(p^4)$. This method can be definitely generalized for $X_0(p^r)$ with higher values of $r$.

To derive the asymptotics for the geometric part of the Arakelov self-intersection number of the relative dualizing sheaf, we follow the line of proof from \cite{BanerjeeBorahChaudhuri}. For the modular curve $X_0(p^r)$ with $r \in \{3,4\}$, we compute the intersection matrices of the special fiber of the Edixhoven's regular model \cite{MR1056773} for the modular curve $X_0(p^r)$. 
Similar to $r=2$, these intersection matrices depend on the parity of $p$ modulo $12$.

For $r=3$, we observe that Edixhoven's model is the minimal regular model, and we denote this minimal regular model by $\cX_0(p^3)$ (see \S~\ref{incidince_matrric_power3}). 
For $r=4$, we compute incidence matrices of the special fiber of the Edixhoven's regular model \cite{MR1056773}. However in this situation, Edixhoven's model is not minimal. 
We derive the minimal regular model $\cX_0(p^4)$ from the Edixhoven's model by the three successive blow downs similar to $r=2$ (see \S~\ref{incidince_matrric_power4}).

On $\cX_0(p^r)$ with $r \in \{3,4\}$ we have the canonical divisor $K_{\cX_0(p^r)}$, and we have horizontal divisors $H_m$ (for $m \in \{ 0, \infty\}$) corresponding to the cusps $0$ and $\infty$.
By solving a system of linear equations using the software {\tt SAGE} \cite{sage}, we construct the divisors $V_m$ for $m$ as above  such that the divisors $D_m = K_{\cX_0(p^r)} - (2 g_{p^r} - 2) H_m + V_m$ are orthogonal to all vertical divisors of $\cX_0(p^r)$ (see \S \ref{Construction_with_sage_power3} and \S \ref{Construction_with_sage_power4}). 
Then using the Faltings--Hriljac~\cite{MR740897}, we prove that the leading term in the geometric part of the Arakelov self-intersection of the relative dualizing sheaf of $\cX_0(p^r)$ is $g_{p^r}\log (p^r)$. Note that upto $r=2$  \cite{BanerjeeBorahChaudhuri}, these vertical divisors are obtained using only one carefully chosen component. 
For $r \geq 3$, these carefully chosen vertical divisor (needed to apply the theorem of Faltings--Hriljac) are {\it not} supported at  one irreducible component. 
\subsection{The main theorem}
We now state the main theorem of our article:

\begin{thm}
For $r \in \{3,4\}$, the Arakelov self-intersection numbers on the minimal regular model $\cX_0(p^r)$ satisfy the following asymptotic formula:

\begin{align*}
    \overline{\omega}^2_{\cX_0(p^r)}=3g_{p^r}\log (p^r)+o(g_{p^r}\log p)\,\ \text{as}\,\ p\to \infty.
\end{align*}
\end{thm}
From the above theorem and the modular curves studied so far it is tempting to believe that contribution from the geometric part in the asymptotic expression for the modular curve $X_0(N)$  should always be $g_N \log(N)$ for any positive $N$. 
For $r \in \N$, regular model for $X_0(p^{r})$ is obtained by Edixhoven. These models however are not always minimal. When $r$ is even, Edixhoven's model is not minimal because it has  $-1$-curves. In these cases, the minimal regular model should be  obtained by successive contractions. 
On the other hand for all odd $r$, Edixhoven's regular model is already minimal. 

\section{Preliminaries}
\subsection{The canonical Green's functions for \texorpdfstring{$X_0(N)$}{X}}

We have the hyperbolic metric $\hyp(z)$ on $X_0(N)$, which is compatible with the complex structure of $X_0(N)$, and has constant negative curvature equal to $-1$. Locally, we have
\begin{align*}
\hyp(z)=\frac{i}{2}\cdot\frac{dz\wedge d\overline{z}}{{\Im(z)}^{2}}.
\end{align*}
Let $S_{2}(\Gamma_0(N))$ denote the $\mathbb{C}$-vector space of {cusp forms} of weight $2$ with respect to the congruence subgroup $\Gamma_0(N)$ equipped with the Petersson inner product
\begin{align*}
\langle f, g \rangle_\mathrm{pet}= \int_{X_0(N)} f(z) \overline{g(z)}\Im(z)^2 \hyp(z)\,\ \text{with} \,\ f,g\in S_{2}(\Gamma_0(N)).
\end{align*}
Let $\left\lbrace f_{1},\ldots,f_{g_{\Gamma}}\right\rbrace$ denote an orthonormal basis of $S_{2}(\Gamma_0(N))$ with respect to the Petersson 
inner product. Then, the {canonical metric on} $X$ is defined by
\begin{equation*}
 \can(z)=\frac{i}{2g_\Gamma} \sum_{j=1}^{g_{\Gamma}}\left|f_{j}(z)\right|^{2}dz\wedge d\overline{z}.
\end{equation*}
Locally, we have the following relation. 
\begin{align*}
    \can(z)=F(z)\hyp(z),\quad \text{where}\,\ F(z)=\frac{\Im(z)}{g_\Gamma} \sum_{j=1}^{g_{\Gamma}}\left|f_{j}(z)\right|^{2}.
\end{align*}
The canonical Green's function $\mathcal{G}_{\mathrm{can}}(z,w)$ of $X_0(N)$ is a function on $X_0(N)\times X_0(N)$, which is smooth away from the diagonal, and has a logarithmic singularity along the diagonal. Away from the diagonal it is uniquely characterized by following differential equation.
\begin{align*}
d_z d_z^c\, \mathcal{G}_{\mathrm{can}}(z,w)+ \delta_{w}(z)=\can(z), \,\ \text{where} \,\ z, w\in X_0(N).
\end{align*}
Here $d_z=\left(\partial_z + \overline{\partial}_z \right), d_z^{c}= \left( \partial_z - \overline{\partial}_z
\right)\slash 4\pi i$, then we have $d_zd_z^{c}= -{\partial_z\overline{\partial}_z}\slash{2\pi i}$. The $\delta_{w}(z)$ is the {Dirac delta distribution}.
The canonical Green's function satisfies the normalization condition
\begin{align*}
 \int_{X_0(N)}\mathcal{G}_{\mathrm{can}}(z,w)\can(z)=0 \,\ \text{with}\,\ w\in X_0(N).
\end{align*}
For example, in \cite{Lang_Book}, p. 26, Lang has explicitly written down canonical Green’s functions on quotient spaces of genus zero having elliptic fixed points. For compact hyperbolic Riemann surfaces, Jorgenson--Kramer \cite{jorgenson_kramer_2006} studied bounds for the canonical Green's functions using hyperbolic heat kernels. Later, these bounds are extended for more general non-compact orbisurfaces by Aryasomayajula \cite{AnilPaper}. 
\subsection{The genus of \texorpdfstring{$X_0(N)$}{}}
For a positive integer $N$, we know that the genus $g_N$ of $X_0(N)$ is given by
\[
g_N=1+\frac{\mu}{12}-\frac{\nu_2}{4}-\frac{\nu_3}{3}-\frac{\nu_{\infty}}{2},
\]
where $\mu=[SL_2(\mathbb{Z}): \Gamma_0(N)]$, $\nu_2$ and $\nu_3$ are numbers of elliptic fixed points of order $2$ and $3$ respectively, and $\nu_{\infty}$ is the number of cusps of $X_0(N)$. These numbers can be computed using the following formulae:
\begin{align*}
\mu & =N\prod_{p \mid N}\left(1+\frac{1}{p}\right);\quad   \nu_{\infty}  =\sum_{d\mid N, d>0}\phi\left( \gcd \left( d, \frac{N}{d}\right) \right);\\
\nu_2 &=
\begin{cases}
0 & \text{if $4\mid N$},\\
\prod_{p\mid N} \left(1+\left(\frac{-1}{p}\right)\right) & \text{otherwise};\\
\end{cases}\\
\nu_3 &=
\begin{cases}
0 & \text{if $9\mid N$},\\
\prod_{p\mid N} \left(1+\left(\frac{-3}{p}\right)\right)& \text{otherwise}.
\end{cases}
\end{align*}
The number of elliptic fixed points of $\Ga_0(N)$ is equal to $e_N=\nu_2+\nu_3$. Here $(\frac{\cdot}{p})$ is the quadratic residue symbol and $\phi$ is the Euler function. 
\begin{example}\label{genusp3}
Let $N=p^3$, where $p$ is an odd prime. Then we have $[SL_2(\mathbb{Z}): \Gamma_0(p^3)] = p^2(p+1)$, and $\nu_{\infty}=2p$. Then we get
\begin{align*}
g_{p^3}= 1+ \frac{p(p+4)(p-3)-c}{12},
\end{align*}
where
\begin{equation*}
c= \begin{cases}14  & \text{if}\,\ p \equiv 1 \,(\hspace{-0.4cm}\mod 12),\\
6  & \text{if}\,\ p \equiv 5 \,(\hspace{-0.4cm}\mod 12),\\
8  & \text{if}\,\ p \equiv 7 \,(\hspace{-0.4cm}\mod 12),\\
0  & \text{if}\,\ p \equiv 11 \,(\hspace{-0.4cm}\mod 12).
\end{cases}
\end{equation*}
\end{example}

\begin{example}
Let $N=p^4$, where $p$ is an odd prime. Then we have $[SL_2(\mathbb{Z}): \Gamma_0(p^4)] = p^3(p+1)$, and $\nu_{\infty}=p(p+1)$. Then we get
\begin{align*}
g_{p^4}= 1+ \frac{p(p+1)(p^2-6)-c}{12},
\end{align*}
where $c$ is same as in Example \ref{genusp3}.

\end{example}
\section{Intersection matrices of minimal regular models}
\subsection{For the modular curve  \texorpdfstring{$X_0(p^3)$} {}}\label{incidince_matrric_power3}
Let $\cX_0(p^3)$ be the regular model constructed by Edixhoven \cite{MR1056773}. The regular model
depends on the residue $r$ of $p$ modulo 12. However, in all the cases the special fiber has the 
components $C_{3,0},\, C_{0,3},\,C_{2,1},\, C_{1,2}$, along with some other components which depend on $r$. 
The multiplicity of the component $C_{a,b}$ ($C_{a,b} \in\lbrace C_{3,0}, C_{0,3}, C_{2,1}, C_{1,2}\rbrace$) is $\phi\big(p^{\min(a,b)}\big)$. From \cite[p. 158]{MR1326710}, we know that the local intersection of any two vertical components $C_{a,b}$ and $C_{c,d}$ (for $C_{a,b}, C_{c,d} \in\lbrace C_{3,0}, C_{0,3}, C_{2,1}, C_{1,2}\rbrace$) at a supersingular point $\alpha_i$ (shown in figures \ref{fig:1mod12}, \ref{fig:5mod12}, \ref{fig:7mod12}, and \ref{fig:11mod12}) is given by the following formula.
\begin{equation}\label{localintersection}
    i_{\alpha_i}\big(C_{a,b},\, C_{c,d}\big)= \begin{cases} 1& \text{if}\,\ (a-b)(c-d)\leq 0,\\
p^{\min(|a-b|, |c-d|)}  & \text{if}\,\ (a-b)(c-d)> 0.
    \end{cases}
\end{equation}
In the following subsections, we shall explicitly describe the special fiber of the minimal regular model 
$\cX_0(p^3)$. We shall also compute the local intersection numbers among the various components in the fiber. 
The Arakelov intersection numbers in this case are obtained by simply multiplying the local intersection numbers 
by $\log(p)$. 
\begin{figure}
  \begin{center}
    \includegraphics[scale=0.4]{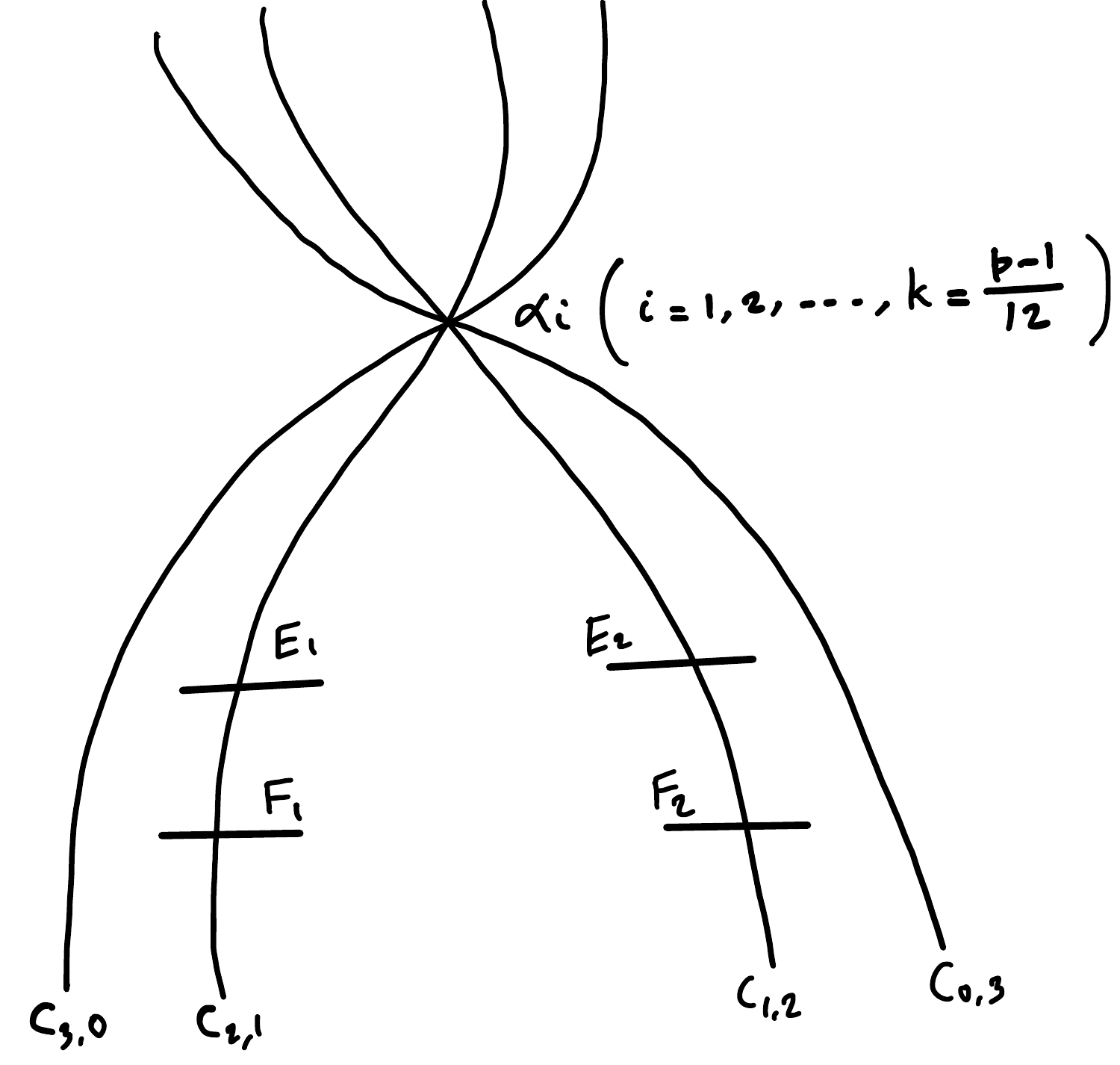}
  \end{center}      
  \caption{The special fiber $\cX_0(p^3)_{\ff_p}$ when $p \equiv 1 \pmod{12}$.} 
  \label{fig:1mod12} 
\end{figure}
\subsection{Case \texorpdfstring{$p \equiv 1 \pmod{12}$}{}}
Following Edixhoven~\cite{MR1056773}, we draw the special fiber $V_p = \cX_0(p^3)_{\ff_p}$ in Figure~\ref{fig:1mod12}, where each component is a $\pp^1$. In this case, both $j=1728$ and $j=0$ are ordinary.
\begin{prop} \label{pmod1imp}
  The local intersection numbers of the vertical components supported on the special fiber of $\cX_0(p^3)$ for $p \equiv 1 \pmod{12}$ are given 
  by the following matrix:
  
  \begin{equation*}
    \renewcommand*{\arraystretch}{2}
    \begin{array}{l|cccccccc}  
      & C_{3,0} & C_{0,3} &  C_{2,1} & \hphantom{0} C_{1,2} \hphantom{0} & \hphantom{0}  E_1 \hphantom{0} & \hphantom{0} E_2
      \hphantom{0} & \hphantom{0} F_1 \hphantom{0} & \hphantom{0} F_2 \\ \hline
      C_{3,0} & -\frac{p^2(p-1)}{12} & \frac{p-1}{12} & \frac{p(p-1)}{12} & \frac{p-1}{12} & 0 & 0 & 0 & 0 \\
      C_{0,3} & \frac{p-1}{12} & -\frac{p^2(p-1)}{12} & \frac{p-1}{12} & \frac{p(p-1)}{12} & 0 & 0 & 0 & 0 \\
      C_{2,1} & \frac{p(p-1)}{12} & \frac{p-1}{12} & -\frac{p+5}{6} & \frac{p-1}{12} & 1 & 0 & 1 & 0 \\ 
       C_{1,2} & \frac{p-1}{12} & \frac{p(p-1)}{12} & \frac{p-1}{12} & -\frac{p+5}{6} & 0 & 1 & 0 & 1 \\ 
      E_1 & 0 & 0 & 1 & 0 & -2 & 0 & 0 & 0 \\
      E_2 & 0 & 0 & 0 & 1 & 0 & -2 & 0 & 0 \\
      F_1 & 0 & 0 & 1 & 0 & 0 & 0 & -3 & 0 \\
      F_2 & 0 & 0 & 0 & 1 & 0 & 0 & 0 & -3.
    \end{array}
  \end{equation*}
In the above matrix, $E_1, E_2$ correspond to $j=1728$, and $F_1, F_2$ correspond to $j=0$.
\end{prop}
\begin{proof}

  The self-intersections $E_1^2=-2,\, E_2^2=-2, \, F_1^2=-3$ and $F_2^2=-3$ are calculated by Edixhoven (see \cite[Fig. 1.3.3.1, Fig. 1.3.3.3, Fig 1.3.6.1 and Fig. 1.3.6.3 ]{MR1056773}). The multiplicities of $E_1$ and $E_2$ are both equal to $\frac{p-1}{2}$. The multiplicities of $F_1$ and $F_2$ are both equal to $\frac{p-1}{3}$. 
  
  Now, by using the formula \eqref{localintersection}, for $i=1,\ldots, k=\frac{p-1}{12}$ we get
  \begin{align*}
      & i_{\alpha_i}\big(C_{3,0},\, C_{2,1}\big)= p, \,\ \,\ i_{\alpha_i}\big(C_{3,0},\, C_{1,2}\big)= 1, \\
      & i_{\alpha_i}\big(C_{0,3},\, C_{1,2}\big)= p,\,\ \,\ 
      i_{\alpha_i}\big(C_{0,3},\, C_{2,1}\big)= 1,
  \end{align*}
  Then we have the following intersection numbers.
  \begin{align*}
      & C_{3,0}\cdot C_{2,1}= \sum_{i=1}^{k}i_{\alpha_i}\big(C_{3,0},\, C_{2,1}\big)= \frac{p(p-1)}{12},\\
      &C_{0,3}\cdot C_{1,2}= \sum_{i=1}^{k}i_{\alpha_i}\big(C_{0,3},\, C_{1,2}\big)= \frac{p(p-1)}{12},
      \end{align*}
      and we have
      \begin{align*}
      &C_{3,0}\cdot C_{1,2}= \sum_{i=1}^{k}i_{\alpha_i}\big(C_{3,0},\, C_{1,2}\big)= \frac{p-1}{12},\\&
       C_{0,3}\cdot C_{2,1}= \sum_{i=1}^{k}i_{\alpha_i}\big(C_{0,3},\, C_{2,1}\big)= \frac{p-1}{12}.
  \end{align*}
  Since $V_p$ is the principal divisor $(p)$, we must have $V_p \cdot D = 0$ for any vertical divisor $D$. Moreover, 
  $$V_p = C_{3,0} + C_{0,3} + (p-1)\big( C_{2,1} + C_{1,2}\big)+ \frac{p-1}{2}\big( E_1+E_2\big)  + \frac{p-1}{3}\big( F_1+F_2\big)$$ is the linear combination of all 
  the prime divisors of the special fiber counted with multiplicities. All the other intersection numbers can be 
  easily calculated using these information. More precisely, $V_p\cdot E_1 = 0 \,\ \text{gives}\,\ C_{2,1}\cdot E_1 = 1$, $V_p\cdot E_2 = 0 \,\ \text{gives}\,\ C_{1,2}\cdot E_2 = 1$, $V_p\cdot F_1 = 0 \,\ \text{gives}\,\ C_{2,1}\cdot F_1 = 1$, $V_p\cdot E_2 = 0 \,\ \text{gives}\,\ C_{1,2}\cdot F_2 = 1$, $V_p\cdot C_{2,1} = 0 \,\ \text{gives}\,\ C_{2,1}^2 = -\frac{p+5}{6}$, $V_p\cdot C_{1,2} = 0 \,\ \text{gives}\,\ C_{1,2}^2 = -\frac{p+5}{6}$, $V_p\cdot C_{3,0} = 0 \,\ \text{gives}\,\ C_{3,0}^2 = -\frac{p^2(p-1)}{12}$, and $V_p\cdot C_{0,3} = 0 \,\ \text{gives}\,\ C_{0,3}^2 = -\frac{p^2(p-1)}{12}$. This completes the proof.
  \end{proof}
\subsection{Case \texorpdfstring{$p \equiv 5 \pmod{12}$}{}}
Following Edixhoven~\cite{MR1056773}, we draw the special fiber $V_p = \cX_0(p^3)_{\ff_p}$ which is described by Figure~\ref{fig:5mod12}, where each component is a $\pp^1$. In this case, $j=1728$ is ordinary and $j=0$ is supersingular.

\begin{figure}[h]
  \begin{center}
    \includegraphics[scale=0.45]{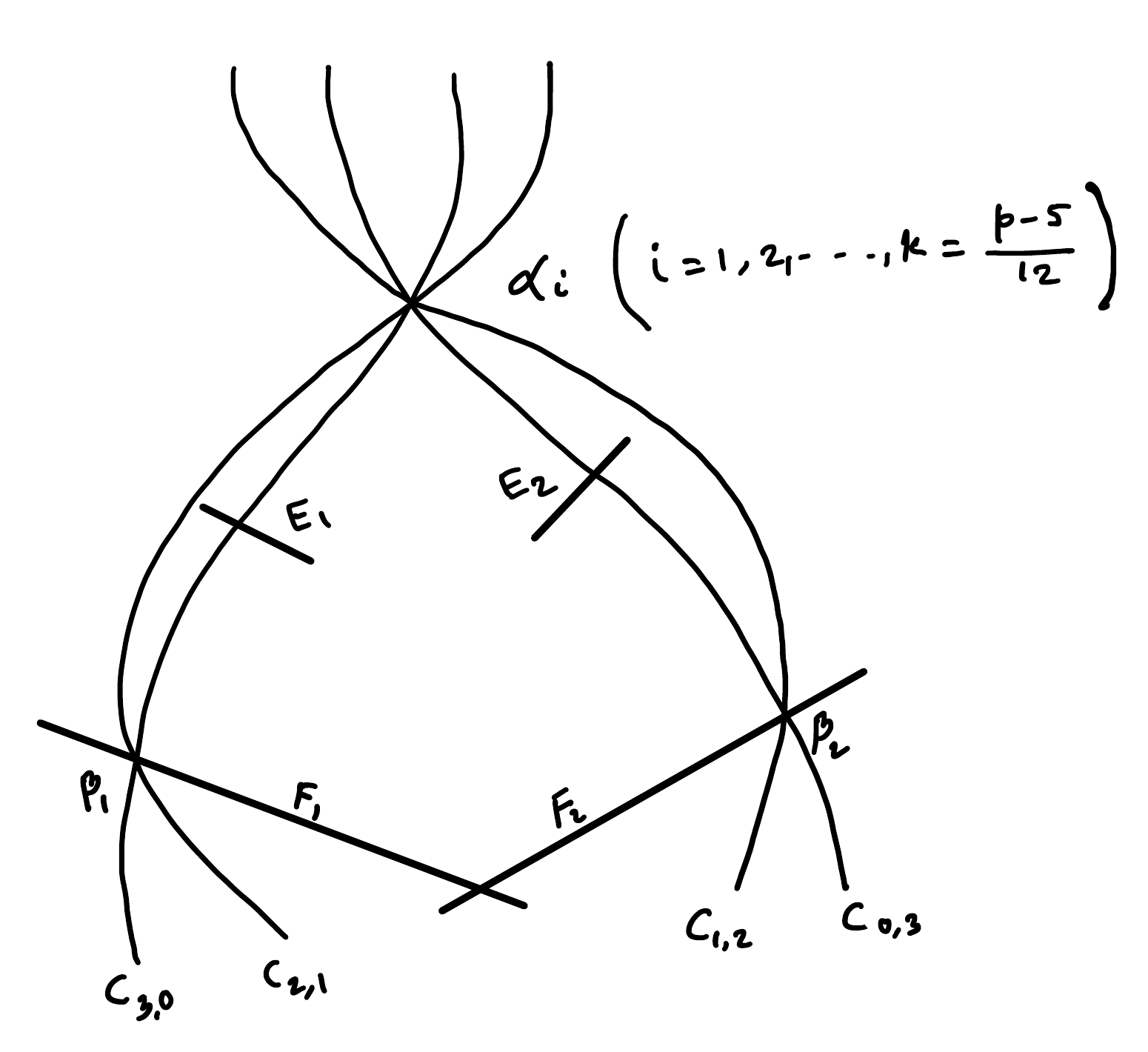}
  \end{center}
  \caption{The special fiber $\cX_0(p^3)_{\ff_p}$ when $p \equiv 5 \pmod{12}$.}
  \label{fig:5mod12}
\end{figure}  

\begin{prop}\label{prop5mod12}
  The local intersection numbers of the vertical components  supported on the special fiber of 
  $\cX_0(p^3)$ for $p \equiv 5 \pmod {12}$ are given by the following matrix: 
  \begin{equation*} \displaystyle
    \renewcommand*{\arraystretch}{2}
   \begin{array}{l|cccccccc}  
      & C_{3,0} & C_{0,3} &  C_{2,1} & \hphantom{0} C_{1,2} \hphantom{0} & \hphantom{0}  E_1 \hphantom{0} & \hphantom{0} E_2
      \hphantom{0} & \hphantom{0} F_1 \hphantom{0} & \hphantom{0} F_2 \\ \hline
     C_{3,0} & -\frac{p^3-p^2+8}{12} & \frac{p-5}{12} & \frac{p^2-p-8}{12} & \frac{p-5}{12} & 0 & 0 & 1 & 0 \\
    C_{0,3} & \frac{p-5}{12} & -\frac{p^3-p^2+8}{12} & \frac{p-5}{12} & \frac{p^2-p-8}{12} & 0 & 0 & 0 & 1 \\
    C_{2,1} & \frac{p^2-p-8}{12} & \frac{p-5}{12} & -\frac{p+7}{6} & \frac{p-5}{12} & 1 & 0 & 1 & 0 \\ 
    C_{1,2} & \frac{p-5}{12} & \frac{p^2-p-8}{12} & \frac{p-5}{12} & -\frac{p+7}{6}  & 0 & 1 & 0 & 1 \\
    E_{1} & 0 & 0 & 1 & 0 & -2 & 0 & 0 & 0 \\
    E_{2} & 0 & 0 & 0 & 1 & 0 & -2 & 0 & 0 \\
    F_{1} & 1 & 0 & 1 & 0 & 0 & 0 & -2 & 1\\
    F_{2} & 0 & 1 & 0 & 1 & 0 & 0 & 1 & -2.
    \end{array}
  \end{equation*}
  In the above matrix,  $E_1, E_2$ correspond to $j=1728$, and $F_1, F_2$ correspond to $j=0$.
\end{prop} 
\begin{proof}
 The self-intersections $E_1^2=-2,\, E_2^2=-2, \, F_1^2=-2$ and $F_2^2=-2$ are calculated by Edixhoven (see \cite[Fig. 1.3.3.1, Fig. 1.3.3.3 and Fig. 1.3.5.3]{MR1056773}). The multiplicities of $E_1$ and $E_2$ are both equal to $\frac{p-1}{2}$. The multiplicities of $F_1$ and $F_2$ are both equal to $p$. Like the previous case, at $\alpha_i$ ($i=1,\ldots, k=\frac{p-5}{12}$) local intersections are given by
  \begin{align*}
      & i_{\alpha_i}\big(C_{3,0},\, C_{2,1}\big)= p, \,\ \,\
      i_{\alpha_i}\big(C_{0,3},\, C_{1,2}\big)= p, \\
      & i_{\alpha_i}\big(C_{3,0},\, C_{1,2}\big)= 1,\,\ \,\ 
      i_{\alpha_i}\big(C_{0,3},\, C_{2,1}\big)= 1.
  \end{align*}
  From \cite[Fig. 1.3.5.3]{MR1056773}, we draw the Figure \ref{fig:5mod12beta}. 
  \begin{figure}[h]
  \begin{center}
    \includegraphics[scale=0.4]{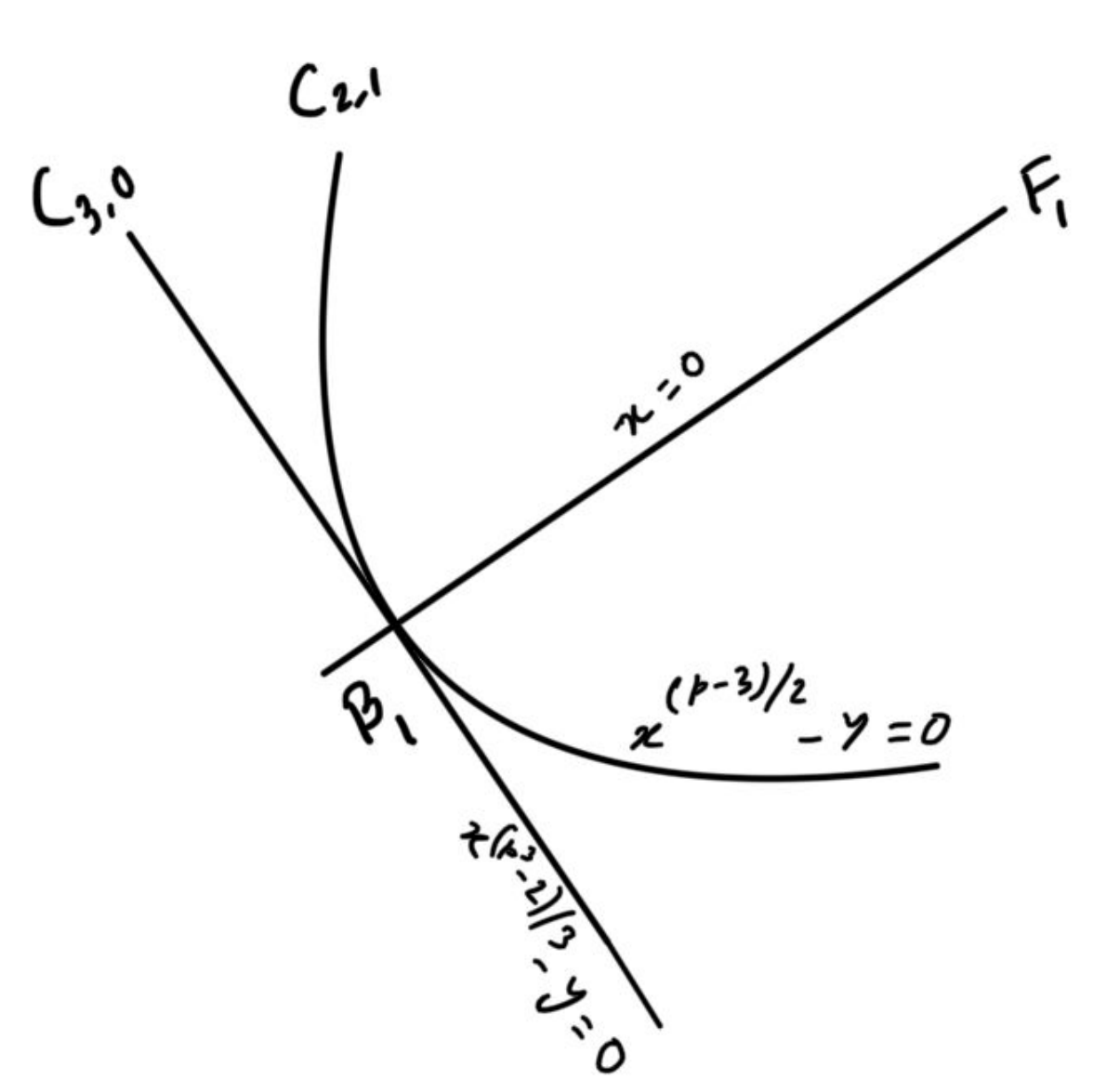}
  \end{center}
  \caption{The intersection point $\beta_1$ when $p \equiv 5 \pmod{12}$.}
  \label{fig:5mod12beta}
\end{figure} 
Local intersections at $\beta_1$ is given by 
\begin{align*}
       & i_{\beta_1}\big(C_{3,0},\, C_{2,1}\big)= \dim_{{\ff}_p} \frac{\ff_p[x,y]_{(x,y)}}{\big( x^{(p^3-2)/3}-y,\,  x^{(p-2)/3}-y\big)}=\frac{p-2}{3},\\
       &i_{\beta_1}\big(C_{3,0},\,F_1\big)= 1, \,\ \,\ i_{\beta_1}\big(C_{2,1},\,F_1\big)= 1.
  \end{align*}
 Then we have
 \begin{align*}
      C_{3,0}\cdot C_{2,1}=& \sum_{i=1}^{k}i_{\alpha_i}\big(C_{3,0},\, C_{2,1}\big)+ i_{\beta_1}\big(C_{3,0},\, C_{2,1}\big)
      =\frac{p(p-5)}{12}+\frac{p-2}{3}=\frac{p^2-p-8}{12}.
 \end{align*}
 Similarly, we have 

  \begin{align*}
      C_{0,3}\cdot C_{1,2}=& \sum_{i=1}^{k}i_{\alpha_i}\big(C_{0,3},\, C_{1,2}\big)+ i_{\beta_2}\big(C_{0,3},\, C_{1,2}\big)
      =\frac{p^2-p-8}{12}.
 \end{align*}
The components $C_{3,0}$ and $C_{1,2}$ intersect only at $\alpha_i$. Similarly, $C_{0,3}$ and $C_{2,1}$ also intersect only at $\alpha_i$ (see Fig. \ref{fig:5mod12}). Then 
 \begin{align*}
      C_{3,0}\cdot C_{1,2}=\sum_{i=1}^{k}i_{\alpha_i}\big(C_{3,0},\, C_{1,2}\big)= \frac{p-5}{12}, \,\ \text{and}\,\ C_{0,3}\cdot C_{2,1}=\sum_{i=1}^{k}i_{\alpha_i}\big(C_{0,3},\, C_{2,1}\big)= \frac{p-5}{12}.
 \end{align*}
 In this case, the vertical divisor corresponding to the special fiber is given by $$V_p=C_{3,0} + C_{0,3} + (p-1)\big( C_{2,1} + C_{1,2}\big)+ \frac{p-1}{2}\big( E_1+E_2\big)  + p\big( F_1+F_2\big).$$
 Now, the remaining calculations are similar as in the previous case.
\end{proof}
\subsection{Case \texorpdfstring{$p \equiv 7 \pmod{12}$}{}}
Following Edixhoven~\cite{MR1056773}, we draw the special fiber $V_p = \cX_0(p^3)_{\ff_p}$ which is described by Figure~\ref{fig:7mod12}, where
each component is a $\pp^1$. In this case $j=1728$ is supersingular and $j=0$ is ordinary.
\begin{figure}[h]
  \begin{center}
    \includegraphics[scale=0.3]{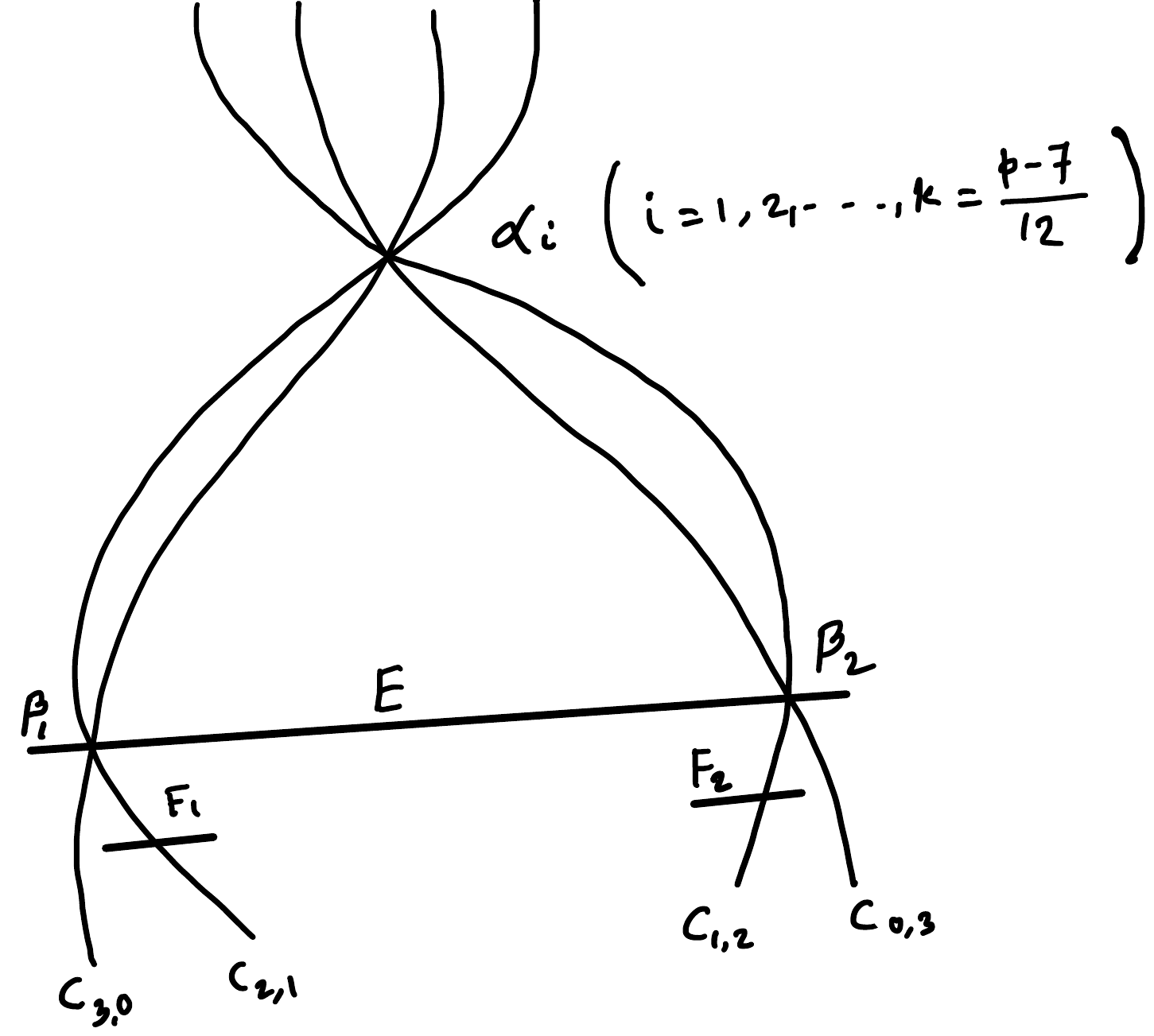}
  \end{center}
  \caption{The special fiber $\cX_0(p^3)_{\ff_p}$ when $p \equiv 7 \pmod{12}$.}
  \label{fig:7mod12}
\end{figure}

\begin{prop} \label{pmod7imp}
  The local intersection numbers of the prime divisors supported on the special fiber
  of $\cX_0(p^3)$ for $p \equiv 7 \pmod {12}$ are given by the following matrix:
  \begin{equation*}
    \renewcommand*{\arraystretch}{2}
    \begin{array}{l|ccccccc}
        & C_{3,0} & C_{0,3} &  C_{2,1} & \hphantom{0} C_{1,2} \hphantom{0} & \hphantom{0}  E \hphantom{0} & \hphantom{0} F_1 \hphantom{0} & \hphantom{0} F_2 \\ \hline
    C_{3,0} & -\frac{p^3-p^2+6}{12} & \frac{p-7}{12} & \frac{p^2-p-6}{12} & \frac{p-7}{12} & 1 & 0 & 0\\
    C_{0,3} & \frac{p-7}{12} & -\frac{p^3-p^2+6}{12} & \frac{p-7}{12} & \frac{p^2-p-6}{12} & 1 & 0 & 0 \\
    C_{2,1} & \frac{p^2-p-6}{12} & \frac{p-7}{12} & -\frac{p+5}{6} & \frac{p-7}{12} & 1 & 1 & 0 \\
    C_{1,2} & \frac{p-7}{12}  & \frac{p^2-p-6}{12} & \frac{p-7}{12}  & -\frac{p+5}{6} & 1 & 0 & 1 \\
    E & 1 & 1 & 1 & 1 & -2 & 0 & 0 \\
    F_1 & 0 & 0 & 1 & 0 & 0 & -3 & 0\\
    F_2 & 0 & 0 & 0 & 1 & 0 & 0 & -3.
    \end{array}
  \end{equation*}
  In the above matrix, $E$ corresponds to $j=1728$, and $F_1, F_2$ correspond to $j=0$.
\end{prop}
\begin{proof}
  The self-intersections $E^2=-2, \, F_1^2=-3$ and $F_2^2=-3$ were calculated by Edixhoven (see \cite[Fig. 1.3.2.3, Fig. 1.3.6.1 and Fig. 1.3.6.3]{MR1056773}). The multiplicity of $E$ is $p$, and the multiplicities of $F_1$ and $F_2$ are equal to $\frac{p-1}{3}$. As before, at $\alpha_i$ ($i=1,\ldots, k=\frac{p-7}{12}$) local intersections are given by
  \begin{align*}
      & i_{\alpha_i}\big(C_{3,0},\, C_{2,1}\big)= p, \,\ \,\
      i_{\alpha_i}\big(C_{0,3},\, C_{1,2}\big)= p, \\
      & i_{\alpha_i}\big(C_{3,0},\, C_{1,2}\big)= 1,\,\ \,\ 
      i_{\alpha_i}\big(C_{0,3},\, C_{2,1}\big)= 1.
  \end{align*}
  From \cite[Fig. 1.3.2.3]{MR1056773}, we draw the Figure \ref{fig:7mod12beta}. 
  \begin{figure}[h]
  \begin{center}
    \includegraphics[scale=0.4]{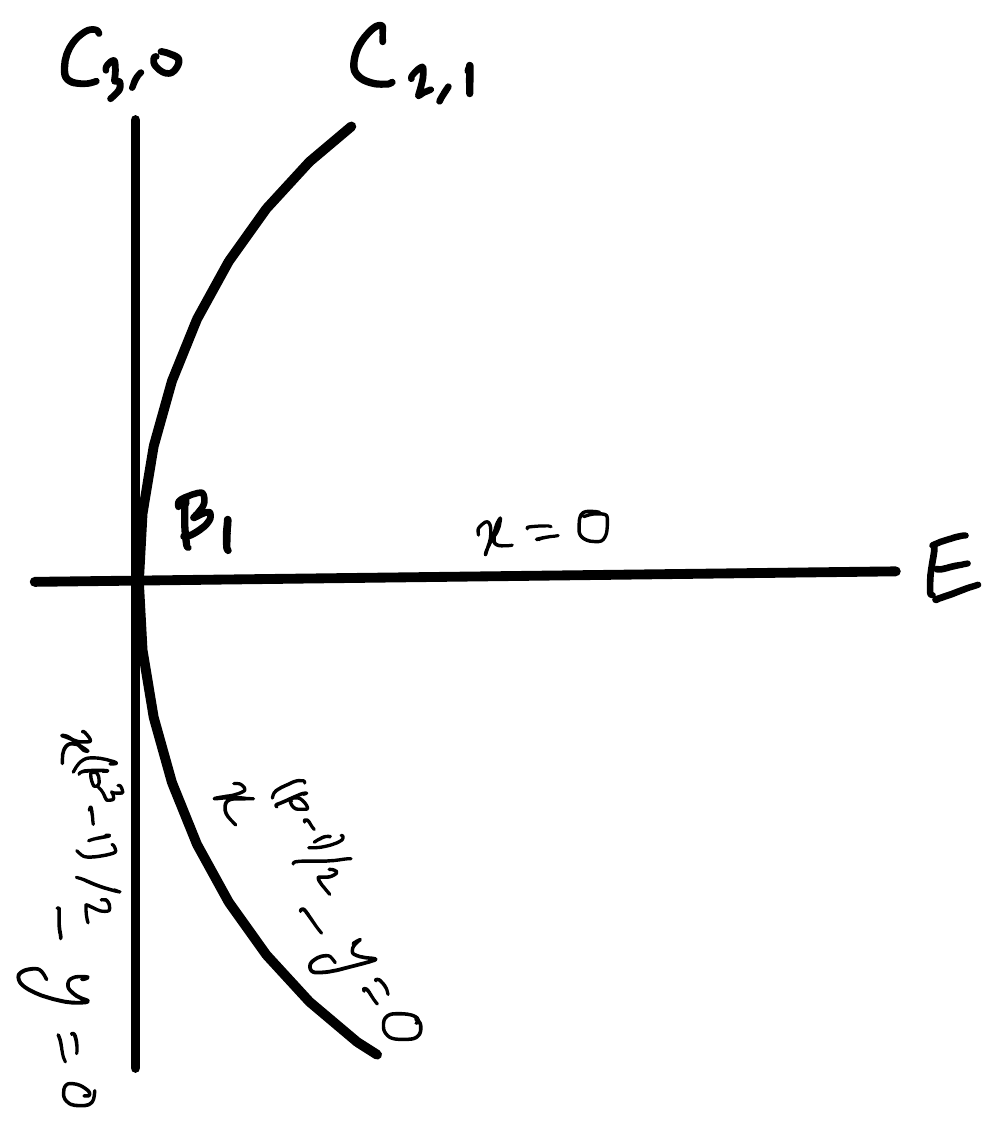}
  \end{center}
  \caption{The intersection point $\beta_1$ when $p \equiv 7 \pmod{12}$.}
  \label{fig:7mod12beta}
\end{figure} 
Local intersections at $\beta_1$ is given by 
\begin{align*}
      & i_{\beta_1}\big(C_{3,0},\,F_1\big)= 1, \,\ \,\ i_{\beta_1}\big(C_{2,1},\,F_1\big)= 1, \\
      & i_{\beta_1}\big(C_{3,0},\, C_{2,1}\big)= \dim_{{\ff_p}} \frac{\ff_p[x,y]_{(x,y)}}{\big( x^{(p^3-1)/2}-y,\,  x^{(p-1)/2}-y\big)}=\frac{p-1}{2}.
  \end{align*}
 Then we have
 \begin{align*}
      C_{3,0}\cdot C_{2,1}=& \sum_{i=1}^{k}i_{\alpha_i}\big(C_{3,0},\, C_{2,1}\big)+ i_{\beta_1}\big(C_{3,0},\, C_{2,1}\big)
      =\frac{p(p-7)}{12}+\frac{p-1}{2}=\frac{p^2-p-6}{12}.
 \end{align*}
 Similarly, we have 
  \begin{align*}
      C_{0,3}\cdot C_{1,2}=& \sum_{i=1}^{k}i_{\alpha_i}\big(C_{0,3},\, C_{1,2}\big)+ i_{\beta_2}\big(C_{0,3},\, C_{1,2}\big)
      =\frac{p^2-p-6}{12}.
 \end{align*}
The components $C_{3,0}$ and $C_{1,2}$ intersect only at $\alpha_i$. Similarly, $C_{0,3}$ and $C_{2,1}$ also intersect only at $\alpha_i$ (see Fig. \ref{fig:7mod12}). Therefore
 \begin{align*}
      C_{3,0}\cdot C_{1,2}=\sum_{i=1}^{k}i_{\alpha_i}\big(C_{3,0},\, C_{1,2}\big)= \frac{p-7}{12}\,\  C_{0,3}\cdot C_{2,1}=\sum_{i=1}^{k}i_{\alpha_i}\big(C_{0,3},\, C_{2,1}\big)= \frac{p-7}{12}.
 \end{align*}
 In this case, the vertical divisor corresponding to the special fiber is given by $$V_p=C_{3,0} + C_{0,3} + (p-1)\big( C_{2,1} + C_{1,2}\big)+ pE  + \frac{p-1}{3}\big( F_1+F_2\big).$$
 The remaining calculations are simple linear algebra as in the previous case.
\end{proof}

\subsection{Case \texorpdfstring{$p \equiv 11 \pmod{12}$}{}}
In this final case, following Edixhoven~\cite{MR1056773}, we draw the the special fiber $V_p = \cX_0(p^3)_{\ff_p}$ which is described by Figure~\ref{fig:11mod12}, where 
each component is a $\pp^1$. In this case both $j=1728$ and $j=0$ are supersingular.

\begin{figure}[h]
  \begin{center}
    \includegraphics[scale=0.29]{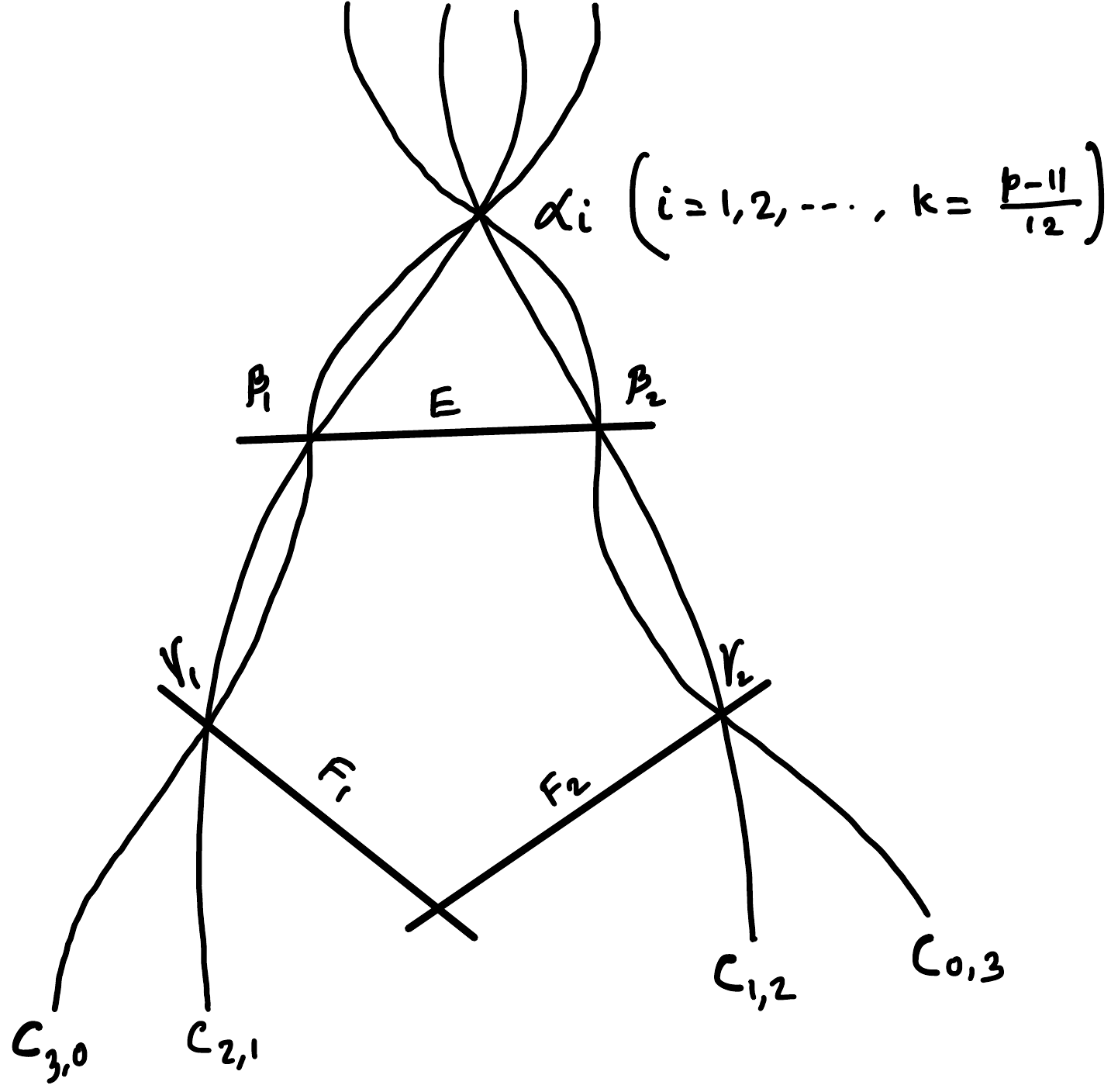}
  \end{center}
  \caption{The special fiber $\cX_0(p^3)_{\ff_p}$ when $p \equiv 11 \pmod{12}$.}
  \label{fig:11mod12}
\end{figure}

\begin{prop} \label{pmod11imp}
  The local intersection numbers of the prime divisors supported on the special fiber
  of $\cX_0(p^3)$ for $p \equiv 11 \pmod {12}$ are given by the following matrix:
  \begin{equation*}
    \renewcommand*{\arraystretch}{2}
    \begin{array}{l|ccccccc}
        & C_{3,0} & C_{0,3} &  C_{2,1} & \hphantom{0} C_{1,2} \hphantom{0} & \hphantom{0}  E \hphantom{0} & \hphantom{0} F_1 \hphantom{0} & \hphantom{0} F_2 \\ \hline
    C_{3,0} & -\frac{p^3-p^2+14}{12} & \frac{p-11}{12} & \frac{p^2-p-14}{12} & \frac{p-11}{12} & 1 & 1 & 0\\
    C_{0,3} & \frac{p-11}{12} & -\frac{p^3-p^2+14}{12} & \frac{p-11}{12} & \frac{p^2-p-14}{12} & 1 & 0 & 1 \\
    C_{2,1} & \frac{p^2-p-14}{12} & \frac{p-11}{12} & -\frac{p+7}{6} & \frac{p-11}{12} & 1 & 1 & 0 \\
    C_{1,2} & \frac{p-11}{12}  & \frac{p^2-p-14}{12} & \frac{p-11}{12}  & -\frac{p+7}{6} & 1 & 0 & 1 \\
    E & 1 & 1 & 1 & 1 & -2 & 0 & 0 \\
    F_1 & 1 & 0 & 1 & 0 & 0 & -2 & 1\\
    F_2 & 0 & 1 & 0 & 1 & 0 & 1 & -2.
    \end{array}
  \end{equation*}
  In the above matrix, $E$ corresponds to $j=1728$, and $F_1, F_2$ correspond to $j=0$.
\end{prop}
\begin{proof}
  The self-intersections $E^2=-2, \, F_1^2=-2$ and $F_2^2=-2$ were calculated by Edixhoven (see \cite[Fig. 1.3.2.3 and Fig. 1.3.5.3]{MR1056773}). The multiplicity of $E$ is $p$, and the multiplicities of $F_1$ and $F_2$ are also equal to $p$. As before, at $\alpha_i$ ($i=1,\ldots, k=\frac{p-11}{12}$) local intersections are given by 
  \begin{align*}
      & i_{\alpha_i}\big(C_{3,0},\, C_{2,1}\big)= p, \,\ \,\
      i_{\alpha_i}\big(C_{0,3},\, C_{1,2}\big)= p, \\
      & i_{\alpha_i}\big(C_{3,0},\, C_{1,2}\big)= 1,\,\ \,\ 
      i_{\alpha_i}\big(C_{0,3},\, C_{2,1}\big)= 1.
  \end{align*}
  From \cite[Fig. 1.3.2.3 and Fig. 1.3.5.3]{MR1056773}, we have Figure \ref{fig:11mod12beta}. 
  \begin{figure}[h]
  \begin{center}
    \includegraphics[scale=0.3]{case3beta1.pdf} \hspace{2cm}\includegraphics[scale=0.3]{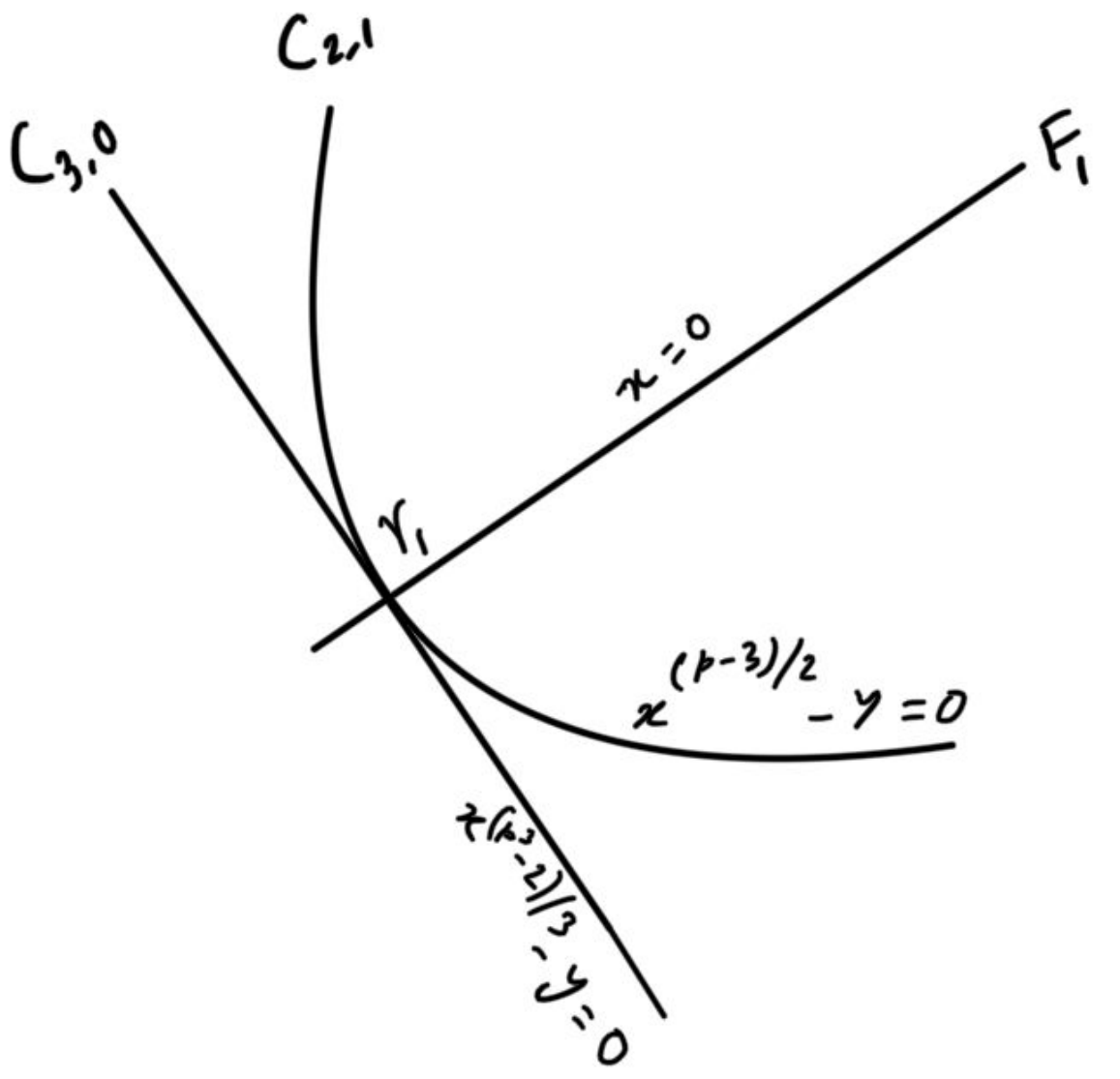}
  \end{center}
  \caption{The intersection point $\beta_1$, and $\gamma_1$ when $p \equiv 11 \pmod{12}$.}\label{fig:11mod12beta}\hspace{2cm} 
  \end{figure} 
Local intersections at $\beta_1$ and $\gamma_1$ are given by 
\begin{align*}
      & i_{\beta_1}\big(C_{3,0},\, C_{2,1}\big)=\frac{p-1}{2}, \,\ \,\ i_{\gamma_1}\big(C_{3,0},\, C_{2,1}\big)=\frac{p-2}{3}.
  \end{align*}
Then we have
\begin{align*}
      C_{3,0}\cdot C_{2,1}=& \sum_{i=1}^{k}i_{\alpha_i}\big(C_{3,0},\, C_{2,1}\big)+ i_{\beta_1}\big(C_{3,0},\, C_{2,1}\big)+i_{\gamma_1}\big(C_{3,0},\, C_{2,1}\big)\\
      =&\frac{p(p-7)}{12}+\frac{p-1}{2}+\frac{p-2}{3}=\frac{p^2-p-14}{12}.
 \end{align*}
 Similarly, we have
  \begin{align*}
      C_{0,3}\cdot C_{1,2}=& \sum_{i=1}^{k}i_{\alpha_i}\big(C_{0,3},\, C_{1,2}\big)+ i_{\beta_2}\big(C_{0,3},\, C_{1,2}\big)+i_{\gamma_2}\big(C_{0,3},\, C_{1,2}\big)
      =\frac{p^2-p-14}{12}.
 \end{align*}
The components $C_{3,0}$ and $C_{1,2}$ intersect only at $\alpha_i$. The components $C_{0,3}$ and $C_{2,1}$ also intersect only at $\alpha_i$ (see Fig. \ref{fig:11mod12}). Hence 
 \begin{align*}
      C_{3,0}\cdot C_{1,2}=\sum_{i=1}^{k}i_{\alpha_i}\big(C_{3,0},\, C_{1,2}\big)= \frac{p-11}{12}\,\ \,\ \text{and}\,\ \,\ C_{0,3}\cdot C_{2,1}=\sum_{i=1}^{k}i_{\alpha_i}\big(C_{0,3},\, C_{2,1}\big)= \frac{p-11}{12}.
 \end{align*}
 In this case the vertical divisor corresponding to the special fiber is given by $$V_p=C_{3,0} + C_{0,3} + (p-1)\big( C_{2,1} + C_{1,2}\big)+ pE  + p\big( F_1+F_2\big).$$
 The remaining calculations are simple linear algebra.
\end{proof}
\subsection{Intersection matrices for minimal regular model for \texorpdfstring{$X_0(p^4)$}{}}\label{incidince_matrric_power4}
Let $\tcX_0(p^4)$ be the regular model constructed by Edixhoven \cite{MR1056773}. The special fiber of the regular model always has components $C_{4,0},\, C_{3,1},\,C_{2,2},\, C_{1,3}$, $C_{0,4}$, along with other components which depend on the residue of $p$ modulo $12$. The multiplicity of the component $C_{a,b}$ ($C_{a,b} \in\lbrace C_{4,0}, C_{0,4}, C_{3,1}, C_{1,3}, C_{2,2}\rbrace$) is $\phi\big(p^{\min(a,b)}\big)$. The local intersection of two vertical components at a supersingular point $\alpha_i$ (shown in figures \ref{fig:1mod12power4}, \ref{fig:5mod12power4}, \ref{fig:7mod12power4}, and \ref{fig:11mod12power4}) is given by the formula \ref{localintersection}. 

However $\tcX_0(p^4)$ is not a minimal regular model. In this section, we recall the regular model of Edixhoven and describe the minimal regular models obtained from them after contracting $-1$-curves. The minimal regular model $\cX_0(p^4)$ is obtained from $\tcX_0(p^4)$ by successive blow downs (contractions)
of curves in the special fiber $\tcX_0(p^4)_{\ff_p}$ and we shall denote by $\pi: \tcX_0(p^4) \to \cX_0(p^4)$ the morphism from Edixhoven's model. 
In the computations, we shall use \cite[Chapter 9, Theorem~2.12]{MR1917232} repeatedly. 

\subsection{Case \texorpdfstring{$p \equiv 1 \pmod{12}$}{}.}
Following Edixhoven~\cite{MR1056773}, we draw the special fiber $V_p = \tcX_0(p^4)_{\ff_p}$ in Figure~\ref{fig:1mod12power4}, 
where each component is a $\pp^1$. In this case both $j=1728$ and $j=0$ are ordinary.

\begin{figure}
  \begin{center}
    \includegraphics[scale=0.3]{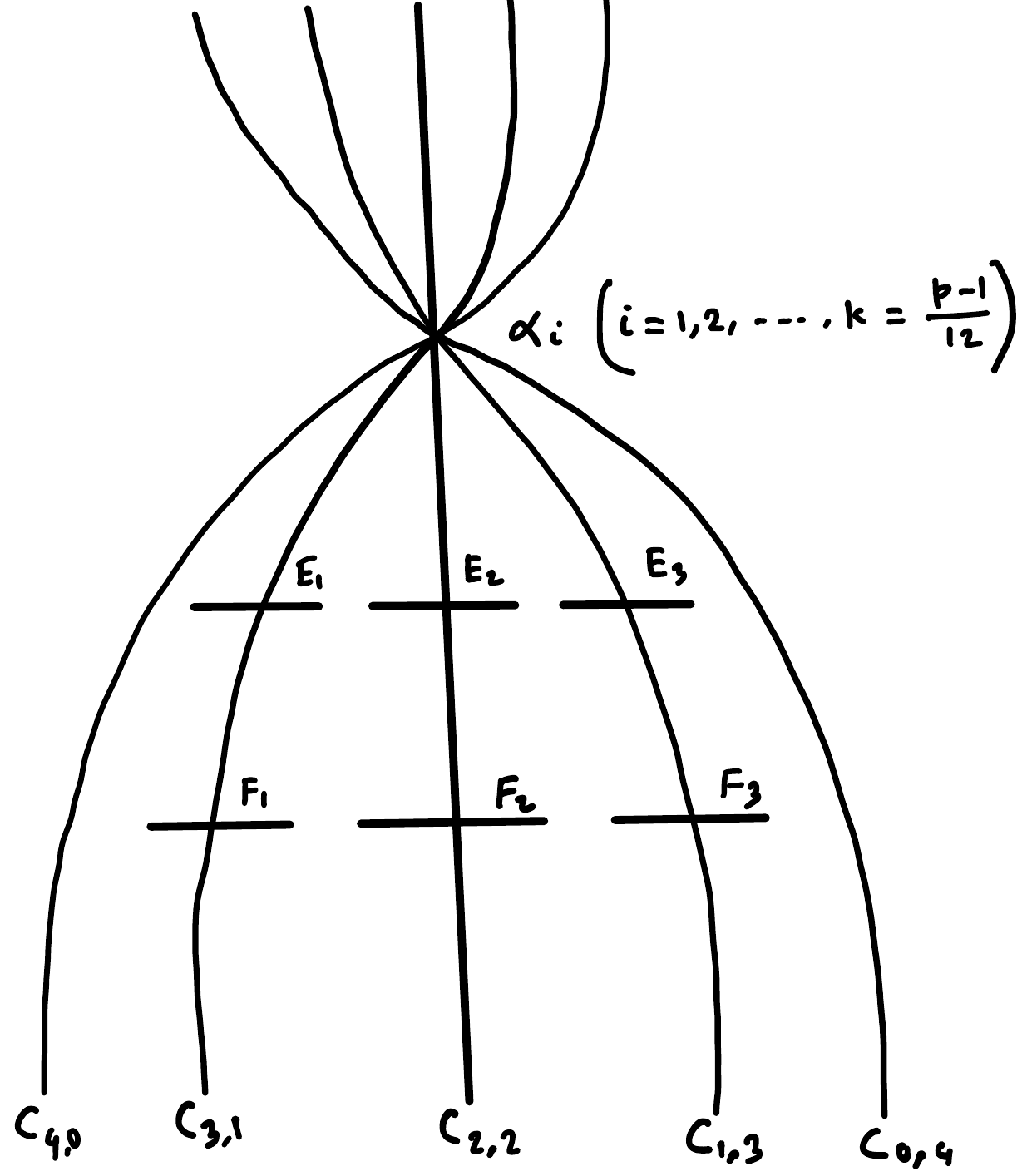}
  \end{center}      
  \caption{The special fiber $\tcX_0(p^4)_{\ff_p}$ when $p \equiv 1 \pmod{12}$.} 
  \label{fig:1mod12power4} 
\end{figure}
\begin{prop} \label{prop1mod12power4}
  The local intersection numbers of the vertical components  supported on the special fiber of $\tcX_0(p^4)$ for $p \equiv 1 \pmod {12}$ are given by the following matrix:
  \begin{equation*}
    \renewcommand*{\arraystretch}{2}
    \begin{array}{l|ccccccccccc}  
      & C_{4,0} & C_{0,4} &  C_{3,1} & \hphantom{0} C_{1,3} \hphantom{0} & \hphantom{0} C_{2,2} \hphantom{0} & \hphantom{0}  E_1 \hphantom{0} & \hphantom{0} E_2
      \hphantom{0} & \hphantom{0}  E_3 \hphantom{0} & \hphantom{0} F_1 \hphantom{0} & \hphantom{0} F_2 & \hphantom{0}  F_3 \hphantom{0} \\ \hline
      C_{4,0} & -\frac{p^3(p-1)}{12} & \frac{p-1}{12} & \frac{p^2(p-1)}{12} & \frac{p-1}{12} & \frac{p-1}{12} & 0 & 0 & 0 & 0 & 0 & 0 \\
      C_{0,4} & \frac{p-1}{12} & -\frac{p^3(p-1)}{12} & \frac{p-1}{12} & \frac{p^2(p-1)}{12} & \frac{p-1}{12} & 0 & 0 & 0 & 0 & 0 & 0 \\
      C_{3,1} & \frac{p^2(p-1)}{12} & \frac{p-1}{12} & -\frac{p^2+5}{6} & \frac{p-1}{12} & \frac{p-1}{12} & 1 & 0 & 0 & 1 & 0 & 0 \\ 
       C_{1,3} & \frac{p-1}{12} & \frac{p^2(p-1)}{12} & \frac{p-1}{12}  & -\frac{p^2+5}{6} & \frac{p-1}{12} & 0 & 0 & 1 & 0 & 0 & 1 \\ C_{2,2} & \frac{p-1}{12} & \frac{p-1}{12} & \frac{p-1}{12}   & \frac{p-1}{12} & -1 & 0 & 1 & 0 & 0 & 1 & 0 \\
      E_1 & 0 & 0 & 1 & 0 & 0 & -2 & 0 & 0 & 0 & 0 & 0 \\
      E_2 & 0 & 0 & 0  & 0 & 1 & 0 & -2 & 0 & 0 & 0 & 0\\
      E_3 & 0 & 0 & 0  & 1 & 0 & 0 & 0 & -2 & 0 & 0 & 0\\
      F_1 & 0 & 0 & 1 & 0 & 0 & 0 & 0 & 0  & -3 & 0 & 0\\
      F_2 & 0 & 0 & 0 & 0 & 1 & 0 & 0 & 0 & 0 & -3 & 0\\
      F_3 & 0 & 0 & 0 & 1 & 0 & 0 & 0 & 0 & 0 & 0 & -3.
    \end{array}
  \end{equation*}
 In the above matrix, $E_1, E_2, E_3$ correspond to $j=1728$, and $F_1, F_2, F_3$ correspond to $j=0$.
\end{prop}
\begin{proof}
At the point $\alpha_i$ ($i=1,\ldots, k=\frac{p-1}{12}$) local intersections are given by
  \begin{align*}
      & i_{\alpha_i}\big(C_{4,0},\, C_{3,1}\big)= p^2, \,\ \,\
      i_{\alpha_i}\big(C_{0,4},\, C_{1,3}\big)= p^2, \\
      & i_{\alpha_i}\big(C_{4,0},\, C_{1,3}\big)= 1,\,\ \,\ 
      i_{\alpha_i}\big(C_{0,4},\, C_{3,1}\big)= 1.
  \end{align*}
   Here the vertical divisor corresponding to the special fiber is given by
   $$V_p= C_{4,0}+C_{0,4}+(p-1)\big( C_{3,1} + C_{1,3}\big)+p(p-1)C_{2,2}+ \frac{p-1}{2}\big( E_1 + pE_2+E_3\big)+\frac{p-1}{3}\big( F_1 + p F_2+F_3\big).$$ Since the proof is similar as Proposition \ref{pmod1imp}, we omit the proof here.
\end{proof}
\begin{prop}\label{minmatrixcase1power4}
For $p \equiv 1 \pmod {12}$ the minimal regular model $\cX_0(p^4)$ is obtained from $\tcX_0(p^4)$ by blowing down the prime vertical divisors $C_{2,2}$, $E_2$ and $F_2$ supported on the special fiber. The local intersection numbers of the components  of the special fiber of $\cX_0(p^4)$ are given by the following matrix:
  \begin{equation*}
    \renewcommand*{\arraystretch}{2}
    \begin{array}{l|cccccccc}  
      & C_{4,0}' & C_{0,4}' &  C_{3,1}' & \hphantom{0} C_{1,3}' & \hphantom{0}  E_1'  & \hphantom{0} E_3'
       & \hphantom{0} F_1'  & \hphantom{0} F_3' \\ \hline
      C_{4,0}' & {-\frac{2p^4 - 2p^3 - p^2 + 2p - 1}{24}} & \frac{p^2 - 1}{24} & \frac{2p^3 - p^2 - 2p + 1}{24} & \frac{p^2 - 1}{24} & 0 & 0 & 0 & 0 \\
      C_{0,4}' & \frac{p^2 - 1}{24} & -\frac{ 2p^4 - 2p^3 - p^2 + 2p - 1}{24} & \frac{p^2-1}{24} & \frac{ 2p^3 - p^2 - 2p + 1}{24} & 0 & 0 & 0 & 0 \\
      C_{3,1}' & \frac{2p^3 - p^2 - 2p + 1}{24} & \frac{p^2-1}{24} & -\frac{3p^2 + 2p + 19}{24} & \frac{p^2-1}{24} & 1 & 0 & 1 & 0 \\ 
       C_{1,3}' & \frac{p^2-1}{24} & \frac{ 2p^3 - p^2 - 2p + 1}{24} & \frac{p^2-1}{24} & -\frac{3p^2 + 2p + 19}{24} & 0 & 1 & 0 & 1 \\ 
      E_1' & 0 & 0 & 1 & 0 & -2 & 0 & 0 & 0 \\
      E_3' & 0 & 0 & 0 & 1 & 0 & -2 & 0 & 0 \\
      F_1' & 0 & 0 & 1 & 0 & 0 & 0 & -3 & 0 \\
      F_3' & 0 & 0 & 0 & 1 & 0 & 0 & 0 & -3.
    \end{array}
  \end{equation*}
In the above matrix, $C_{4,0}'$, $C_{0,4}'$, $C_{3,1}'$, $C_{1,3}'$, $E_1'$, $E_3'$, $F_1'$ and $F_3'$ denote the images of $C_{4,0}$, $C_{0,4}$, $C_{3,1}$, $C_{1,3}$, $E_1$, $E_3$, $F_1$ and $F_3$, respectively under the blow down morphism $\tcX_0(p^4) \to \cX_0(p^4)$.
\end{prop}

\begin{proof}
From Proposition \ref{prop1mod12power4}, note that if $p \equiv 1 \pmod{12}$ then the component $C_{2,2}$ is rational and has self-intersection $-1$. By Castelnuovo's criterion \cite[Chapter~9, Theorem~3.8]{MR1917232} we can thus 
 blow down $C_{2,2}$ without introducing a singularity. Let $\cX_0(p^4)'$ be the corresponding arithmetic surface 
  and $\pi_1: \tcX_0(p^4) \to \cX_0(p^4)'$, be the blow down morphism.

  Then $E_2'  = \pi_1(E_2)$. From Liu \cite[Chapter 9, Proposition 2.23]{MR1917232}, we have
      $\pi_1^* E_2' = E_2 + \mu C_{2,2}$.
  Using  \cite[Chapter 9, Theorem~2.12]{MR1917232}, we obtain $0 = C_{2,2}\cdot \pi_1^* E_2' = 1  - \mu$. This implies
  \begin{align*}
       \pi_1^* E_2' = E_2 + C_{2,2}.
  \end{align*}
 Then we deduce that  $(E_2')^2 = (\pi_1^* E_2')^2 = \left( E_2 + C_{2,2}\right)\cdot\left( E_2 +C_{2,2} \right) = -1$. Hence $E_2'$  
  is a rational curve in the special fiber of $\cX_0(p^4)'$ with self intersection $-1$. It can thus be blown 
  down again, and the resulting scheme is again regular. Let $\cX_0(p^4)''$ be the blow down and $\pi_2: \tcX_0(p^4) 
  \to \cX_0(p^4)''$ the corresponding morphism. 
  
  Let $F_2' = \pi_2(F_2)$, and if $\pi_2^*F_2' = F_2 + \mu C_{2,2} + \nu E_2$ for $\mu, \nu \in \zz$ then using the 
  fact that $C_{2,2} \cdot \pi_2^*F_2' = E_2 \cdot \pi_2^*F_2' = 0$ we find $\mu = 2$ and $\nu =1$. This 
  yields 
  \begin{align*}
      \pi_2^*F_2' = F_2 + 2 C_{2,2} + E_2.
  \end{align*}
  This implies $(F_2')^2 = -1$.  We can thus blow down $F_2'$  further to 
  arrive finally at an arithmetic surface $\cX_0(p^4)$, which is the minimal regular model of $X_0(p^4)$ since 
  no further blow down is possible. Let $\pi: \tcX_0(p^4) \to \cX_0(p^4)$ be the morphism obtained by composing
  the sequence of blow downs.
  
  The special fiber of $\cX_0(p^4)$ consists of $C_{4,0}'$, $C_{0,4}'$, $C_{3,1}'$, $C_{1,3}'$, $E_1'$, $E_3'$, $F_1'$ and $F_3'$ that are the images 
  of $C_{4,0}$, $C_{0,4}$, $C_{3,1}$, $C_{1,3}$, $E_1$, $E_3$, $F_1$ and $F_3$, respectively under $\pi$. Let $\pi^* C_{a,b}' = C_{a,b} + \alpha C_{2,2} + \beta E_2 + \gamma F_2$, where $C_{a,b} \in\lbrace C_{4,0}, C_{0,4}, C_{3,1}, C_{1,3} \rbrace$. Since the intersection of $\pi^* C_{a,b}'$ with $C_{2,2}, E_2$ and $F_2$ are zero, then by solving 
  \begin{equation*}
    \begin{bmatrix}
-1 & 1 & 1\\
1 & -2 & 0\\
1 & 0 & -3
\end{bmatrix}
\begin{bmatrix}
\alpha \\
\beta\\
\gamma
\end{bmatrix}=\begin{bmatrix}
-\frac{p-1}{12}\\
0\\
0
\end{bmatrix}
  \end{equation*}
 we get
  \begin{align*}
    &\pi^* C_{a,b}' = C_{a,b} + \frac{p-1}{2} C_{2,2} + \frac{p-1}{4} E_2 + \frac{p-1}{6} F_2,
    \end{align*}
    where $C_{a,b} \in\lbrace C_{4,0}, C_{0,4}, C_{3,1}, C_{1,3} \rbrace$.
    Also, note that
    \begin{align*}
    &\pi^*E_1'=E_1, \,\ \pi^*E_3'=E_3, \,\
    \pi^*F_1'=F_1,\,\ \pi^*F_3'=F_3.
  \end{align*} 
  Finally, using \cite[Chapter 9, Theorem~2.12 (c)]{MR1917232} we get our required matrix.
  For example $C_{4,0}'\cdot C_{3,1}' = \pi^* C_{4,0}' \cdot \pi^* C_{3,1}'$ and the right
  hand side can be calculated using Proposition \ref{prop1mod12power4}.
\end{proof}

\subsection{Case \texorpdfstring{$p \equiv 5 \pmod{12}$}{}.}
From Edixhoven~\cite{MR1056773}, we draw the special fiber $V_p = \tcX_0(p^4)_{\ff_p}$ which is described by 
Figure~\ref{fig:5mod12}, where each component is a $\pp^1$. In this case $j=1728$ is ordinary and $j=0$ is supersingular.

\begin{figure}[h]
  \begin{center}
    \includegraphics[scale=0.3]{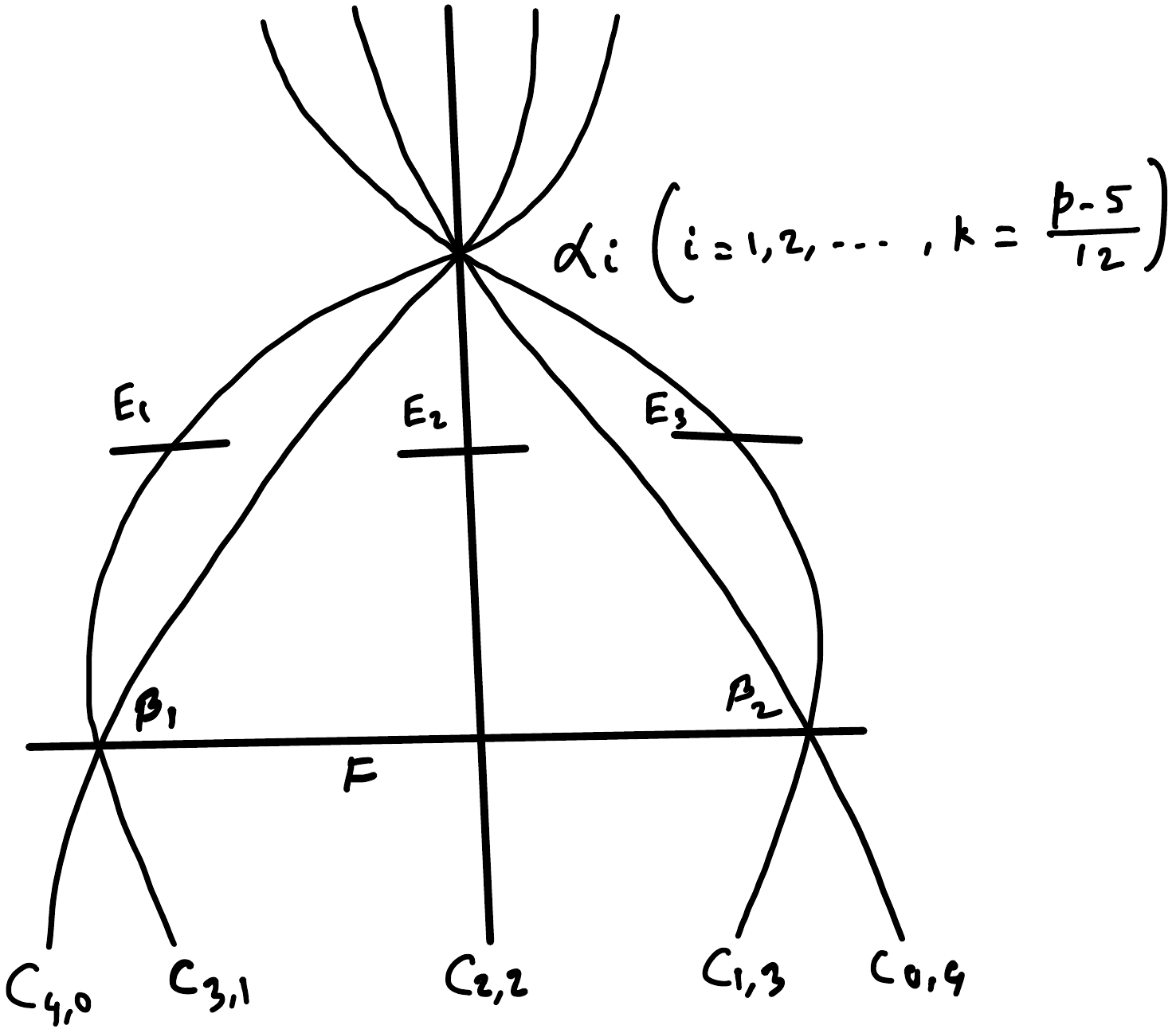}
  \end{center}
  \caption{The special fiber $\tcX_0(p^4)_{\ff_p}$ when $p \equiv 5 \pmod{12}$.}
  \label{fig:5mod12power4}
\end{figure} 
\begin{prop} 
  The local intersection numbers of the vertical components  supported on the special fiber of 
  $\tcX_0(p^4)$ for $p \equiv 5 \pmod {12}$ are given by the following matrix: 
  \begin{equation*} \displaystyle
    \renewcommand*{\arraystretch}{2}
   \begin{array}{l|ccccccccc}  
      & C_{4,0} & C_{0,4} &  C_{3,1} & \hphantom{0} C_{1,3} \hphantom{0} & \hphantom{0} C_{2,2} \hphantom{0} & \hphantom{0}  E_1 \hphantom{0} & \hphantom{0} E_2
      \hphantom{0} & \hphantom{0}  E_3 \hphantom{0} & \hphantom{0} F \hphantom{0}  \\ \hline
     C_{4,0} & -\frac{p^4-p^3+4}{12} & \frac{p-5}{12} & \frac{p^3-p^2-4}{12} & \frac{p-5}{12} & \frac{p-5}{12} & 0 & 0 & 0 & 1 \\
    C_{0,4} & \frac{p-5}{12} & -\frac{p^4-p^3+4}{12} & \frac{p-5}{12} & \frac{p^3-p^2-4}{12} & \frac{p-5}{12} & 0 & 0 & 0 & 1 \\
    C_{3,1} & \frac{p^3-p^2-4}{12} & \frac{p-5}{12} & -\frac{p^2+5}{12} & \frac{p-5}{12} & \frac{p-5}{12} & 1 & 0 & 0 & 1 \\ 
    C_{1,3} & \frac{p-5}{12} & \frac{p^3-p^2-4}{12} & \frac{p-5}{12} & -\frac{p^2+5}{12} & \frac{p-5}{12} & 0 & 0 & 1 & 1 \\
    C_{2,2} & \frac{p-5}{12} & \frac{p-5}{12} & \frac{p-5}{12} & \frac{p-5}{12} & -1 & 0 & 1 & 0 & 1 \\
    E_{1} & 0 & 0 & 1 & 0 & 0 & -2 & 0 & 0 & 0 \\
    E_{2} & 0 & 0 & 0 & 0 & 1 & 0 & -2 & 0 & 0 \\
    E_{3} & 0 & 0 & 0 & 1 & 0 & 0 & 0 & -2 & 0\\
    F & 1 & 1 & 1 & 1 & 1 & 0 & 0 & 0 & -3.
    \end{array}
  \end{equation*}
  In the above matrix,  components $E_1, E_2, E_3$ correspond to $j=1728$, and $F$ corresponds to $j=0$.
\end{prop} 
\begin{proof}
   The principal divisor is given by
   $$V_p= C_{4,0}+C_{0,4}+(p-1)\big( C_{3,1} + C_{1,3}\big)+p(p-1)C_{2,2}+ \frac{p-1}{2}\big( E_1 + pE_2+E_3\big)+\frac{p(p+1)}{3}F.$$ 
   At the point $\alpha_i$ ($i=1,\ldots, k=\frac{p-5}{12}$) local intersections are same as before. From \cite[Fig. 1.3.4.1]{MR1056773}, we have the Figure \ref{fig:5mod12betapower4}. 
   \begin{figure}[h]
  \begin{center}
    \includegraphics[scale=0.4]{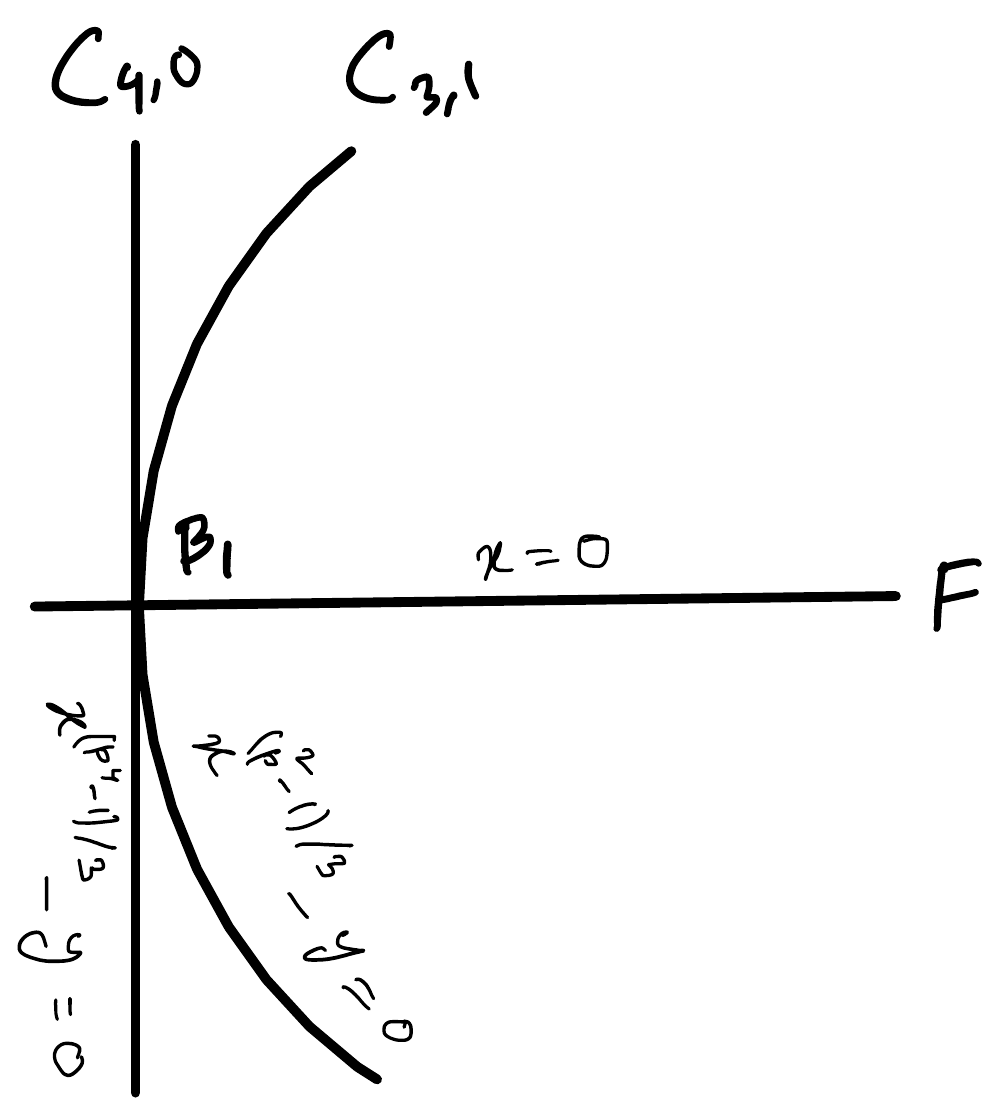}
  \end{center}
  \caption{The intersection point $\beta_1$ when $p \equiv 5 \pmod{12}$.}
  \label{fig:5mod12betapower4}
\end{figure} 
   The local intersections at the point $\beta_1$ is given by
   \begin{align*}
       & i_{\beta_1}\big(C_{4,0},\, C_{3,1}\big)= \dim_{{\ff_p}} \frac{\ff_p[x,y]_{(x,y)}}{\big( x^{(p^2-1)/3}-y,\,  x^{(p^4-1)/3}-y\big)}=\frac{p^2-1}{3},\\
       &i_{\beta_1}\big(C_{4,0},\,F\big)= 1, \,\ \,\ i_{\beta_1}\big(C_{3,1},\,F\big)= 1.
  \end{align*}
  Using these local intersection numbers we get our required matrix.
\end{proof}
\begin{prop}\label{minmatrixcase2power4}
For $p \equiv 5 \pmod {12}$ the minimal regular model $\cX_0(p^4)$ is obtained from
$\tcX_0(p^4)$ by blowing down $C_{2,2}$, $E_2$ and $F$.
The local intersection numbers of the components of the special fiber of $\cX_0(p^4)$ are given by the following matrix.
 \begin{equation*}
    \renewcommand*{\arraystretch}{2}
    \begin{array}{l|cccccc}  
      & C_{4,0}' & C_{0,4}' &  C_{3,1}' & \hphantom{0} C_{1,3}' \hphantom{0} & \hphantom{0}  E_1'  & \hphantom{0} E_3'
        \\ \hline
      C_{4,0}' & -\frac{ 2p^4 - 2p^3 - p^2 + 2p - 1}{24} & \frac{ p^2 - 1}{24} & \frac{ 2p^3 - p^2 - 2p + 1 }{24} & \frac{p^2 - 1}{24} & 0 & 0 \\
      C_{0,4}' & \frac{p^2 - 1}{24} & -\frac{ 2p^4 - 2p^3 - p^2 + 2p - 1}{24} & \frac{p^2-1}{24} & \frac{ 2p^3 - p^2 - 2p + 1}{24} & 0 & 0  \\
      C_{3,1}' & \frac{2p^3 - p^2 - 2p + 1}{24} & \frac{p^2-1}{24} & -\frac{3p^2 + 2p + 11}{24} & \frac{p^2-1}{24} & 1 & 0  \\ 
       C_{1,3}' & \frac{p^2-1}{24} & \frac{ 2p^3 - p^2 - 2p + 1}{24} & \frac{p^2-1}{24} & -\frac{3p^2 + 2p + 11}{24} & 0 & 1  \\ 
      E_1' & 0 & 0 & 1 & 0 & -2 & 0 \\
      E_3' & 0 & 0 & 0 & 1 & 0 & -2. 
    \end{array}
  \end{equation*}
In the above matrix, $C_{4,0}'$, $C_{0,4}'$, $C_{3,1}'$, $C_{1,3}'$, $E_1'$ and $E_3'$ denote the images of $C_{4,0}$, $C_{0,4}$, $C_{3,1}$,
 $C_{1,3}$, $E_1$ and $E_3$,
  respectively under the blow down
  morphism $\tcX_0(p^4) \to \cX_0(p^4)$.
\end{prop}

\begin{proof}
  As before the minimal regular model $\cX_0(p^4)$ is obtained by blowing down $C_{2,2}$, then the image of $E_2$ and then the image of $F$. Let $\pi: \tcX_0(p^4) \to \cX_0(p^4)$ be the morphism obtained by composing
  the sequence of blow downs.
  
  The special fiber of $\cX_0(p^4)$ consists of $C_{4,0}'$, $C_{0,4}'$, $C_{3,1}'$, $C_{1,3}'$, $E_1'$ and $E_3'$ that are the images 
  of $C_{4,0}$, $C_{0,4}$, $C_{3,1}$, $C_{1,3}$, $E_1$ and $E_3$, respectively, under $\pi$. Let $\pi^* C_{a,b}' = C_{a,b} + \alpha C_{2,2} + \beta E_2 + \gamma F$, where $C_{a,b} \in\lbrace C_{4,0}, C_{0,4}, C_{3,1}, C_{1,3} \rbrace$. 
  Since the intersection of $\pi^* C_{a,b}'$ with $C_{2,2}, E_2$ and $F$ are zero, then by solving 
  \begin{align*}
       \begin{bmatrix}
-1 & 1 & 1\\
1 & -2 & 0\\
1 & 0 & -3
\end{bmatrix}\begin{bmatrix}
\alpha \\
\beta\\
\gamma
\end{bmatrix}=\begin{bmatrix}
-\frac{p-5}{12}\\
0\\
-1
\end{bmatrix}
  \end{align*}
 we get
  \begin{align*}
    &\pi^* C_{a,b}' = C_{a,b} + \frac{p-1}{2} C_{2,2} + \frac{p-1}{4} E_2+ \frac{p+1}{6} F,
    \end{align*}
    where $C_{a,b} \in\lbrace C_{4,0}, C_{0,4}, C_{3,1}, C_{1,3} \rbrace$.
    Also, note that
    \begin{align*}
    &\pi^*E_1'=E_1, \,\ \pi^*E_3'=E_3.
  \end{align*} 
  Finally, using \cite[Chapter 9, Theorem~2.12 (c)]{MR1917232} we get our required matrix.
\end{proof}
\subsection{Case \texorpdfstring{$p \equiv 7 \pmod{12}$}{}.}
From Edixhoven~\cite{MR1056773}, we draw the special fiber $V_p = \tcX_0(p^4)_{\ff_p}$ which is described by Figure~\ref{fig:7mod12power4}, where each component is a $\pp^1$. In this case $j=1728$ is supersingular and $j=0$ is ordinary.
\begin{figure}[h]
  \begin{center}
    \includegraphics[scale=0.4]{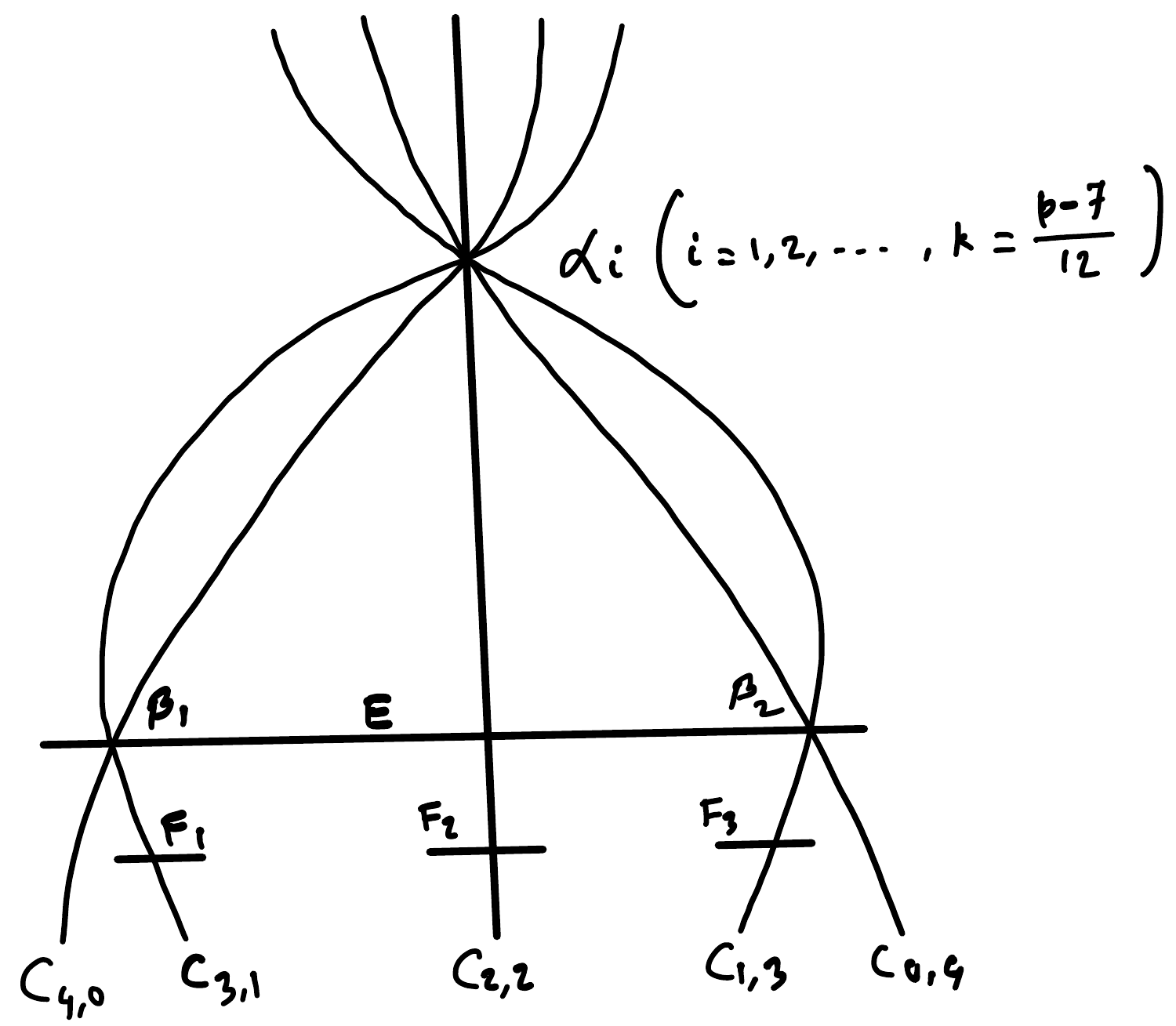}
  \end{center}
  \caption{The special fiber $\tcX_0(p^4)_{\ff_p}$ when $p \equiv 7 \pmod{12}$.}
  \label{fig:7mod12power4}
\end{figure}

\begin{prop} \label{prop7mod12power4}
  The local intersection numbers of the prime divisors supported on the special fiber
  of $\tcX_0(p^4)$ for $p \equiv 7 \pmod {12}$ are given by the following matrix.
  \begin{equation*}
    \renewcommand*{\arraystretch}{2}
    \begin{array}{l|ccccccccc}  
      & C_{4,0} & C_{0,4} &  C_{3,1} & \hphantom{0} C_{1,3} \hphantom{0} & \hphantom{0} C_{2,2} \hphantom{0} & \hphantom{0}  E \hphantom{0} & \hphantom{0} F_1
      \hphantom{0} & \hphantom{0}  F_2 \hphantom{0} & \hphantom{0} F_3 \hphantom{0}  \\ \hline
    C_{4,0} & -\frac{p^4-p^3+6}{12} & \frac{p-7}{12} & \frac{p^3-p^2-6}{12} & \frac{p-7}{12} & \frac{p-7}{12} & 1 & 0 & 0 & 0\\
    C_{0,4} & \frac{p-7}{12} & -\frac{p^4-p^3+6}{12} & \frac{p-7}{12} & \frac{p^3-p^2-6}{12} & \frac{p-7}{12} & 1 & 0 & 0 & 0 \\
    C_{3,1} & \frac{p^3-p^2-6}{12} & \frac{p-7}{12} & -\frac{p^2+5}{6} & \frac{p-7}{12} & \frac{p-7}{12} & 1 & 1 & 0 & 0  \\
    C_{1,3} & \frac{p-7}{12}  & \frac{p^3-p^2-6}{12} & \frac{p-7}{12}  & -\frac{p^2+5}{6} & \frac{p-7}{12} & 1 & 0 & 0 & 1 \\
     C_{2,2} & \frac{p-7}{12} & \frac{p-7}{12} & \frac{p-7}{12} & \frac{p-7}{12} & -1 & 1 & 0 & 1 & 0 \\
    E & 1 & 1 & 1 & 1 & 1 & -2 & 0 & 0 & 0 \\
    F_1 & 0 & 0 & 1 & 0 & 0 & 0 & -3 & 0 & 0\\
    F_2 & 0 & 0 & 0 & 0 & 1 & 0 & 0 & -3 & 0\\
    F_3 & 0 & 0 & 0 & 1 & 0 & 0 & 0 & 0 & -3.
    \end{array}
  \end{equation*}
   In the above matrix, $E$ corresponds to $j=1728$, and $F_1, F_2, F_3$ correspond to $j=0$.
\end{prop}
\begin{proof}
   The principal divisor is given by
   $$V_p= C_{4,0}+C_{0,4}+(p-1)\big( C_{3,1} + C_{1,3}\big)+p(p-1)C_{2,2}+\frac{p(p+1)}{2}E+ \frac{p-1}{3}\big( F_1 + pF_2+F_3\big).$$ 
   At the point $\alpha_i$ ($i=1,\ldots, k=\frac{p-7}{12}$) local intersections are same as before.  From \cite[Fig. 1.3.2.3]{MR1056773}, we have the Figure \ref{fig:7mod12betapower4}. 
  \begin{figure}[h]
  \begin{center}
    \includegraphics[scale=0.4]{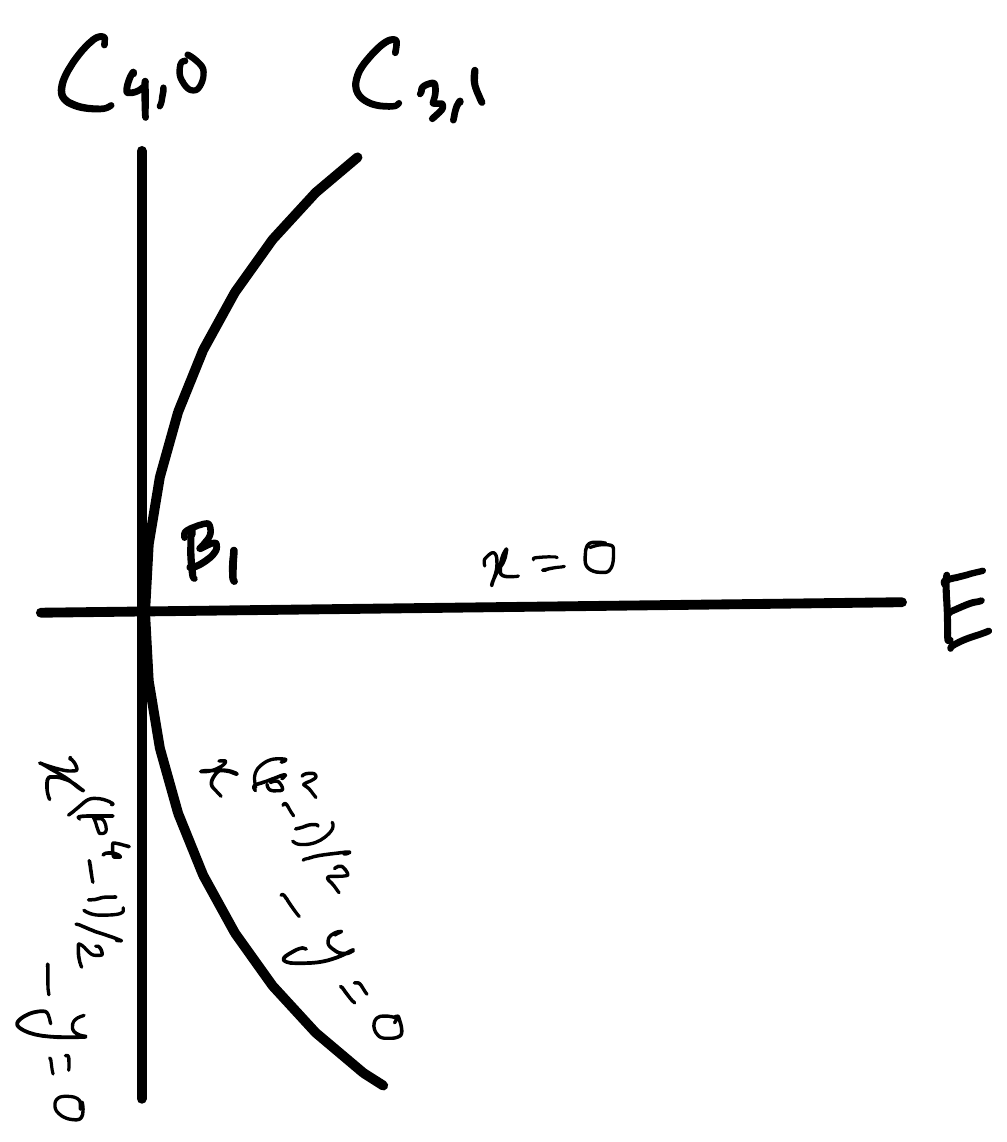}
  \end{center}
  \caption{The intersection point $\beta_1$ when $p \equiv 7 \pmod{12}$.}
  \label{fig:7mod12betapower4}
\end{figure} 
The local intersections at the point $\beta_1$ is given by
   \begin{align*}
       & i_{\beta_1}\big(C_{4,0},\, C_{3,1}\big)= \dim_{{\ff_p}} \frac{\ff_p[x,y]_{(x,y)}}{\big( x^{(p^2-1)/2}-y,\,  x^{(p^4-1)/2}-y\big)}=\frac{p^2-1}{2},\\
       &i_{\beta_1}\big(C_{4,0},\,E\big)= 1, \,\ \,\ i_{\beta_1}\big(C_{3,1},\,E\big)= 1.
  \end{align*}
  Using these local intersection numbers we get our required matrix.
\end{proof}
\begin{prop}\label{minmatrixcase3power4}
For $p \equiv 7 \pmod {12}$, the minimal regular model $\cX_0(p^4)$ is obtained from $\tcX_0(p^4)$ by blowing down $C_{2,2}$, $E$ and $F_2$. The local intersection numbers of the prime divisors supported on the special fiber of $\cX_0(p^4)$ are given by the following matrix.
  \begin{equation*}
    \renewcommand*{\arraystretch}{2}
    \begin{array}{l|cccccc}  
      & C_{4,0}' & C_{0,4}' &  C_{3,1}' & \hphantom{0} C_{1,3}' \hphantom{0} & \hphantom{0}  F_1'  & \hphantom{0} F_3'
        \\ \hline
      C_{4,0}' & -\frac{ 2p^4 - 2p^3 - p^2 + 2p - 1}{24} & \frac{ p^2 - 1}{24} & \frac{ 2p^3 - p^2 - 2p + 1 }{24} & \frac{p^2 - 1}{24} & 0 & 0 \\
      C_{0,4}' & \frac{p^2 - 1}{24} & -\frac{ 2p^4 - 2p^3 - p^2 + 2p - 1}{24} & \frac{p^2-1}{24} & \frac{ 2p^3 - p^2 - 2p + 1}{24} & 0 & 0  \\
      C_{3,1}' & \frac{2p^3 - p^2 - 2p + 1}{24} & \frac{p^2-1}{24} & -\frac{3p^2 + 2p + 7}{24} & \frac{p^2-1}{24} & 1 & 0  \\ 
       C_{1,3}' & \frac{p^2-1}{24} & \frac{ 2p^3 - p^2 - 2p + 1}{24} & \frac{p^2-1}{24} & -\frac{3p^2 + 2p + 7}{24} & 0 & 1  \\ 
      E_1' & 0 & 0 & 1 & 0 & -3 & 0 \\
      E_3' & 0 & 0 & 0 & 1 & 0 & -3. 
    \end{array}
  \end{equation*}
 In the above matrix, $C_{4,0}'$ $C_{0,4}'$, $C_{3,1}'$, $C_{1,3}'$, $F_1'$ and $F_3'$ denote the images of $C_{4,0}$, $C_{0,4}$, $C_{3,1}$, $C_{1,3}$, $F_1$ and $F_3$, respectively under the blow down morphism $\tcX_0(p^4) \to \cX_0(p^4)$.
\end{prop}
\begin{proof}
  The minimal regular model $\cX_0(p^4)$ is obtained by blowing down $C_{2,2}$, then the image of $E$ and then the image of $F_2$. Let $\pi: \tcX_0(p^4) \to \cX_0(p^4)$ be the morphism obtained by composing the sequence of blow downs.
  
  The special fiber of $\cX_0(p^4)$ consists of $C_{4,0}'$ $C_{0,4}'$, $C_{3,1}'$, $C_{1,3}'$, $F_1'$ and $F_3'$ that are the images 
  of $C_{4,0}$, $C_{0,4}$, $C_{3,1}$, $C_{1,3}$, $F_1$ and $F_3$, respectively under $\pi$.
  Let $\pi^* C_{a,b}' = C_{a,b} + \alpha C_{2,2} + \beta E_2 + \gamma F$, where $C_{a,b} \in\lbrace C_{4,0}, C_{0,4}, C_{3,1}, C_{1,3}\rbrace$. Since the intersection of $\pi^* C_{a,b}'$ with $C_{2,2}, E$ and $F_2$ are zero, then by solving 
  \begin{align*}
       \begin{bmatrix}
-1 & 1 & 1\\
1 & -2 & 0\\
1 & 0 & -3
\end{bmatrix}\begin{bmatrix}
\alpha \\
\beta\\
\gamma
\end{bmatrix}=\begin{bmatrix}
-\frac{p-7}{12}\\
-1\\
0
\end{bmatrix}
  \end{align*}
 we get
  \begin{align*}
    &\pi^* C_{a,b}' = C_{a,b} + \frac{p-1}{2} C_{2,2} + \frac{p+1}{4} E + \frac{p-1}{6} F_2,
    \end{align*}
    where $C_{a,b} \in\lbrace C_{4,0}, C_{0,4}, C_{3,1}, C_{1,3} \rbrace$.
    Also, note that
    \begin{align*}
    &\pi^*F_1'=F_1, \,\ \pi^*F_3'=F_3.
  \end{align*} 
  Finally, using \cite[Chapter 9, Theorem~2.12 (c)]{MR1917232} we get our required matrix.
\end{proof}
\vspace{.02cm}
\subsection{Case \texorpdfstring{$p \equiv 11 \pmod{12}$}{}.}
From Edixhoven~\cite{MR1056773}, we draw the special fiber $V_p = \tcX_0(p^4)_{\ff_p}$ in Figure~\ref{fig:11mod12power4}, where each component is a $\pp^1$. In this case both $j=1728$ and $j=0$ are supersingular.

\begin{figure}[h]
  \begin{center}
    \includegraphics[scale=0.5]{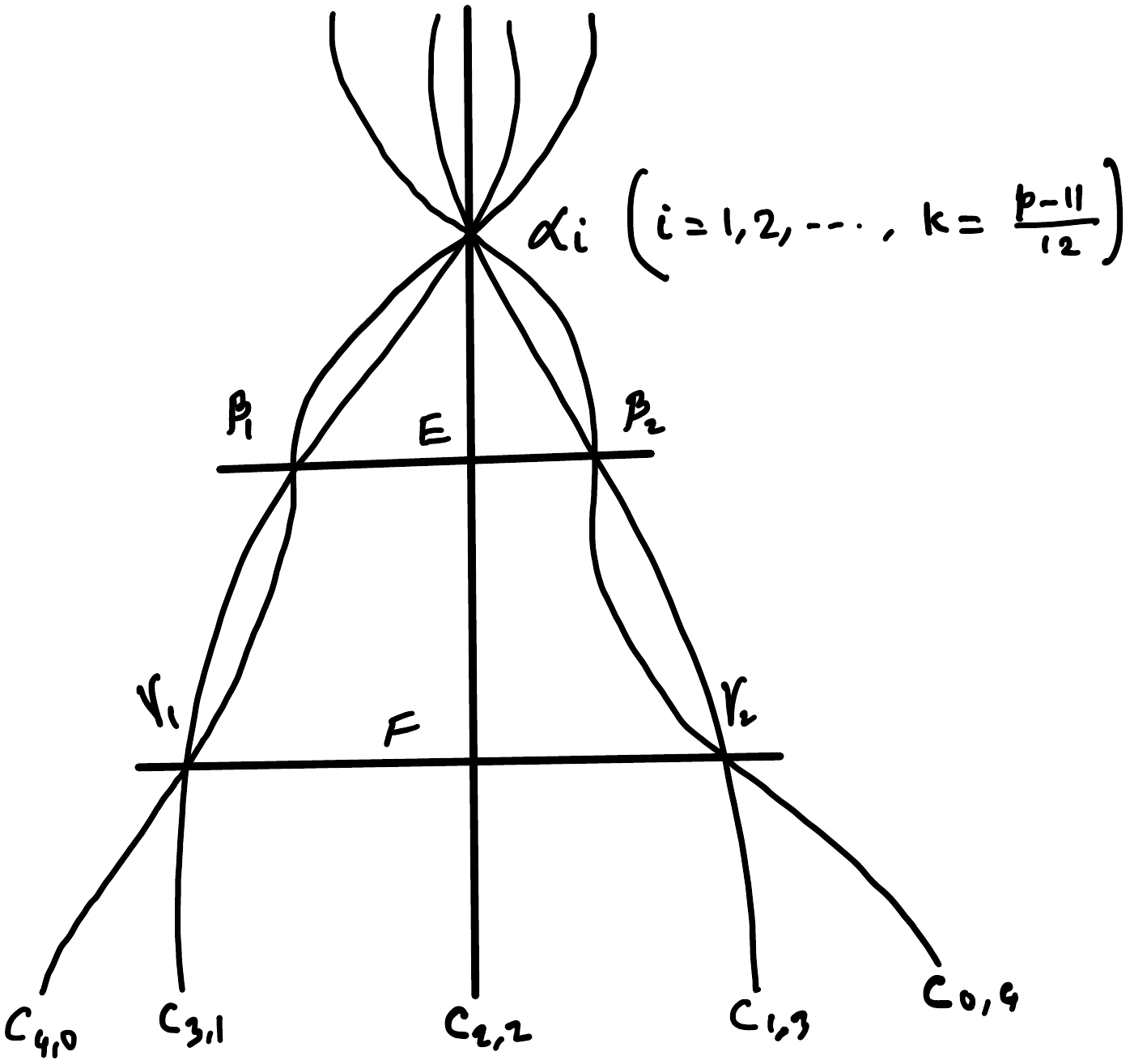}
  \end{center}
  \caption{The special fiber $\tcX_0(p^4)_{\ff_p}$ when $p \equiv 11 \pmod{12}$.}
  \label{fig:11mod12power4}
\end{figure}
\begin{prop} \label{prop11mod12power4}
  The local intersection numbers of the prime divisors supported on the special fiber
  of $\tcX_0(p^4)$ for $p \equiv 11 \pmod {12}$ are given by the following matrix:

  \begin{equation*}
    \renewcommand*{\arraystretch}{2}
    \begin{array}{l|ccccccc}
        & C_{4,0} & C_{0,4} &  C_{3,1} & \hphantom{0} C_{1,3} \hphantom{0} & \hphantom{0} C_{2,2} \hphantom{0} & \hphantom{0}  E \hphantom{0} & \hphantom{0} F  \\ \hline
    C_{4,0} & -\frac{p^4-p^3+10}{12} & \frac{p-11}{12} & \frac{p^3-p^2-10}{12} & \frac{p-11}{12} & \frac{p-11}{12} & 1 & 1 \\
    C_{0,4} & \frac{p-11}{12} & -\frac{p^4-p^3+10}{12} & \frac{p-11}{12} & \frac{p^3-p^2-10}{12} & \frac{p-11}{12} & 1 & 1 \\
    C_{3,1} & \frac{p^3-p^2-10}{12} & \frac{p-11}{12} & -\frac{p^2+5}{6} & \frac{p-11}{12} & \frac{p-11}{12} & 1 & 1 \\
    C_{1,3} & \frac{p-11}{12}  & \frac{p^3-p^2-10}{12} & \frac{p-11}{12}  & -\frac{p^2+5}{6} & \frac{p-11}{12} & 1 & 1 \\
     C_{2,2} & \frac{p-11}{12} & \frac{p-11}{12} & \frac{p-11}{12} & \frac{p-11}{12} & -1 & 1 & 1 \\
    E & 1 & 1 & 1 & 1 & 1 & -2 & 0 \\
    F & 1 & 1 & 1 & 1 & 1 & 0 & -3.
    \end{array}
  \end{equation*}
  In the above matrix, $E$ corresponds to $j=1728$, and $F$ corresponds to $j=0$.
\end{prop}

\begin{proof}
   The principal divisor is given by
   $$V_p= C_{4,0}+C_{0,4}+(p-1)\big( C_{3,1} + C_{1,3}\big)+p(p-1)C_{2,2}+\frac{p(p+1)}{2}E+ \frac{p(p+1)}{3} F.$$
   Local intersections at $\beta_1$ and $\gamma_1$ are given by 
\begin{align*}
      & i_{\beta_1}\big(C_{4,0},\, C_{3,1}\big)=\dim_{{\ff_p}} \frac{\ff_p[x,y]_{(x,y)}}{\big( x^{(p^2-1)/2}-y,\,  x^{(p^4-1)/2}-y\big)}=\frac{p^2-1}{2}, \\& i_{\gamma_1}\big(C_{4,0},\, C_{3,1}\big)=\dim_{{\ff_p}} \frac{\ff_p[x,y]_{(x,y)}}{\big( x^{(p^2-1)/3}-y,\,  x^{(p^4-1)/3}-y\big)}=\frac{p^2-1}{3}.
  \end{align*}
  Using these local intersection numbers we get our required matrix. 
\end{proof}
\vspace{.02cm}
\begin{prop}\label{minmatrixcase4power4}
  For $p \equiv 11 \pmod {12}$ the minimal regular model $\cX_0(p^4)$ is obtained from 
  $\tcX_0(p^4)$ by blowing down $C_{2,2}$, $E$ and $F$.
  The local intersection numbers of the prime divisors supported on the special fiber
  of $\cX_0(p^4)$ are given by the following matrix:
 
  \begin{equation*}
    \renewcommand*{\arraystretch}{2}
    \begin{array}{l|cccc}  
      & C_{4,0}' & C_{0,4}' &  C_{3,1}' & \hphantom{0} C_{1,3}'
        \\ \hline
      C_{4,0}' & -\frac{ 2p^4 - 2p^3 - p^2 + 2p - 1}{24} & \frac{ p^2 - 1}{24} & \frac{ 2p^3 - p^2 - 2p + 1 }{24} & \frac{p^2 - 1}{24} \\
      C_{0,4}' & \frac{p^2 - 1}{24} & -\frac{ 2p^4 - 2p^3 - p^2 + 2p - 1}{24} & \frac{p^2-1}{24} & \frac{ 2p^3 - p^2 - 2p + 1}{24}   \\
      C_{3,1}' & \frac{2p^3 - p^2 - 2p + 1}{24} & \frac{p^2-1}{24} & -\frac{3p^2 + 2p -1}{24} & \frac{p^2-1}{24}  \\ 
       C_{1,3}' & \frac{p^2-1}{24} & \frac{ 2p^3 - p^2 - 2p + 1}{24} & \frac{p^2-1}{24} & -\frac{3p^2 + 2p -1}{24}.
    \end{array}
  \end{equation*}
  In the above matrix, $C_{4,0}'$ $C_{0,4}'$, $C_{3,1}'$ and $C_{1,3}'$ denote the images 
  of $C_{4,0}$, $C_{0,4}$, $C_{3,1}$ and $C_{1,3}$, respectively under the blow down morphism
  $\tcX_0(p^4) \to \cX_0(p^4)$.
\end{prop}
\begin{proof}
  As before the minimal regular model $\cX_0(p^4)$ is obtained by blowing down $C_{2,2}$, then the image of $E$ and then the image of $F$. Let $\pi: \tcX_0(p^4) \to \cX_0(p^4)$ be the morphism obtained by composing the sequence of blow downs.
  
  The special fiber of $\cX_0(p^4)$ consists of $C_{4,0}'$ $C_{0,4}'$, $C_{3,1}'$ and $C_{1,3}'$ that are the images 
  of $C_{4,0}$, $C_{0,4}$, $C_{3,1}$ and $C_{1,3}$, respectively under $\pi$. Let $\pi^* C_{a,b}' = C_{a,b} + \alpha C_{2,2} + \beta E_2 + \gamma F$, where $C_{a,b} \in\lbrace C_{4,0}, C_{0,4}, C_{3,1}, C_{1,3}\rbrace$. Since the intersection of $\pi^* C_{a,b}'$ with $C_{2,2}, E$ and $F$ are zero, then by solving 
  \begin{align*}
       \begin{bmatrix}
-1 & 1 & 1\\
1 & -2 & 0\\
1 & 0 & -3
\end{bmatrix}\begin{bmatrix}
\alpha \\
\beta\\
\gamma
\end{bmatrix}=\begin{bmatrix}
-\frac{p-7}{12}\\
-1\\
-1
\end{bmatrix}
  \end{align*}
 we get
  \begin{align*}
    &\pi^* C_{a,b}' = C_{a,b} + \frac{p-1}{2} C_{2,2} + \frac{p+1}{4} E + \frac{p+1}{6} F,
    \end{align*}
    where $C_{a,b} \in\lbrace C_{4,0}, C_{0,4}, C_{3,1}, C_{1,3} \rbrace$.
  Finally, using \cite[Chapter 9, Theorem~2.12 (c)]{MR1917232} we get our required matrix.
\end{proof}
\section{Arakelov divisor perpendicular to all vertical divisors}
\subsection{For the modular curve \texorpdfstring{$\cX_0(p^3)$}{}}\label{Construction_with_sage_power3}
\begin{figure}
  \begin{center}
    \includegraphics[scale=0.4]{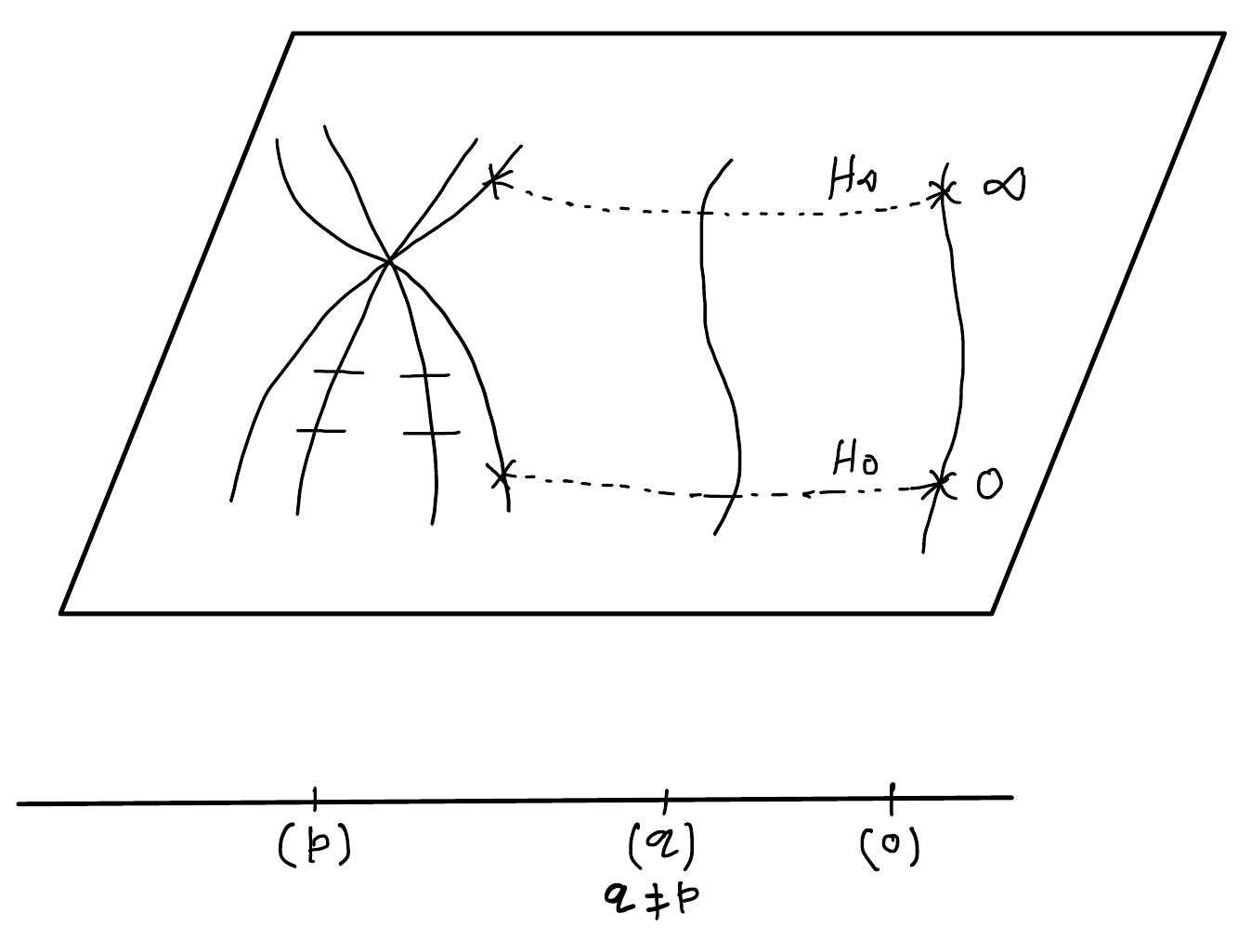}
  \end{center}      
  \caption{$\cX_0(p^3)$ when $p \equiv 1 \pmod{12}$.} 
  \label{fig:alg1mod12} 
\end{figure}
Let $H_0$ and $H_{\infty}$ be the 
sections of $\cX_0(p^3)/\zz$ corresponding to the cusps $0, \infty \in X_0(p^3)(\qq)$. From Liu~\cite[Chapter~9, Proposition~1.30, and Corollary~1.32]{MR1917232}, we know that the horizontal divisor $H_0$ intersects exactly one component of multiplicity $1$ of the special fiber at an 
$\ff_p$ rational point transversally. Without loss of generality we assume that $H_0$ intersects $C_{0,3}$. It follows from the cusp and component labeling of Katz and 
Mazur~\cite[p.~296]{MR772569} that $H_{\infty}$ must intersects the only other component of multiplicity $1$ namely $C_{3,0}$.

Let $ K_{\cX_0(p^3)}$ be a canonical divisor of $\cX_0(p^3)$, that is any divisor whose corresponding line bundle is the relative dualizing sheaf. Note that, in this section, while computing $V_0\cdot V_\infty$ we write $``\cdots"$, these are some insignificant terms involving some powers of the prime $p$.
\subsection{Case \texorpdfstring{$p \equiv 1 \pmod{12}$}{}}In this case, we prove the following results.
\begin{lem}\label{case1Vmpower3}
For $p \equiv 1 \pmod{12}$,
consider the vertical divisors 
\begin{align*}
V_0= x_1C_{3,0}+x_2C_{0,3}+x_3C_{2,1}+x_4C_{1,2}+x_5E_{1}+x_6E_{2}+x_7F_{1}+x_8F_{2};\\
V_\infty= x'_1C_{3,0}+x'_2C_{0,3}+x'_3C_{2,1}+x'_4C_{1,2}+x'_5E_{1}+x'_6E_{2}+x'_7F_{1}+x'_8F_{2}.
\end{align*}
Then the divisors 
  \begin{equation*}
    D_m = K_{\cX_0(p^3)} - (2 g_{p^3} - 2) H_m + V_m, \qquad m\in \{0, \infty\}
  \end{equation*}
  are orthogonal to all vertical divisors of $\cX_0(p^3)$ with respect to the Arakelov intersection pairing, and $x_i, x'_i$ $(i=1,\ldots, 8)$ are given as follows.
  \begin{align*}
    & x_1=\frac{2p^4 - 2p^3 - 16p^2 - 16p + 14}{p^4 - p^2},\,\ x_2=\frac{-p^3 - 4p^2 + 21p + 14}{p^3 - p},\\
    &x_3=\frac{p^4 - p^3 - 14p^2 - 2p + 14}{p^3 + p^2},\,\ x_4=-1,\,\ x_5=\frac{\frac{1}{2}p^4 - \frac{1}{2}p^3 - 7p^2 - p + 7}{p^3 + p^2},\\
    & x_6=-\frac{1}{2},\,\ x_7=\frac{\frac{1}{3}p^4 - \frac{13}{3}p^2 - \frac{2}{3}p + \frac{14}{3}}{p^3 + p^2}, \,\ x_8=0,
\end{align*}
and 
\begin{align*}
     & x'_1=\frac{-2p^4 - 4p^3 + 34p^2 + 16p - 14}{p^4 - p^2},\,\ x'_2=\frac{p^3 - 2p^2 - 3p - 14}{p^3 - p},\\
    &x'_3=\frac{-p^4 - p^3 + 12p^2 + 2p - 14}{p^3 + p^2},\,\ x'_4=-1,\,\ x'_5=\frac{-\frac{1}{2}p^4 - \frac{1}{2}p^3 + 6p^2 + p - 7}{p^3 + p^2},\\
    & x'_6=-\frac{1}{2},\,\ x'_7=\frac{-\frac{1}{3}p^4 + \frac{13}{3}p^2 + \frac{2}{3}p - \frac{14}{3}}{p^3 + p^2}, \,\ x'_8=0.
\end{align*}
\end{lem}
\begin{proof}
  Suppose $V_0= x_1C_{3,0}+x_2C_{0,3}+x_3C_{2,1}+x_4C_{1,2}+x_5E_{1}+x_6E_{2}+x_7F_{1}+x_8F_{2}$,
  satisfies the hypothesis of the lemma. Then for any prime vertical divisor $W$ supported on the 
  special fiber we must have $(K_{\cX_0(p^3)} - (2 g_{p^3} - 2) H_0 + V_0)\cdot W = 0$. 
  This yields
  \begin{equation*}
    V_0\cdot W = - K_{\cX_0(p^3)}\cdot W + (2 g_{p^3} - 2) H_0 \cdot W,
  \end{equation*}
  where $W \in \{ C_{3,0}, C_{0,3}, C_{2,1}, C_{1,2}, E_1, E_2, F_1, F_2\}$.
Since any component of the special fiber has genus $0$, using the adjunction formula from Liu~\cite[Chapter~9, Theorem~1.37]{MR1917232}, we have $K_{\cX_0(p^3)}\cdot W = - (W^2 + 2)$.
Moreover $H_0$ only intersects $C_{0,3}$ transversally and no other component, so we have a system of linear equations involving the $x_i$ $(i=1,\ldots, 8)$. These linear equations are given by
\begin{equation}\label{vdotC}
    V_0\cdot W=\begin{cases}
    2+ C^2_{0,3}+ (2 g_{p^3} - 2), & W=C_{0,3},\\
    2+ W^2, & W \in \{ C_{3,0}, C_{2,1}, C_{1,2}, E_1, E_2, F_1, F_2\}.
    \end{cases}
\end{equation}
The intersection numbers $C^2_{a,b}$, $E_1^2$, $E_2^2$, $F_1^2$ and $F_2^2$ were computed in Proposition \ref{pmod1imp}.
Now, by using SageMath \cite{sage} we solve the equations \eqref{vdotC} and get the required vertical divisor $V_0$. 

Similarly, to determine $V_{\infty}$ we solve the system of linear equations in $x_i'$ $(i=1,\ldots, 8)$ are given by 
\vspace{.02cm}
\begin{equation}
    V_\infty \cdot W=\begin{cases}
    2+ C^2_{3,0}+ (2 g_{p^3} - 2), & W=C_{3,0},\\
    2+ W^2, & W \in \{ C_{0,3}, C_{2,1}, C_{1,2}, E_1, E_2, F_1, F_2\}.
    \end{cases}
\end{equation}
  Finally, for any prime $q \in \zz\setminus \{p\}$ the fiber $V$ over $(q) \in \spec \zz$ is 
  is irreducible and $\langle V_m, V \rangle = 0$ since $V_m$ is supported on the fiber 
  over $(p)$. 
  Moreover, by the adjunction formula \cite[Chapter 9, Proposition 1.35]
  {MR1917232}, $\langle K_{\cX_0(p^3)}, V \rangle = (2 g_{p^3} - 2)\log p$. 
  The horizontal divisor $H_m$ meets any fiber transversally at a smooth $\ff_p$ rational 
  point which gives $\langle D_m, V \rangle = 0$.This completes the proof.
\end{proof}
\vspace{.02cm}
\begin{prop}
With the above notations for $p \equiv 1 \pmod{12}$, we have
\begin{align*}
     &\frac{1}{g_{p^3}-1}\left(g_{p^3}\langle V_0, V_{\infty}\rangle-\frac{V^2_0+V^2_\infty}{2}\right)= 3g_{p^3}\log (p)+o (g_{p^3}\log p)\,\ \text{as}\,\ p\to \infty.
\end{align*}
\end{prop}
\begin{proof}
  From Example \ref{genusp3}, we recall that
  \begin{align*}
      g_{p^3}-1=\frac{p(p+4)(p-3)-14}{12}.
  \end{align*}
  \vspace{.02cm}
  Now, note that by multiplying $V_0$ and $V_\infty$, we get the following expression.
 
  \begin{align*}
     V_0\cdot V_{\infty}=& x_1x'_1C^2_{3,0}+x_2x'_2C^2_{0,3}+x_3x'_3C^2_{2,1}+x_4x'_4C^2_{1,2}+x_5x'_5E_1^2+x_6x'_6E_2^2+x_7x'_7F_1^2\\&+(x_1x'_2+x_2x'_1)C_{3,0}C_{0,3}+(x_1x'_3+x_3x'_1)C_{3,0}C_{2,1}+(x_1x'_4+x_4x'_1)C_{3,0}C_{1,2}\\
    &+(x_1x'_5+x_5x'_1)C_{3,0}E_1+(x_1x'_6+x_6x'_1)C_{3,0}E_2+(x_1x'_7+x_7x'_1)C_{3,0}F_1\\
    &+(x_2x'_3+x_3x'_2)C_{0,3}C_{2,1}
    +(x_2x'_4+x_4x'_2)C_{0,3}C_{1,2}+(x_2x'_5+x_5x'_2)C_{0,3}E_1\\
    &+(x_2x'_6+x_6x'_2)C_{0,3}E_2+(x_2x'_7+x_7x'_2)C_{0,3}F_1+(x_3x'_4+x_4x'_3)C_{2,1}C_{1,2}\\&+(x_3x'_5+x_5x'_3)C_{2,1}E_1+(x_3x'_6+x_6x'_3)C_{2,1}E_2+(x_3x'_7+x_7x'_3)C_{2,1}F_1\\
    &+(x_4x'_5+x_5x'_4)C_{1,2}E_1+(x_4x'_6+x_6x'_4)C_{1,2}E_2+(x_4x'_7+x_7x'_4)C_{1,2}F_1\\
    &+(x_5x'_6+x_6x'_5)E_1E_2+(x_5x'_7+x_7x'_5)E_1F_1+(x_6x'_7+x_7x'_6)E_2F_1. 
  \end{align*}
Using Lemma \ref{case1Vmpower3}, we have
\begin{align*}
    V_0\cdot V_{\infty}=&\left(\frac{-4p^{8}+\cdots}{p^{8}+\cdots}\right)C^2_{3,0}+\left(\frac{-p^6+\cdots}{p^6+\cdots}\right)C^2_{0,3}+\left(\frac{-p^{8}+\cdots}{p^6+\cdots}\right)C^2_{2,1}\\&+\left(\frac{-4p^{8}+\cdots}{p^7+\cdots}\right)C_{3,0}C_{2,1}+\cdots
\end{align*}
Now, using Proposition \ref{pmod1imp}, we get
\begin{align*}
   V_0\cdot V_{\infty}=&  \left(\frac{-4p^{8}+\cdots}{p^{8}+\cdots}\right)\left(\frac{-p^{3}+\cdots}{12}\right)+
    \left(\frac{-p^6+\cdots}{p^6+\cdots}\right)\left(\frac{-p^{3}+\cdots}{12}\right)\\&+\left(\frac{-p^{8}+\cdots}{p^6+\cdots}\right)\left(\frac{-p+\cdots}{6}\right)+\left(\frac{-4p^{8}+\cdots}{p^7+\cdots}\right)\left(\frac{p^2+\cdots}{12}\right)+\cdots
\end{align*}
Using the expression of genus, we derive the asymptotics 
 \begin{align}\label{mod1eq1}
    g_{p^3}\cdot \frac{\langle V_0, V_{\infty}\rangle}{ g_{p^3}-1}& =g_{p^3}\cdot \frac{ V_0\cdot V_{\infty}\log p}{g_{p^3}-1}=3g_{p^3}\log (p)  + o(\log p) \,\ \text{as}\,\ p\to \infty.
 \end{align}
 From Lemma \ref{case1Vmpower3} and from Proposition \ref{pmod1imp}, we also derive the estimate 
 \begin{align}\label{mod1eq2}
     \frac{V^2_0+V^2_\infty}{2(g_{p^3}-1)}=o (g_{p^3}\log p)\,\ \text{as}\,\ p\to \infty.
 \end{align}
Then the proof directly follows from \eqref{mod1eq1} and \eqref{mod1eq2}.
\end{proof}
\subsection{Case \texorpdfstring{$p \equiv 5 \pmod{12}$}{}}In this case, we prove the following results.
\begin{lem}\label{case2Vmpower3}
For $p \equiv 5 \pmod{12}$,
consider the vertical divisors 
\begin{align*}
V_0= x_1C_{3,0}+x_2C_{0,3}+x_3C_{2,1}+x_4C_{1,2}+x_5E_{1}+x_6E_{2}+x_7F_{1}+x_8F_{2};\\
V_\infty= x'_1C_{3,0}+x'_2C_{0,3}+x'_3C_{2,1}+x'_4C_{1,2}+x'_5E_{1}+x'_6E_{2}+x'_7F_{1}+x'_8F_{2}.
\end{align*}
Then the divisors 
  \begin{equation*}
    D_m = K_{\cX_0(p^3)} - (2 g_{p^3} - 2) H_m + V_m, \qquad m\in \{0, \infty\}
  \end{equation*}
  are orthogonal to all vertical divisors of $\cX_0(p^3)$ with respect to the Arakelov intersection pairing, and $x_i, x'_i$ $(i=1,\ldots, 8)$ are given as follows.
\begin{align*}
    & x_1=\frac{\frac{5}{3}p^4 - \frac{2}{3}p^3 - \frac{25}{3}p^2 - 16p + 2}{p^4 - p^2},\,\ x_2=\frac{-\frac{4}{3}p^4 - \frac{8}{3}p^3 + \frac{86}{3}p^2 - 10p - 4}{p^4 - p^2},\\
    &x_3=\frac{\frac{2}{3}p^4 + \frac{1}{3}p^3 - \frac{19}{3}p^2 - 10p + 2}{p^3 + p^2},\,\ x_4=\frac{-\frac{1}{3}p^4 + \frac{1}{3}p^3 + \frac{20}{3}p^2 - 16p - 4}{p^3 + p^2},\\ &x_5=\frac{\frac{1}{3}p^4 + \frac{1}{6}p^3 - \frac{19}{6}p^2 - 5p + 1}{p^3 + p^2},\,\
     x_6=\frac{-\frac{1}{6}p^4 + \frac{1}{6}p^3 + \frac{10}{3}p^2 - 8p - 2}{p^3 + p^2},\\ &x_7=\frac{\frac{1}{3}p^3 + \frac{1}{3}p^2 - 4p - 2}{p^2 - p}, \,\ x_8=0,
\end{align*}
and 
\begin{align*}
     & x'_1=\frac{-\frac{5}{3}p^4 - \frac{10}{3}p^3 + \frac{97}{3}p^2 - 4p - 2}{p^4 - p^2},\,\ x'_2=\frac{\frac{4}{3}p^4 - \frac{4}{3}p^3 - \frac{14}{3}p^2 - 10p + 4}{p^4 - p^2},\\
    &x'_3=\frac{-\frac{2}{3}p^4 - \frac{1}{3}p^3 + \frac{31}{3}p^2 - 10p - 2}{p^3 + p^2},\,\ x'_4=\frac{\frac{1}{3}p^4 - \frac{1}{3}p^3 - \frac{8}{3}p^2 - 4p + 4}{p^3 + p^2},\\ &x'_5=\frac{-\frac{1}{3}p^4 - \frac{1}{6}p^3 + \frac{31}{6}p^2 - 5p - 1}{p^3 + p^2},\,\
    x'_6=\frac{\frac{1}{6}p^4 - \frac{1}{6}p^3 - \frac{4}{3}p^2 - 2p + 2}{p^3 + p^2},\\ &x'_7=\frac{-\frac{1}{3}p^3 - \frac{1}{3}p^2 + 4p + 2}{p^2 - p}, \,\ x'_8=0.
\end{align*}
\end{lem}

\begin{proof}
To determine $V_0$ and $V_\infty$ we follow the line of proof of Proposition \ref{case1Vmpower3}. More precisely we get $x_i$ $(i=1,\ldots, 8)$, by solving the following linear equations using SageMath \cite{sage}
\begin{equation*}
    V_0\cdot W=\begin{cases}
    2+ C^2_{0,3}+ (2 g_{p^3} - 2), & W=C_{0,3},\\
    2+ W^2, & W \in \{ C_{3,0}, C_{2,1}, C_{1,2}, E_1, E_2, F_1, F_2\}.
    \end{cases}
\end{equation*}
Note that, here $C^2_{a,b}$, $E_i^2$, and $F_i^2$ were computed in Proposition \ref{prop5mod12}. Similarly, we compute $x'_i$ ($i=1,\ldots, 8$), and this completes the proof.
\end{proof}
\begin{prop}
With the above notations for $p \equiv 5 \pmod{12}$, we have
\begin{align*}
     &\frac{1}{g_{p^3}-1}\left(g_{p^3}\langle V_0, V_{\infty}\rangle-\frac{V^2_0+V^2_\infty}{2}\right)= 3g_{p^3}\log (p)+o (g_{p^3}\log p)\,\ \text{as}\,\ p\to \infty.
\end{align*}
\end{prop}
\begin{proof}
  From Example \ref{genusp3}, we recall that
  \begin{align*}
      g_{p^3}-1=\frac{p(p+4)(p-3)-6}{12}.
  \end{align*}
  Now, note that
  \begin{align*}
    V_0\cdot V_{\infty}=& x_1x'_1C^2_{3,0}+x_2x'_2C^2_{0,3}+x_3x'_3C^2_{2,1}+x_4x'_4C^2_{1,2}+x_5x'_5E_1^2+x_6x'_6E_2^2+x_7x'_7F_1^2\\&+(x_1x'_2+x_2x'_1)C_{3,0}C_{0,3}+(x_1x'_3+x_3x'_1)C_{3,0}C_{2,1}+(x_1x'_4+x_4x'_1)C_{3,0}C_{1,2}\\
    &+(x_1x'_5+x_5x'_1)C_{3,0}E_1+(x_1x'_6+x_6x'_1)C_{3,0}E_2+(x_1x'_7+x_7x'_1)C_{3,0}F_1\\
    &+(x_2x'_3+x_3x'_2)C_{0,3}C_{2,1}
    +(x_2x'_4+x_4x'_2)C_{0,3}C_{1,2}+(x_2x'_5+x_5x'_2)C_{0,3}E_1\\
    &+(x_2x'_6+x_6x'_2)C_{0,3}E_2+(x_2x'_7+x_7x'_2)C_{0,3}F_1+(x_3x'_4+x_4x'_3)C_{2,1}C_{1,2}\\&+(x_3x'_5+x_5x'_3)C_{2,1}E_1+(x_3x'_6+x_6x'_3)C_{2,1}E_2+(x_3x'_7+x_7x'_3)C_{2,1}F_1\\
    &+(x_4x'_5+x_5x'_4)C_{1,2}E_1+(x_4x'_6+x_6x'_4)C_{1,2}E_2+(x_4x'_7+x_7x'_4)C_{1,2}F_1\\
    &+(x_5x'_6+x_6x'_5)E_1E_2+(x_5x'_7+x_7x'_5)E_1F_1+(x_6x'_7+x_7x'_6)E_2F_1.
  \end{align*}
Using Lemma \ref{case2Vmpower3}, we have
\begin{align*}
    V_0\cdot V_{\infty}=& \left(\frac{-\frac{25}{9}p^{8}+\cdots}{p^{8}+\cdots}\right)C^2_{3,0}+\left(\frac{-\frac{16}{9}p^8+\cdots}{p^8+\cdots}\right)C^2_{0,3}+\left(\frac{-\frac{4}{9}p^{8}+\cdots}{p^6+\cdots}\right)C^2_{2,1}+\left(\frac{-\frac{1}{9}p^{8}+\cdots}{p^6+\cdots}\right) C^2_{1,2}\\&+\left(\frac{-\frac{20}{9}p^{8}+\cdots}{p^7+\cdots}\right)C_{3,0}C_{2,1}+\left(\frac{-\frac{8}{9}p^{8}+\cdots}{p^7+\cdots}\right)C_{0,3}C_{1,2}+\left(\frac{\frac{4}{9}p^{8}+\cdots}{p^6+\cdots}\right)C_{2,1}C_{1,2}+\cdots
\end{align*}
Now, using Proposition \ref{prop5mod12}, we get
\begin{align*}
   V_0\cdot V_{\infty}=& \left(\frac{-\frac{25}{9}p^{8}+\cdots}{p^{8}+\cdots}\right)\left(\frac{-p^{3}+\cdots}{12}\right)+\left(\frac{-\frac{16}{9}p^8+\cdots}{p^8+\cdots}\right)\left(\frac{-p^{3}+\cdots}{12}\right)\\&+\left(\frac{-\frac{4}{9}p^{8}+\cdots}{p^6+\cdots}\right)\left(\frac{-p+\cdots}{6}\right)+\left(\frac{-\frac{1}{9}p^{8}+\cdots}{p^6+\cdots}\right) \left(\frac{-p+\cdots}{6}\right)\\&+\left(\frac{-\frac{20}{9}p^{8}+\cdots}{p^7+\cdots}\right)\left(\frac{p^2+\cdots}{12}\right)+\left(\frac{-\frac{8}{9}p^{8}+\cdots}{p^7+\cdots}\right)\left(\frac{p^2+\cdots}{12}\right)\\&+\left(\frac{\frac{4}{9}p^{8}+\cdots}{p^6+\cdots}\right)\left(\frac{p+\cdots}{12}\right)+\cdots
\end{align*}
Then, using the expression of genus, we derive the asymptotics
 \begin{align*}
    g_{p^3}\cdot \frac{\langle V_0, V_{\infty}\rangle}{ g_{p^3}-1}& =g_{p^3}\cdot \frac{ V_0\cdot V_{\infty}\log p}{g_{p^3}-1}=3g_{p^3}\log (p)  + o(\log p) \,\ \text{as}\,\ p\to \infty.
 \end{align*}
 From Lemma \ref{case2Vmpower3} and from Proposition \ref{prop5mod12}, we also derive the estimate 
 \begin{align*}
     \frac{V^2_0+V^2_\infty}{2(g_{p^3}-1)}=o (g_{p^3}\log p)\,\ \text{as}\,\ p\to \infty.
 \end{align*}
 This completes the proof.
\end{proof}
\subsection{Case \texorpdfstring{$p \equiv 7 \pmod{12}$}{}.}In this case, we prove the following results.
\begin{lem}\label{case3Vmpower3}
For $p \equiv 7 \pmod{12}$,
consider the vertical divisors 
\begin{align*}
V_0= x_1C_{3,0}+x_2C_{0,3}+x_3C_{2,1}+x_4C_{1,2}+x_5E+x_6F_{1}+x_7F_{2};\\
 V_\infty= x'_1C_{3,0}+x'_2C_{0,3}+x'_3C_{2,1}+x'_4C_{1,2}+x'_5E+x'_6F_{1}+x'_7F_{2}.
 \end{align*}
Then the divisors 
  \begin{equation*}
    D_m = K_{\cX_0(p^3)} - (2 g_{p^3} - 2) H_m + V_m, \qquad m\in \{0, \infty\}
  \end{equation*}
  are orthogonal to all vertical divisors of $\cX_0(p^3)$ with respect to the Arakelov intersection pairing, and $x_i, x'_i$ $(i=1,\ldots, 7)$ are given as follows.
\begin{align*}
    & x_1=\frac{2p^3 - 4p^2 - 12p + 8}{p^3 - p^2},\,\ x_2=\frac{-p^3 - 4p^2 + 21p + 8}{p^3 - p},\,\
    x_3=\frac{p^4 - p^3 - 14p^2 + 4p + 8}{p^3 + p^2},\\ &x_4=-1,\,\ x_5=\frac{\frac{1}{2}p^4 - p^3 - \frac{19}{2}p^2 + 12p + 4}{p^3 - p},\,\
     x_6=\frac{\frac{1}{3}p^4 - \frac{13}{3}p^2 + \frac{4}{3}p + \frac{8}{3}}{p^3 + p^2},\,\ x_7=0,
\end{align*}
and 
\begin{align*}
     & x'_1=\frac{-2p^4 - 4p^3 + 34p^2 + 4p - 8}{p^4 - p^2},\,\ x'_2=\frac{p^3 - 2p^2 - 3p - 8}{p^3 - p},\,\
    x'_3=\frac{-p^4 - p^3 + 12p^2 - 4p - 8}{p^3 + p^2},\\ &x'_4=-1,\,\ x'_5=\frac{-\frac{1}{2}p^4 - p^3 + \frac{7}{2}p^2 + 8p - 4}{p^3 - p},\,\
     x'_6=\frac{-\frac{1}{3}p^4 + \frac{13}{3}p^2 - \frac{4}{3}p - \frac{8}{3}}{p^3 + p^2},\,\ x'_7=0.
\end{align*}
\end{lem}
\begin{proof}
  The proof is analogous to Proposition \ref{case1Vmpower3}. 
\end{proof}
\begin{cor}
With the above notations for $p \equiv 7 \pmod{12}$, we have
\begin{align*}
     &\frac{1}{g_{p^3}-1}\left(g_{p^3}\langle V_0, V_{\infty}\rangle-\frac{V^2_0+V^2_\infty}{2}\right)= 3g_{p^3}\log (p)+o (g_{p^3}\log p)\,\ \text{as}\,\ p\to \infty.
\end{align*}
\end{cor}
\begin{proof}
  From Example \ref{genusp3}, we recall that
  \begin{align*}
      g_{p^3}-1=\frac{p(p+4)(p-3)-8}{12}.
  \end{align*}
  Now, note that
  \begin{align*}
      V_0\cdot V_{\infty}=&x_1x'_1C^2_{3,0}+x_2x'_2C^2_{0,3}+x_3x'_3C^2_{2,1}+x_4x'_4C^2_{1,2}+x_5x'_5E^2+x_6x'_6F_1^2\\&+(x_1x'_2+x_2x'_1)C_{3,0}C_{0,3}+(x_1x'_3+x_3x'_1)C_{3,0}C_{2,1}+(x_1x'_4+x_4x'_1)C_{3,0}C_{1,2}\\&
    +(x_1x'_5+x_5x'_1)C_{3,0}E+(x_1x'_6+x_6x'_1)C_{3,0}F_1+(x_2x'_3+x_3x'_2)C_{0,3}C_{2,1}\\&
    +(x_2x'_4+x_4x'_2)C_{0,3}C_{1,2}+(x_2x'_5+x_5x'_2)C_{0,3}E+(x_2x'_6+x_6x'_2)C_{0,3}F_1\\&
    +(x_3x'_4+x_4x'_3)C_{2,1}C_{1,2}+(x_3x'_5+x_5x'_3)C_{2,1}E+(x_3x'_6+x_6x'_3)C_{2,1}F_1\\&+(x_4x'_5+x_5x'_4)C_{1,2}E
    +(x_4x'_6+x_6x'_4)C_{1,2}F_1+(x_5x'_6+x_6x'_5)EF_1
  \end{align*}
 From Lemma \ref{case3Vmpower3}, we have
 \begin{align*}
    V_0\cdot V_{\infty}=& \left(\frac{-4p^{7}+\cdots}{p^{7}+\cdots}\right)C^2_{3,0}+\left(\frac{-p^6+\cdots}{p^6+\cdots}\right)C^2_{0,3}\\&+\left(\frac{-p^{8}+\cdots}{p^6+\cdots}\right)C^2_{2,1}+\left(\frac{-4p^{7}+\cdots}{p^6+\cdots}\right)C_{3,0}C_{2,1}+\cdots
 \end{align*}
Now, using Proposition \ref{pmod7imp}, we get
\begin{align*}
     V_0\cdot V_{\infty}=&\left(\frac{-4p^{7}+\cdots}{p^{7}+\cdots}\right)\left(\frac{-p^{3}+\cdots}{12}\right)+\left(\frac{-p^6+\cdots}{p^6+\cdots}\right)\left(\frac{-p^{3}+\cdots}{12}\right)\\&+\left(\frac{-p^{8}+\cdots}{p^6+\cdots}\right)\left(\frac{-p+\cdots}{6}\right)+\left(\frac{-4p^{7}+\cdots}{p^6+\cdots}\right)\left(\frac{p^2+\cdots}{12}\right)+\cdots
\end{align*}
Now, using the expression of genus, we derive
\begin{align*}
    g_{p^3}\cdot \frac{\langle V_0, V_{\infty}\rangle}{ g_{p^3}-1}& =g_{p^3}\cdot \frac{ V_0\cdot V_{\infty}\log p}{g_{p^3}-1}=3g_{p^3}\log (p)  + o(\log p) \,\ \text{as}\,\ p\to \infty.
 \end{align*}
 Using Lemma \ref{case3Vmpower3} and from Proposition \ref{pmod7imp}, we also derive the estimate 
 \begin{align*}
     \frac{V^2_0+V^2_\infty}{2(g_{p^3}-1)}=o (g_{p^3}\log p)\,\ \text{as}\,\ p\to \infty.
 \end{align*}
 This completes the proof.
\end{proof}
\subsection{Case \texorpdfstring{$p \equiv 11 \pmod{12}$}{}.}In this case, we prove the following results.
\begin{lem}\label{case4Vmpower3}
For $p \equiv 11 \pmod{12}$,
consider the vertical divisors 
\begin{align*}
V_0= x_1C_{3,0}+x_2C_{0,3}+x_3C_{2,1}+x_4C_{1,2}+x_5E+x_6F_{1}+x_7F_{2};\\
V_\infty= x'_1C_{3,0}+x'_2C_{0,3}+x'_3C_{2,1}+x'_4C_{1,2}+x'_5E+x'_6F_{1}+x'_7F_{2}.
\end{align*}
Then the divisors 
  \begin{equation*}
    D_m = K_{\cX_0(p^3)} - (2 g_{p^3} - 2) H_m + V_m, \qquad m\in \{0, \infty\}
  \end{equation*}
  are orthogonal to all vertical divisors of $\cX_0(p^3)$ with respect to the Arakelov intersection pairing, and $x_i, x'_i$ $(i=1,\ldots, 7)$ are given as follows.
\begin{align*}
    & x_1=\frac{\frac{5}{3}p^2 - \frac{7}{3}p - 6}{p^2 - p},\,\ x_2=\frac{-\frac{4}{3}p^3 - \frac{8}{3}p^2 + \frac{86}{3}p - 18}{p^3 - p},\,\
    x_3=\frac{\frac{2}{3}p^2 - \frac{1}{3}p - 6}{p},\\& x_4=\frac{-\frac{1}{3}p^3 + \frac{1}{3}p^2 + \frac{20}{3}p - 18}{p^2 + p},\,\
    x_5=\frac{\frac{1}{6}p^2 + \frac{1}{6}p - 2}{p - 1},\,\
     x_6=\frac{\frac{1}{3}p^2 + \frac{1}{3}p - 4}{p - 1},\,\ x_7=0,
\end{align*}
and 
\begin{align*}
     & x'_1=\frac{-\frac{5}{3}p^3 - \frac{10}{3}p^2 + \frac{97}{3}p - 14}{p^3 - p},\,\ x'_2=\frac{\frac{4}{3}p^2 - \frac{8}{3}p - 2}{p^2 - p},\,\
    x'_3=\frac{-\frac{2}{3}p^3 - \frac{1}{3}p^2 + \frac{31}{3}p - 14}{p^2 + p},\\
   &  x'_4=\frac{\frac{1}{3}p^2 - \frac{2}{3}p - 2}{p},\,\
    x'_5=\frac{-\frac{1}{6}p^2 - \frac{1}{6}p + 2}{p - 1},\,\
     x'_6=\frac{-\frac{1}{3}p^2 - \frac{1}{3}p + 4}{p - 1},\,\ x'_7=0.
\end{align*}
\end{lem}
\begin{proof}
  The proof is analogous to Proposition \ref{case1Vmpower3}.
\end{proof}
\begin{prop}
With the above notations for $p \equiv 11 \pmod{12}$, we have
\begin{align*}
     &\frac{1}{g_{p^3}-1}\left(g_{p^3}\langle V_0, V_{\infty}\rangle-\frac{V^2_0+V^2_\infty}{2}\right)= 3g_{p^3}\log (p)+o (g_{p^3}\log p)\,\ \text{as}\,\ p\to \infty.
\end{align*}
\end{prop}
\begin{proof}
  From Example \ref{genusp3}, we recall that
  \begin{align*}
      g_{p^3}-1=\frac{p(p+4)(p-3)}{12}.
  \end{align*}
  Now, note that 
  \begin{align*}
    V_0\cdot V_{\infty}=&  x_1x'_1C^2_{3,0}+x_2x'_2C^2_{0,3}+x_3x'_3C^2_{2,1}+x_4x'_4C^2_{1,2}+x_5x'_5E^2+x_6x'_6F_1^2\\&+(x_1x'_2+x_2x'_1)C_{3,0}C_{0,3}+(x_1x'_3+x_3x'_1)C_{3,0}C_{2,1}+(x_1x'_4+x_4x'_1)C_{3,0}C_{1,2}\\&
    +(x_1x'_5+x_5x'_1)C_{3,0}E+(x_1x'_6+x_6x'_1)C_{3,0}F_1+(x_2x'_3+x_3x'_2)C_{0,3}C_{2,1}\\&
    +(x_2x'_4+x_4x'_2)C_{0,3}C_{1,2}+(x_2x'_5+x_5x'_2)C_{0,3}E+(x_2x'_6+x_6x'_2)C_{0,3}F_1\\&
    +(x_3x'_4+x_4x'_3)C_{2,1}C_{1,2}+(x_3x'_5+x_5x'_3)C_{2,1}E+(x_3x'_6+x_6x'_3)C_{2,1}F_1\\&+(x_4x'_5+x_5x'_4)C_{1,2}E
    +(x_4x'_6+x_6x'_4)C_{1,2}F_1+(x_5x'_6+x_6x'_5)EF_1.
  \end{align*}
Using Lemma \ref{case4Vmpower3}, we have
\begin{align*}
    V_0\cdot V_{\infty}=& \left(\frac{-\frac{25}{9}p^{5}+\cdots}{p^{5}+\cdots}\right)C^2_{3,0}+\left(\frac{-\frac{16}{9}p^5+\cdots}{p^5+\cdots}\right)C^2_{0,3}+\left(\frac{-\frac{4}{9}p^{5}+\cdots}{p^3+\cdots}\right)C^2_{2,1}+\left(\frac{-\frac{1}{9}p^{5}+\cdots}{p^3+\cdots}\right) C^2_{1,2}\\&+\left(\frac{-\frac{20}{9}p^{5}+\cdots}{p^4+\cdots}\right)C_{3,0}C_{2,1}+\left(\frac{-\frac{8}{9}p^{5}+\cdots}{p^4+\cdots}\right)C_{0,3}C_{1,2}+\left(\frac{\frac{4}{9}p^{6}+\cdots}{p^4+\cdots}\right)C_{2,1}C_{1,2}+\cdots
\end{align*}
Now, using Proposition \ref{pmod11imp}, we get
\begin{align*}
     V_0\cdot V_{\infty}=&\left(\frac{-\frac{25}{9}p^{5}+\cdots}{p^{5}+\cdots}\right)\left(\frac{-p^{3}+\cdots}{12}\right)+\left(\frac{-\frac{16}{9}p^5+\cdots}{p^5+\cdots}\right)\left(\frac{-p^{3}+\cdots}{12}\right)\\&+\left(\frac{-\frac{4}{9}p^{5}+\cdots}{p^3+\cdots}\right)\left(\frac{-p+\cdots}{6}\right)+\left(\frac{-\frac{1}{9}p^{5}+\cdots}{p^3+\cdots}\right) \left(\frac{-p+\cdots}{6}\right)\\&+\left(\frac{-\frac{20}{9}p^{5}+\cdots}{p^4+\cdots}\right)\left(\frac{p^2+\cdots}{12}\right)+\left(\frac{-\frac{8}{9}p^{5}+\cdots}{p^4+\cdots}\right)\left(\frac{p^2+\cdots}{12}\right)\\&+\left(\frac{\frac{4}{9}p^{6}+\cdots}{p^4+\cdots}\right)\left(\frac{p+\cdots}{12}\right)+\cdots
\end{align*}
Now, using the expression of genus, we derive
 \begin{align*}
    g_{p^3}\cdot \frac{\langle V_0, V_{\infty}\rangle}{ g_{p^3}-1}& =g_{p^3}\cdot \frac{ V_0\cdot V_{\infty}\log p}{g_{p^3}-1}=3g_{p^3}\log (p)  + o(\log p) \,\ \text{as}\,\ p\to \infty.
 \end{align*}
 From Lemma \ref{case4Vmpower3} and from Proposition \ref{pmod11imp}, we also derive
 the estimate 
 \begin{align*}
     \frac{V^2_0+V^2_\infty}{2(g_{p^3}-1)}=o (g_{p^3}\log p)\,\ \text{as}\,\ p\to \infty.
 \end{align*}
 This completes the proof.
\end{proof}

\subsection{For the modular curve \texorpdfstring{$\cX_0(p^4)$}{X}}\label{Construction_with_sage_power4}
Let $H_0$ and $H_{\infty}$ be the 
sections of $\cX_0(p^4)/\zz$ corresponding to the cusps $0, \infty \in X_0(p^4)(\qq)$. The horizontal divisors
$H_0$ and $H_{\infty}$ intersect exactly one of the curves with multiplicity one in the special fiber at an $\ff_p$ rational point transversally 
(cf.  Liu~\cite[Chapter~9, Proposition~1.30, and Corollary~1.32]{MR1917232}). Without loss of generality we assume that $H_0$ intersects $C_{0,4}'$, and $H_{\infty}$ intersects $C_{4,0}'$. It follows from the cusp and component labelling of Katz and Mazur~\cite[p.~296]{MR772569} that the components $C_0$ and $C_{\infty}$ intersect in a single point.

Let $K_{\tcX_0(p^4)}$ be the canonical divisor of Edixhoven's regular model $\tcX_0(p^4)$, and let $ K_{\cX_0(p^4)}$ be the canonical divisor of the minimal regular model $\tcX_0(p^4)$ after the three successive blow downs. Note that, in this section, while computing $V_0\cdot V_\infty$ we write $``\cdots"$, these are some insignificant terms involving some powers of the prime $p$.
\subsection{Case \texorpdfstring{$p \equiv 1 \pmod{12}$}{}}In this case, we prove the following results.
\begin{lem}\label{pullbackCanonicalDiv}
Let $\pi: \tcX_0(p^4) \to \cX_0(p^4)$ be the morphism which contracts $C_{2,2}, E_2$ and $F_2$. Then
\begin{align*}
    \pi^*K_{\cX_0(p^4)}=K_{\tcX_0(p^4)}-4 C_{2,2}-2 E_2-F_2.
\end{align*}
\end{lem}
\begin{proof}
  Let $\pi^*K_{\cX_0(p^4)}=K_{\tcX_0(p^4)}+a C_{2,2}+b E_2+cF_2$ for some integer $a,b$ and $c$. Then Liu~\cite[Chapter~9, Theorem~2.12 (a)]{MR1917232} implies 
  \begin{align*}
      C_{2,2}\cdot \pi^*K_{\cX_0(p^4)}=0,\,\ E_2\cdot \pi^*K_{\cX_0(p^4)}=0,\,\ F_2\cdot \pi^*K_{\cX_0(p^4)}=0.
  \end{align*}
  
Then using Proposition \ref{prop1mod12power4} and using the adjunction formula from Liu~\cite[Chapter~9, Theorem~1.37]{MR1917232}, we get the following linear equations 
 \begin{align*}
       \begin{bmatrix}
-1 & 1 & 1\\
1 & -2 & 0\\
1 & 0 & -3
\end{bmatrix}\begin{bmatrix}
\alpha \\
\beta\\
\gamma
\end{bmatrix}=\begin{bmatrix}
1\\
0\\
-1
\end{bmatrix}
  \end{align*}
  and by solving them we get
 $a=-4, b=-2$ and $c=-1$.
\end{proof}
\begin{lem}\label{Vmcase1power4}
For $p \equiv 1 \pmod{12}$,
consider the vertical divisors 
\begin{align*}
V_0= x_1C_{4,0}'+x_2C_{0,4}'+x_3C_{3,1}'+x_4C_{1,3}'+x_5E_{1}'+x_6E_{3}'+x_7F_{1}'+x_8F_{3}';\\
V_\infty= x'_1C_{4,0}'+x'_2C_{0,4}'+x'_3C_{3,1}'+x'_4C_{1,3}'+x'_5E_{1}'+x'_6E_{3}'+x'_7F_{1}'+x'_8F_{3}'.
\end{align*}
Then the divisors 
  \begin{equation*}
    D_m = K_{\cX_0(p^4)} - (2 g_{p^4} - 2) H_m + V_m, \qquad m\in \{0, \infty\}
  \end{equation*}
  are orthogonal to all vertical divisors of $\cX_0(p^4)$ with respect to the Arakelov intersection pairing, and $x_i, x'_i$ $(i=1, \ldots 8)$ are given as follows:
\begin{align*}
    &x_1=\frac{3p^5 - 2p^4 - 17p^3 - 30p + 28}{p^5 - p^3},\,\ x_2= \frac{-p^4 - 4p^3 + 9p^2 + 12p + 14}{p^4 - p^2},\\
    &x_3=\frac{2p^5 - p^4 - 15p^3 - 16p + 28}{p^4 + p^3},\,\ x_4=-1,\,\
    x_5=\frac{p^5 - \frac{1}{2}p^4 - \frac{15}{2}p^3 - 8p + 14}{p^4 + p^3},\\ &x_6=-\frac{1}{2},\,\
    x_7=\frac{\frac{2}{3}p^5 - \frac{14}{3}p^3 - \frac{16}{3}p + \frac{28}{3}}{p^4 + p^3},\,\ x_8=0
\end{align*}
and
\begin{align*}
    &x'_1=\frac{-3p^5 - 4p^4 + 23p^3 + 12p^2 + 30p - 28}{p^5 - p^3},\,\ x'_2=\frac{p^4 - 2p^3 - 3p^2 - 14}{p^4 - p^2},\\
    &x'_3=\frac{-2p^5 - p^4 + 13p^3 + 16p - 28}{p^4 + p^3},\,\ x'_4=-1,\,\
    x'_5=\frac{-p^5 - \frac{1}{2}p^4 + \frac{13}{2}p^3 + 8p - 14}{p^4 + p^3},\\
    &x'_6=-\frac{1}{2},\,\
    x'_7=\frac{-\frac{2}{3}p^5 + \frac{14}{3}p^3 + \frac{16}{3}p - \frac{28}{3}}{p^4 + p^3},\,\
    x'_8=0.
\end{align*}
\end{lem}
\begin{proof}
  The proof follows the  lines of Lemma \ref{case1Vmpower3}.
  To to determine $V_0$ note that for any prime vertical divisor $W$ supported on the 
  special fiber we must have $(K_{\cX_0(p^4)} - (2 g_{p^4} - 2) H_0 + V_0)\cdot W = 0$. 
  This yields
  \begin{equation*}
    V_0\cdot W = - K_{\cX_0(p^4)}\cdot W + (2 g_{p^4} - 2) H_0 \cdot W,
  \end{equation*}
  where $W \in \{ C_{4,0}', C_{0,4}', C_{3,1}', C_{1,3}', E_1', E_3', F_1', F_3'\}$. 

  From Liu~\cite[Chapter~9, Theorem~2.12 (c)]{MR1917232}, we use the following identity
  \begin{align*}
    K_{\cX_0(p^4)}\cdot  C'_{a,b}=\pi^*K_{\cX_0(p^4)}\cdot  \pi^*C'_{a,b}.
  \end{align*}
  Recalling that 
  $\displaystyle \pi^* C'_{a,b}=C_{a,b}+\frac{p-1}{2}C_{2,2}+\frac{p-1}{4}E_2+\frac{p-1}{6}F_2$,
  and by using Lemma \ref{pullbackCanonicalDiv}, we compute
  \begin{align*}
     \pi^*K_{\cX_0(p^4)}\cdot \pi^* C'_{a,b}=
     \left(K_{\tcX_0(p^4)}-4 C_{2,2}-2 E_2-F_2\right)  \cdot
     \left(C_{a,b}+\frac{p-1}{2}C_{2,2}+\frac{p-1}{4}E_2+\frac{p-1}{6}F_2\right).
  \end{align*}
  Now using Proposition \ref{prop1mod12power4} and using the adjunction formula 
  from Liu~\cite[Chapter~9, Theorem~1.37]{MR1917232}, we get the following intersections
  \begin{equation*}
    \begin{array}{llr}
      K_{\cX_0(p^4)} \cdot W & = \dfrac{p^4-p^3-4p-20}{12}, & W \in \{ C_{4,0}', C_{0,4}' \}; \\
      K_{\cX_0(p^4)} \cdot Z & = \dfrac{p^2-2p-5}{6},       & Z \in \{ C_{3,1}', C_{1,3}' \}; 
      \end{array}
       \end{equation*}
      \begin{equation*}
      \begin{array}{llr}
      K_{\cX_0(p^4)} \cdot E_i' & = 0,                      & i = 1, 3; \\
      K_{\cX_0(p^4)} \cdot F'_i & = 1,                      & i = 1, 3.
    \end{array}
  \end{equation*}
  Finally noting that $H_0$ only intersects $C_{0,4}'$ transversally and no other component
  of the special fiber we get the equations
  \begin{equation} \label{vdotC'}
    V_0 \cdot Z=
      \begin{cases}
         -K_{\cX_0(p^4)}\cdot  C'_{0,4}+ (2 g_{p^4} - 2) & Z = C_{0,4}', \\
         -K_{\cX_0(p^4)}\cdot C'_{a,b} & Z \in \{ C_{4,0}', C_{3,1}', C_{1,3}', E_1', E_3', 
           F_1', F_3'\}.
      \end{cases}
  \end{equation}
  To obtain $V_0$ we use SageMath \cite{sage} to solving the linear equations \eqref{vdotC'}.
  
  Similarly, $V_\infty$ is obtained by solving
  \begin{equation*}
    V_0 \cdot Z=
      \begin{cases}
         -K_{\cX_0(p^4)}\cdot  C'_{4,0}+ (2 g_{p^4} - 2) & Z = C_{4,0}', \\
         -K_{\cX_0(p^4)}\cdot C'_{a,b} & Z \in \{ C_{0,4}', C_{3,1}', C_{1,3}', E_1', E_3', 
           F_1', F_3'\}.
      \end{cases}
  \end{equation*}
  Finally for any prime $q \in \zz$ other than $p$ the fiber $V = \cX_0(p^4)_{\ff_q}$ is 
  irreducible and $V \cdot V_m = 0$ since $V_m$ is supported on the fiber over $(p) \in \spec \zz$.
  This yields $D_m \cdot V = 0$.
\end{proof}
\begin{prop}\label{gpartcase1power4}
With the above notations for $p \equiv 1 \pmod{12}$, we have
\begin{align*}
     &\frac{1}{g_{p^4}-1}\left(g_{p^4}\langle V_0, V_{\infty}\rangle-\frac{V^2_0+V^2_\infty}{2}\right)= 4g_{p^4}\log (p)+o (g_{p^4}\log p)\,\ \text{as}\,\ p\to \infty.
\end{align*}
\end{prop}
\begin{proof}
From Example \ref{genusp3}, we recall that
  \begin{align*}
      g_{p^4}-1=\frac{p(p+1)(p^2-6)-14}{12}.
  \end{align*}
Now, note that
\begin{align*}
     V_0\cdot V_{\infty}&= x_1{x'}_1{C'}^2_{4,0}+x_2{x'}_2{C'}^2_{0,4}+x_3{x'}_3{C'}^2_{3,1}+x_4{x'}_4{C'}^2_{1,3}+x_5{x'}_5 {E'}_1^2+x_6{x'}_6{E'}_3^2\\&+x_7{x'}_7{F'}_1^2+x_8{x'}_8{F'}_3^2+(x_1{x'}_2+x_2{x'}_1){C'}_{4,0}{C'}_{0,4}+(x_1{x'}_3+x_3{x'}_1)C_{4,0}'C_{3,1}'\\
    &+(x_1x'_4+x_4x'_1)C_{4,0}'C_{1,3}'+(x_1{x'}_5+x_5{x'}_1)C_{4,0}'{E'}_1+(x_1{x'}_6+x_6{x'}_1)C_{4,0}'{E'}_3\\
    &+(x_1{x'}_7+x_7{x'}_1)C_{4,0}'{F'}_1+(x_1{x'}_8+x_8{x'}_1)C_{4,0}'{F'}_3+(x_2{x'}_3+x_3{x'}_2)C_{0,4}'C_{3,1}'
    \\
    &+(x_2{x'}_4+x_4{x'}_2)C_{0,4}'C_{1,3}'+(x_2{x'}_5+x_5{x'}_2)C_{0,4}'{E'}_1+(x_2{x'}_6+x_6{x'}_2)C_{0,4}'{E'}_3\\&+(x_2{x'}_7+x_7{x'}_2)C_{0,4}'{F'}_1+(x_2{x'}_8+x_8{x'}_2)C_{0,4}'{F'}_3+(x_3{x'}_4+x_4{x'}_3)C_{3,1}'C_{1,3}'
    \\&+(x_3{x'}_5+x_5{x'}_3)C_{3,1}'{E'}_1+(x_3{x'}_6+x_6{x'}_3)C_{3,1}'{E'}_3'+(x_3{x'}_7+x_7{x'}_3)C_{3,1}'{F'}_1\\
    &+(x_3{x'}_8+x_8{x'}_3)C_{3,1}'{F'}_3+(x_4{x'}_5+x_5{x'}_4)C_{1,3}'{E'}_1+(x_4{x'}_6+x_6{x'}_4)C_{1,3}'{E'}_3\\&+(x_4{x'}_7+x_7{x'}_4)C_{1,3}'{F'}_1+(x_4{x'}_8+x_8{x'}_4)C_{1,3}'{F'}_3+(x_5{x'}_6+x_6{x'}_5){E'}_1{E'}_3\\&+(x_5{x'}_7+x_7{x'}_5){E'}_1{F'}_1+(x_5{x'}_8+{x'}_5){E'}_1{F'}_3+(x_6{x'}_7+x_7{x'}_6){E'}_3{F'}_1\\&+(x_6{x'}_8+x_8{x'}_6){E'}_3{F'}_3+(x_7{x'}_8+x_8{x'}_7){F'}_1{F'}_3.
\end{align*}
Using Lemma \ref{Vmcase1power4} we have 
\begin{align*}
 V_0\cdot V_{\infty}=&
     \left(\frac{-9p^{10}+\cdots}{p^{10}+\cdots}\right)C'^2_{4,0}+\left(\frac{-p^8+\cdots}{p^8+\cdots}\right)C'^2_{0,4}+\left(\frac{-4p^{10}+\cdots}{p^8+\cdots}\right)C'^2_{3,1}\\
     &+\left(\frac{-12p^{10}+\cdots}{p^9+\cdots}\right)C_{4,0}'C_{3,1}'+\cdots
\end{align*}
Now, using Proposition \ref{minmatrixcase1power4}, we have
\begin{align*}
     V_0\cdot V_{\infty}=&\left(\frac{-9p^{10}+\cdots}{p^{10}+\cdots}\right)\left(\frac{-p^{4}+\cdots}{12}\right)
    +\left(\frac{-p^8+\cdots}{p^8+\cdots}\right)\left(\frac{-p^{4}+\cdots}{12}\right)\\&+\left(\frac{-4p^{10}+\cdots}{p^8+\cdots}\right)\left(\frac{-p^{2}+\cdots}{8}\right)+\left(\frac{-12p^{10}+\cdots}{p^9+\cdots}\right)\left(\frac{p^{3}+\cdots}{12}\right)+\cdots
\end{align*}
 Using the expression of genus, we get the asymptotics
 \begin{align*}
     g_{p^4} \cdot \frac{\langle V_0, V_{\infty}\rangle}{ g_{p^4}-1}& =g_{p^4} \cdot \frac{V_0\cdot V_{\infty}\log p}{g_{p^4}-1}=4g_{p^4}\log (p)  + o(\log p) \,\ \text{as}\,\ p\to \infty.
 \end{align*}
 From Lemma \ref{Vmcase1power4} and from Proposition \ref{minmatrixcase1power4}, we also derive the estimate 
 \begin{align*}
     \frac{V^2_0+V^2_\infty}{2(g_{p^4}-1)}=o (g_{p^4}\log p)\,\ \text{as}\,\ p\to \infty.
 \end{align*}
 This completes the proof.
\end{proof}
\subsection{Case \texorpdfstring{$p \equiv 5 \pmod{12}$}{}}In this case, we prove the following results.
\begin{lem}\label{pullbackCanonicalDivcase2}
Let $\pi: \tcX_0(p^4) \to \cX_0(p^4)$ be the morphism which contracts $C_{2,2}, E_2$ and $F$. Then
\begin{align*}
    \pi^*K_{\cX_0(p^4)}=K_{\tcX_0(p^4)}-4 C_{2,2}-2 E_2-F.
\end{align*}
\end{lem}
\begin{proof}
  The proof is same as in Lemma \ref{pullbackCanonicalDiv}, only we need to replace $F_2$ by $F$.
\end{proof}
\begin{lem}\label{Vmcase2power4}
For $p \equiv 5 \pmod{12}$,
consider the vertical divisors 
\begin{align*}
V_0= x_1C_{4,0}'+x_2C_{0,4}'+x_3C_{3,1}'+x_4C_{1,3}'+x_5E_{1}'+x_6E_{3}';\\
V_\infty= x'_1C_{4,0}'+x'_2C_{0,4}'+x'_3C_{3,1}'+x'_4C_{1,3}'+x'_5E_{1}'+x'_6E_{3}'.
\end{align*}
Then the divisors 
  \begin{equation*}
    D_m = K_{\cX_0(p^4)} - (2 g_{p^4} - 2) H_m + V_m, \qquad m\in \{0, \infty\}
  \end{equation*}
  are orthogonal to all vertical divisors of $\cX_0(p^4)$ with respect to the Arakelov intersection pairing, and $x_i, x'_i$ $(i=1, \ldots, 6)$ are given as follows.
\begin{align*}
    &x_1=\frac{3p^5 - p^4 - 16p^3 - 6p + 12}{p^5 - p^3},\,\ x_2= \frac{-p^4 - 3p^3 + 10p^2 + 12p + 6}{p^4 - p^2},\\
    &x_3=\frac{2p^5 - 14p^3 + 12}{p^4 + p^3},\,\ x_4=0,\,\
    x_5=\frac{p^5 - 7p^3 + 6}{p^4 + p^3},\,\ x_6=0
\end{align*}
and
\begin{align*}
    &x'_1=0,\,\ x'_2=\frac{4p^5 + 2p^4 - 26p^3 - 12p^2 - 12p + 12}{p^5 - p^3},\\
    &x'_3=\frac{p^4 + 3p^3 - 10p^2 - 12p - 6}{p^3 + p^2},\,\ x'_4=\frac{3p^5 + 3p^4 - 24p^3 - 12p^2 - 6p + 12}{p^4 + p^3},\\
    &x'_5=\frac{\frac{1}{2}p^4 + \frac{3}{2}p^3 - 5p^2 - 6p - 3}{p^3 + p^2},\,\
    x'_6=\frac{\frac{3}{2}p^5 + \frac{3}{2}p^4 - 12p^3 - 6p^2 - 3*p + 6}{p^4 + p^3}.
\end{align*}
\end{lem}

\begin{proof}
  The proof is similar as in Lemma \ref{Vmcase1power4}, but in this case we have 
  $\displaystyle \pi^* C'_{a,b}=C_{a,b}+\frac{p-1}{2}C_{2,2}+\frac{p-1}{4}E_2+\frac{p+1}{6}F$ 
  and by using Lemma \ref{pullbackCanonicalDivcase2}, we compute 
  $\pi^*K_{\cX_0(p^4)}\cdot \pi^* C'_{a,b}$.
  Then using Proposition \ref{prop1mod12power4} and using the adjunction formula from
  Liu~\cite[Chapter~9, Theorem~1.37]{MR1917232}, we get the following equations
  \begin{equation*}
    \begin{array}{llr}
      K_{\cX_0(p^4)} \cdot W & =\dfrac{p^4-p^3-4p-12}{12}, & W \in \{ C'_{4,0}, C'_{0,4} \}; \\
      K_{\cX_0(p^4)} \cdot Z & =\dfrac{p^2-2p-3}{6},       & Z \in \{ C'_{3,1}, C'_{1,3} \}; \\ 
      K_{\cX_0(p^4)}\cdot E'_{i} & = 0,                    & i = 1, 3.
    \end{array}
  \end{equation*}
 Then we follow the same strategy which we use in Lemma \ref{Vmcase1power4}.
\end{proof}
\begin{prop}\label{gpartcase2power4}
With the above notations for $p \equiv 5 \pmod{12}$, we have
\begin{align*}
     &\frac{1}{g_{p^4}-1}\left(g_{p^4}\langle V_0, V_{\infty}\rangle-\frac{V^2_0+V^2_\infty}{2}\right)= 4g_{p^4}\log (p)+o (g_{p^4}\log p)\,\ \text{as}\,\ p\to \infty.
\end{align*}
\end{prop}
\begin{proof}
From Example \ref{genusp3}, we recall that
  \begin{align*}
      g_{p^4}-1=\frac{p(p+1)(p^2-6)-6}{12}.
  \end{align*}
Now, note that
\begin{align*}
    V_0\cdot V_{\infty}= & \ x_1{x'}_1{C'}^2_{4,0}+x_2{x'}_2{C'}^2_{0,4}+x_3{x'}_3{C'}^2_{3,1}+x_4{x'}_4{C'}^2_{1,3}+x_5{x'}_5{E'}_1^2+x_6{x'}_6{E'}_3^2\\
    &+(x_1{x'}_2+x_2{x'}_1)C_{4,0}'C_{0,4}'+(x_1{x'}_3+x_3{x'}_1)C_{4,0}'C_{3,1}'+(x_1{x'}_4+x_4{x'}_1)C_{4,0}'C_{1,3}'\\
    &+(x_1{x'}_5+x_5{x'}_1)C_{4,0}'{E'}_1+(x_1{x'}_6+x_6{x'}_1)C_{4,0}'{E'}_3+(x_2{x'}_3+x_3{x'}_2)C_{0,4}'C_{3,1}'
    \\
    &+(x_2{x'}_4+x_4{x'}_2)C_{0,4}'C_{1,3}'+(x_2{x'}_5+x_5{x'}_2)C_{0,4}'{E'}_1+(x_2{x'}_6+x_6{x'}_2)C_{0,4}'{E'}_3\\
    &+(x_3{x'}_4+x_4{x'}_3)C_{3,1}'C_{1,3}'+(x_3{x'}_5+x_5{x'}_3)C_{3,1}'{E'}_1+(x_3{x'}_6+x_6{x'}_3)C_{3,1}'{E'}_3\\&+(x_4{x'}_5+x_5{x'}_4)C_{1,3}'{E'}_1+(x_4{x'}_6+x_6{x'}_4)C_{1,3}'{E'}_3+(x_5{x'}_6+x_6{x'}_5){E'}_1{E'}_3.
\end{align*}
Using Lemma \ref{Vmcase2power4}, we have
\begin{align*}
     V_0\cdot V_{\infty}&=\left(\frac{-4p^9+\cdots}{p^9+\cdots}\right)C'^2_{0,4}+\left(\frac{2p^{9}+\cdots}{p^8+\cdots}\right)C'^2_{3,1}+\left(\frac{1/2p^{9}+\cdots}{p^7+\cdots}\right)E'^2_{1}\\&+\left(\frac{7p^{10}+\cdots}{p^9+\cdots}\right)C_{0,4}'C_{3,1}'+\left(\frac{7/2p^{10}+\cdots}{p^9+\cdots}\right)C_{0,4}'E_{1}'+\left(\frac{2p^{9}+\cdots}{p^7+\cdots}\right)C_{3,1}'E'_1.
     \end{align*}
  Now, using Proposition \ref{minmatrixcase2power4}, we get   
     \begin{align*}
     V_0\cdot V_{\infty}&=\left(\frac{-4p^9+\cdots}{p^9+\cdots}\right)\left(\frac{-p^4+\cdots}{12}\right)+\cdots
\end{align*}
Using the expression of genus, we get
 \begin{align*}
     g_{p^4} \cdot \frac{\langle V_0, V_{\infty}\rangle}{ g_{p^4}-1}& =g_{p^4} \cdot \frac{ V_0\cdot V_{\infty}\log p}{g_{p^4}-1}=4g_{p^4}\log (p)  + o(\log p) \,\ \text{as}\,\ p\to \infty.
 \end{align*}
 From Lemma \ref{Vmcase2power4} and from Proposition \ref{minmatrixcase2power4}, we also derive the estimate
 \begin{align*}
     \frac{V^2_0+V^2_\infty}{2(g_{p^4}-1)}=o (g_{p^4}\log p)\,\ \text{as}\,\ p\to \infty.
 \end{align*}
 This completes the proof.
\end{proof}

\subsection{Case \texorpdfstring{$p \equiv 7 \pmod{12}$}{}}In this case, we prove the following results.
\begin{lem}\label{pullbackCanonicalDivcase3}
Let $\pi: \tcX_0(p^4) \to \cX_0(p^4)$ be the morphism which contracts $C_{2,2}, E_2$ and $F$. Then
\begin{align*}
    \pi^*K_{\cX_0(p^4)}=K_{\tcX_0(p^4)}-4 C_{2,2}-2 E-F_2.
\end{align*}
\end{lem}
\begin{proof}
  The proof is same as in Lemma \ref{pullbackCanonicalDiv}, simply replace $E_2$ by $E$.
\end{proof}
\begin{lem}\label{Vmcase3power4}
For $p \equiv 7 \pmod{12}$,
consider the vertical divisors 
\begin{align*}
V_0= x_1C_{4,0}'+x_2C_{0,4}'+x_3C_{3,1}'+x_4C_{1,3}'+x_5F_{1}'+x_6F_{3}';\\
V_\infty= x'_1C_{4,0}'+x'_2C_{0,4}'+x'_3C_{3,1}'+x'_4C_{1,3}'+x'_5F_{1}'+x'_6F_{3}'.
\end{align*}
Then the divisors 
  \begin{equation*}
    D_m = K_{\cX_0(p^4)} - (2 g_{p^4} - 2) H_m + V_m, \qquad m\in \{0, \infty\}
  \end{equation*}
  are orthogonal to all vertical divisors of $\cX_0(p^4)$ with respect to the Arakelov intersection pairing, and $x_i, x'_i$ $(i=1, \ldots, 6)$ are given as follows.
\begin{align*}
    &x_1=\frac{3p^5 - 2p^4 - 17p^3 - 12p + 16}{p^5 - p^3},\,\
    x_2= \frac{-p^4 - 4p^3 + 9p^2 + 12p + 8}{p^4 - p^2},\\
    &x_3=\frac{2p^5 - p^4 - 15p^3 - 4p + 16}{p^4 + p^3},\,\ x_4=-1,\,\
    x_5=\frac{\frac{2}{3}p^5 - \frac{14}{3}p^3 - \frac{4}{3}p + \frac{16}{3}}{p^4 + p^3},\,\ x_6=0
\end{align*}
and
\begin{align*}
    &x'_1=\frac{-3p^5 - 4p^4 + 23p^3 + 12p^2 + 12p - 16}{p^5 - p^3},\,\
     x'_2=\frac{p^4 - 2p^3 - 3p^2 - 8}{p^4 - p^2},\\
    &x'_3=\frac{-2p^5 - p^4 + 13p^3 + 4p - 16}{p^4 + p^3},\,\  x'_4=-1,\,\
    x'_5=\frac{-\frac{2}{3}p^5 + \frac{14}{3}p^3 + \frac{4}{3}p - \frac{16}{3}}{p^4 + p^3},\,\ 
    x'_6=0.
\end{align*}
\end{lem}

\begin{proof}
  In this case we recall 
  $\displaystyle \pi^* C'_{a,b}=C_{a,b}+\frac{p-1}{2}C_{2,2}+\frac{p+1}{4}E+\frac{p-1}{6}F_2$ 
  and by using Lemma \ref{pullbackCanonicalDivcase3}, we compute 
  $\pi^*K_{\cX_0(p^4)}\cdot \pi^* C'_{a,b}$.  Then using Proposition \ref{prop1mod12power4} 
  and using the adjunction formula from Liu~\cite[Chapter~9, Theorem~1.37]{MR1917232}, 
  we get the following intersection numbers
  \begin{equation*}
    \begin{array}{llr} 
      K_{\cX_0(p^4)} \cdot W & = \dfrac{p^4-p^3-4p-14}{12}, & W \in \{ C'_{4,0}, C'_{0,4} \}; \\
      K_{\cX_0(p^4)} \cdot Z & = \dfrac{p^2-2p-5}{6},       & Z \in \{ C'_{3,1}, C'_{1,3} \}; \\
      K_{\cX_0(p^4)}\cdot F'_{i} & = 1,                     & i = 1, 3.
    \end{array}
  \end{equation*}
  Then by following the same strategy as before, we complete the proof.
\end{proof}

\begin{prop}\label{gpartcase3power4}
With the above notations for $p \equiv 7 \pmod{12}$, we have
\begin{align*}
     &\frac{1}{g_{p^4}-1}\left(g_{p^4}\langle V_0, V_{\infty}\rangle-\frac{V^2_0+V^2_\infty}{2}\right)= 4g_{p^4}\log (p)+o (g_{p^4}\log p)\,\ \text{as}\,\ p\to \infty.
\end{align*}
\end{prop}
\begin{proof}
  From Example \ref{genusp3}, we recall that
  \begin{align*}
      g_{p^4}-1=\frac{p(p+1)(p^2-6)-8}{12}.
  \end{align*}
 Now, note that
  \begin{align*}
      V_0\cdot V_{\infty}=& \ x_1{x'}_1{C'}^2_{4,0}+x_2{x'}_2{C'}^2_{0,4}+x_3{x'}_3{C'}^2_{3,1}+x_4{x'}_4{C'}^2_{1,3}+x_5{x'}_5{F'}_1^2+x_6{x'}_6{F'}_3^2\\
    &+(x_1{x'}_2+x_2{x'}_1)C_{4,0}'C_{0,4}'+(x_1{x'}_3+x_3{x'}_1)C_{4,0}'C_{3,1}'+(x_1{x'}_4+x_4{x'}_1)C_{4,0}'C_{1,3}'\\
    &+(x_1{x'}_5+x_5{x'}_1)C_{4,0}'F'_1+(x_1{x'}_6+x_6{x'}_1)C_{4,0}'F'_3+(x_2{x'}_3+x_3{x'}_2)C_{0,4}'C_{3,1}'
    \\
    &+(x_2{x'}_4+x_4{x'}_2)C_{0,4}'C_{1,3}'+(x_2{x'}_5+x_5{x'}_2)C_{0,4}'{F'}_1+(x_2{x'}_6+x_6{x'}_2)C_{0,4}'{F'}_3\\
    &+(x_3{x'}_4+x_4{x'}_3)C_{3,1}'C_{1,3}'+(x_3{x'}_5+x_5{x'}_3)C_{3,1}'{F'}_1+(x_3{x'}_6+x_6{x'}_3)C_{3,1}'{F'}_3\\&+(x_4{x'}_5+x_5{x'}_4)C_{1,3}'{F'}_1+(x_4{x'}_6+x_6{x'}_4)C_{1,3}'{F'}_3+(x_5{x'}_6+x_6{x'}_5){F'}_1{F'}_3.
  \end{align*}
   Using Lemma \ref{Vmcase3power4} we have
   \begin{align*}
     V_0\cdot V_{\infty}=&   \left(\frac{-9p^{10}+\cdots}{p^{10}+\cdots}\right)C'^2_{4,0}+\left(\frac{-p^8+\cdots}{p^8+\cdots}\right)C'^2_{0,4}+\left(\frac{-4p^{10}+\cdots}{p^8+\cdots}\right)C'^2_{3,1}\\&+\left(\frac{-12p^{10}+\cdots}{p^9+\cdots}\right)C_{4,0}'C_{3,1}'+\cdots
   \end{align*}
  Now, using Proposition \ref{minmatrixcase3power4}, we get
  \begin{align*}
       V_0\cdot V_{\infty}=&\left(\frac{-9p^{10}+\cdots}{p^{10}+\cdots}\right)\left(\frac{-p^4+\cdots}{12}\right)+\left(\frac{-p^8+\cdots}{p^8+\cdots}\right)\left(\frac{-p^4+\cdots}{12}\right)\\&+\left(\frac{-4p^{10}+\cdots}{p^8+\cdots}\right)\left(\frac{-p^2+\cdots}{8}\right)+\left(\frac{-12p^{10}+\cdots}{p^9+\cdots}\right)\left(\frac{p^3+\cdots}{12}\right)+\cdots
  \end{align*}
  Using the expression of genus, we get
 \begin{align*}
     g_{p^4} \cdot \frac{\langle V_0, V_{\infty}\rangle}{ g_{p^4}-1}& =g_{p^4} \cdot \frac{ V_0\cdot V_{\infty}\log p}{g_{p^4}-1}=4g_{p^4}\log (p)  + o(\log p) \,\ \text{as}\,\ p\to \infty.
 \end{align*}
 From Lemma \ref{Vmcase3power4} and from Proposition \ref{minmatrixcase3power4}, we also derive the estimate
 \begin{align*}
     \frac{V^2_0+V^2_\infty}{2(g_{p^4}-1)}=o (g_{p^4}\log p)\,\ \text{as}\,\ p\to \infty.
 \end{align*}
 This completes the proof.
\end{proof}
\subsection{Case \texorpdfstring{$p \equiv 11 \pmod{12}$}{}}In this case, we prove the following results.
\begin{lem}\label{pullbackCanonicalDivcase4}
 Let $\pi: \tcX_0(p^4) \to \cX_0(p^4)$ be the morphism which contracts $C_{2,2}, E_2$ and $F$. Then
\begin{align*}
    \pi^*K_{\cX_0(p^4)}=K_{\tcX_0(p^4)}-4 C_{2,2}-2 E-F.
\end{align*}
\end{lem}
\begin{proof}
  The proof is same as in Lemma \ref{pullbackCanonicalDiv}, here we replace $E_2$ by $E$, and replace $F_2$ by $F$.
\end{proof}
\begin{lem}\label{Vmcase4power4}
For $p \equiv 11 \pmod{12}$,
consider the vertical divisors 
\begin{align*}
V_0= x_1C_{4,0}'+x_2C_{0,4}'+x_3C_{3,1}'+x_4C_{1,3}';\\
V_\infty= x'_1C_{4,0}'+x'_2C_{0,4}'+x'_3C_{3,1}'+x'_4C_{1,3}'.
\end{align*}
Then the divisors 
  \begin{equation*}
    D_m = K_{\cX_0(p^4)} - (2 g_{p^4} - 2) H_m + V_m, \qquad m\in \{0, \infty\}
  \end{equation*}
  are orthogonal to all vertical divisors of $\cX_0(p^4)$ with respect to the Arakelov intersection pairing, and $x_i, x'_i$ $(i=1,\ldots, 4)$ are given as follows.
\begin{align*}
    &x_1=\frac{3p^3 - 4p^2 - 12p + 12}{p^3 - p^2},\,\ x_2= \frac{-p^2 - 2p + 12}{p^2 - p},\,\
    x_3=\frac{2p^3 - 2p^2 - 12p + 12}{p^2},\,\ x_4=0
\end{align*}
and
\begin{align*}
    &x'_1=\frac{-3p^3 + 24p - 12}{p^3 - p^2},\,\
     x'_2=\frac{p - 2}{p - 1},\,\
    x'_3=\frac{-2p^3 + 2p^2 + 12p - 12}{p^2},\,\  x'_4=0.
\end{align*}
\end{lem}

\begin{proof}
  In this case 
  $\displaystyle \pi^* C'_{a,b}=C_{a,b}+\frac{p-1}{2}C_{2,2}+\frac{p+1}{4}E+\frac{p+1}{6}F$ 
  and by using Lemma \ref{pullbackCanonicalDivcase4}, we compute 
  $\pi^*K_{\cX_0(p^4)}\cdot \pi^* C'_{a,b}$. Then using Proposition \ref{prop1mod12power4} and 
  the adjunction formula from Liu~\cite[Chapter~9, Theorem~1.37]{MR1917232}, we
  get the following intersection numbers
  \begin{align*}
    K_{\cX_0(p^4)} \cdot C'_{4,0} & = \frac{p^4-p^3-4p-6}{12}, &
    K_{\cX_0(p^4)} \cdot C'_{0,4} & = \frac{p^4-p^3-4p-6}{12}, \\
    K_{\cX_0(p^4)} \cdot C'_{3,1} & = \frac{p^2-2p-3}{6},      &
    K_{\cX_0(p^4)} \cdot C'_{1,3} & = \frac{p^2-2p-3}{6}.
  \end{align*}
  Then by following the same strategy as before, we complete the proof.
\end{proof}

\begin{prop}\label{gpartcase4power4}
With the above notations for $p \equiv 11 \pmod{12}$, we have
\begin{align*}
     &\frac{1}{g_{p^4}-1}\left(g_{p^4}\langle V_0, V_{\infty}\rangle-\frac{V^2_0+V^2_\infty}{2}\right)= 4g_{p^4}\log (p)+o (g_{p^4}\log p)\,\ \text{as}\,\ p\to \infty.
\end{align*}
\end{prop}
\begin{proof}
  From  Example \ref{genusp3}, we recall that
  \begin{align*}
      g_{p^4}-1=\frac{p(p+1)(p^2-6)}{12}.
  \end{align*}
  Now, note that
  \begin{align*}
      V_0\cdot V_{\infty}=& x_1{x'_1}{C'}^2_{4,0}+x_2{x'}_2{C'}^2_{0,4}+x_3{x'}_3{C'}^2_{3,1}+x_4{x'}_4{C'}^2_{1,3}+(x_1{x'}_2+x_2{x'}_1)C_{4,0}'C_{0,4}'\\
    &+(x_1{x'}_3+x_3{x'}_1)C_{4,0}'C_{3,1}'+(x_1{x'}_4+x_4{x'}_1)C_{4,0}'C_{1,3}'+(x_2{x'}_3+x_3{x'}_2)C_{0,4}'C_{3,1}'\\
    &+(x_2{x'}_4+x_4{x'}_2)C_{0,4}'C_{1,3}'+(x_3{x'}_4+x_4{x'}_3)C_{3,1}'C_{1,3}'.
  \end{align*}
  Using Lemma \ref{Vmcase4power4} we have
  \begin{align*}
    V_0\cdot V_{\infty}=& \left(\frac{-9p^6  +\cdots}{p^6  +\cdots}\right)C'^2_{4,0}+\left(\frac{-p^3 +\cdots}{p^3 +\cdots}\right)C'^2_{0,4}+\left(\frac{-4p^6  +\cdots}{p^4}\right)C'^2_{3,1}\\
    &+\left(\frac{-12p^6 +\cdots}{+p^5\cdots}\right)C_{4,0}'C_{3,1}'+\cdots
  \end{align*}
Then, using Proposition \ref{minmatrixcase4power4}, we get
\begin{align*}
  V_0\cdot V_{\infty}=& \left(\frac{-9p^6  +\cdots}{p^6 +\cdots}\right)\left(\frac{-p^4+\cdots}{12}\right)+\left(\frac{-p^3 +\cdots}{p^3 +\cdots}\right)\left(\frac{-p^4+\cdots}{12}\right)\\
    &+\left(\frac{-4p^6  +\cdots}{p^4}\right)\left(\frac{-p^2+\cdots}{8}\right)
    +\left(\frac{-12p^6 +\cdots}{p^5+\cdots}\right)\left(\frac{p^3+\cdots}{12}\right)+\cdots
\end{align*}
Now, using the expression of genus, we derive
 \begin{align*}
     g_{p^4} \cdot \frac{\langle V_0, V_{\infty}\rangle}{ g_{p^4}-1}& =g_{p^4} \cdot \frac{ V_0\cdot V_{\infty}\log p}{g_{p^4}-1}=4g_{p^4}\log (p)  + o(\log p) \,\ \text{as}\,\ p\to \infty.
 \end{align*}
 From Lemma \ref{Vmcase4power4} and from Proposition \ref{minmatrixcase4power4}, we also derive the estimate 
 \begin{align*}
     \frac{V^2_0+V^2_\infty}{2(g_{p^4}-1)}=o (g_{p^4}\log p)\,\ \text{as}\,\ p\to \infty.
 \end{align*}
 This completes the proof.
\end{proof}
\section{The Arakelov self-intersection number \texorpdfstring{$\overline{\omega}^2_{\cX_0(p^r)}$}{} where \texorpdfstring{$r=3,4$}{}}\label{intersection_number}
We continue with the notation from the previous Section. Let $H_0$ and $H_{\infty}$ be the 
sections of $\cX_0(p^r)/\zz$ ($r=3,4$) corresponding to the cusps $0, \infty \in X_0(p^r)(\qq)$. 
Without loss of generality, we assume that $H_0$ intersects the component $C_{0,r}$ and 
$H_{\infty}$ meets the component $C_{r,0}$ of the special fiber. 

Recall that, for $r=3,4$ the divisors 
  \begin{equation*}
    D_m = K_{\cX_0(p^r)} - (2 g_{p^r} - 2) H_m + V_m, \qquad m\in \{0, \infty\}
  \end{equation*}
  are orthogonal to all vertical divisors of the minimal regular model $\cX_0(p^r)$ with respect to the Arakelov intersection
  pairing.

\begin{prop}\label{Propomega}
With the above notations, we have the following equality of the Arakelov self-intersection number of the relative dualizing sheaf:
  \begin{align*}
     \overline{\omega}^2_{\cX_0(p^r)}=-4g_{p^r}(g_{p^r}-1)  \langle H_0, H_{\infty}\rangle+\frac{1}{g_{p^r}-1}\left(g_{p^r}\langle V_0, V_{\infty}\rangle-\frac{V^2_0+V^2_\infty}{2}\right)+h. 
  \end{align*}
  In the above equality $h=\dfrac{h_0+h_\infty}{2}<0$, where $h_m=-2\, \text{N\'{e}ron-Tate height of} \, \cO(D_m)$ with $m\in\{0, \infty\}$.
\end{prop}
\begin{proof}
   From a theorem of Faltings-Hriljac~\cite[Theorem 4]{MR740897}, we know
  \begin{align*}
      \langle D_m, D_m \rangle =-2 \Big( \text{N\'{e}ron-Tate height of } \cO(D_m) \Big)
    :=h_m.
  \end{align*}
  This implies
  \begin{align*}
      \langle D_m,\,\ K_{\cX_0(p^r)} - (2 g_{p^r} - 2) H_m + V_m \rangle =h_m.
  \end{align*}
 Using the fact $\langle D_m, V_m \rangle = 0$, we get
 \begin{align*}
     \left \langle K_{\cX_0(p^r)} - (2 g_{p^r} - 2) H_m + V_m, K_{\cX_0(p^r)} - (2 g_{p^r} - 2) H_m \right\rangle =h_m.
  \end{align*}
  This yields
  \begin{align}\label{Omega1}
     \overline{\omega}^2_{\cX_0(p^r)} =& 
    -(2g_{p^r}-2)^2 H_m^2 + 2(2g_{p^r}-2) \langle K_{\cX_0(p^r)}, H_m\rangle - 
               \langle K_{\cX_0(p^r)}, V_m \rangle \notag \\&+
               (2g_{p^r}-2) \langle H_m, V_m\rangle+h_m.
  \end{align}
  Again $\langle D_m, V_m \rangle = 0$ implies 
  \begin{align}\label{eq1}
      \left\langle K_{\cX_0(p^r)}, V_m \rangle
  - (2g_{p^r}-2) \langle H_m, V_m\right \rangle + V_m^2 = 0.
  \end{align}
  From Lang \cite[Ch. IV, Sec. 5, Corollary 5.6]{Lang_Book}), we have the adjunction formula 
  \begin{align}\label{adj}
      \left\langle K_{\cX_0(p^r)}, H_m 
 \right \rangle = - H_m^2.
  \end{align}
 Using \eqref{eq1} and \eqref{adj} in the formula \eqref{Omega1}, we get
 \begin{align*}
     \overline{\omega}^2_{\cX_0(p^r)}= -4g_{p^r}(g_{p^r}-1)H_m^2 + V_m^2+h_m.
 \end{align*}
 Now, substituting $H_m=(H_0+H_\infty)/2,\, V_m=(V_0+V_\infty)/2$ and $h_m=(h_0+h_\infty)/2$, we get
  \begin{align}\label{omega2}
     \overline{\omega}^2_{\cX_0(p^r)}= -2g_{p^r}(g_{p^r}-1)\left({H^2_0+H^2_\infty}\right) + \frac{V^2_0+V^2_\infty}{2}+\frac{h_0+h_\infty}{2}.
 \end{align}
 Consider the divisor $D_{\infty} - D_0 = (2g_{p^r}-2)(H_0 - H_{\infty}) + (V_{\infty} - V_0)$. 
  The generic fiber of the line bundle corresponding to the above divisor is supported on cusps. Hence by the  
  Manin-Drinfeld theorem \cite{MR0314846}, \cite{MR0318157}, $D_{\infty}- D_0$ is a torsion element of the Jacobian $J_0(p^r)$. Moreover the divisor $D_{\infty}- D_0$ satisfies the hypothesis of the Faltings-Hriljac theorem, 
  which along with the vanishing of N\'eron-Tate height at torsion points implies 
  $ \langle D_0- D_{\infty},  D_0- D_{\infty}\rangle=0$. Hence, we obtain
  \begin{equation}\label{eq2}
    H_0^2 + H_{\infty}^2 = 2 \langle H_0, H_{\infty} \rangle + 
                           \frac{V_0^2 - 2 \langle V_0, V_{\infty}\rangle + V_{\infty}^2}{(2g_{p^r}-2)^2}.
  \end{equation}
  By substituting \eqref{eq2} in \eqref{omega2} we get our required formula.
\end{proof}
\begin{lem}\label{NeronTateheight}
Let $h_m$ be same as in Proposition \ref{Propomega}. Then we have
 \begin{equation*}
    h_m=
    \begin{cases}
    0 & \text{if $p \equiv 11 \pmod {12}$,} \\
    O(\log p)&  \text{if $p \not \equiv 11 \pmod {12}$}.
    \end{cases}
  \end{equation*}
\end{lem}
\begin{proof}
   Note that, for  $p \equiv 11 \pmod {12}$, the modular curve $X_0(p^r)$ ($r=3,4$) has no elliptic fixed points. In this case, for 
  $m \in \{0, \infty\}$, the divisors $D_m$ are  supported at cusps and hence $h_m=0$ (see 
  \cite[Lemma 4.1.1]{AbbesUllmo})
  
   When $p \not\equiv 11 \pmod {12}$, one can express the height $h_m$ in terms of the heights of the Heegner points of $X_0(p^r)$ ($r=3,4$) associted to $\mathbb{Q}\left(\sqrt{-1}\right)$ and
   $\mathbb{Q}\left(\sqrt{-3}\right)$ when these points exists (See \cite[section 6]{MichelUllmo} and \cite[p. 307]{MR833192}).
  Let $i$ and $j$ be the points 
  on $X_0(1)$ corresponding to the points $i$ and $j=e^{\frac{2 i\pi }{3}}$ of the complex upper-half plane $\hh$. Let $\mathcal{H}_i$ (resp. $\mathcal{H}_j$) be the divisor of $X_0(p^r)$ consisting of elliptic fixed points
  lying above $i$ (resp.  $j$). Its degree is
  \begin{align*}
      \nu_2=1+\left(\frac{-1}{p}\right), \,\ \left( \text{resp.}\,\ \nu_3=1+\left(\frac{-3}{p}\right)\right).
  \end{align*}
  From \cite[Lemma 6.1]{MichelUllmo} (see also \cite{BanerjeeBorahChaudhuri}), we have
  \begin{align}\label{eq:heegner}
       h_m = \frac{1}{36} h_{NT}\left( 3(\mathcal{H}_i-\nu_2 \infty)+4(\mathcal{H}_j-\nu_3 \infty)\right).
  \end{align}
  Let $f_{p^r}: X_0(p^r) \rightarrow X_0(1)$ be the natural projection. 
  In \cite[Lemma 6.2]{MichelUllmo} the authors showed that the preimages of $i$ (resp. $j$) under $f_{p^r}$ with ramification
  index $1$ are Heegner points of discriminant $-4$ (resp. $-3$), these are precisely the elliptic fixed points of $X_0(p^r)$ lying
  over $i$ (resp. j).
  
  Now, let $c$ be an elliptic fixed point of $X_0(p^r)$ lying above $i$ or $j$. By \cite[p. 307]{MichelUllmo}, we have
 \[
  h_{NT}((c)-(\infty))=<c,c>_{\infty}+<c, c>_{\mathrm{fin}}. 
 \]
Similarly as in \cite{BanerjeeBorahChaudhuri}, we compute that  $<c, c>_{\mathrm{fin}}=2 \log(p^r)$ if $c \in X_0(p^r)$ is a Heegner point lying above $i$ (resp.  $<c, c>_{\mathrm{fin}}
  =3 \log(p^r)$ if $c$ lies above $j$). 
  
  The simplification of $ <c,c>_{\infty}$ follows from \cite[Section 6, p. 671]{MichelUllmo}. Recall that
  \[
  <c,c>_{\infty}= \lim_{s \to 1} \left( H(s)+
  \frac{4 \pi}{v_{\Ga_0(p^r)}(s-1)} \right)+O(1); 
 \]
 with 
 \[
 H(s)=- 8 \sum_{n=1}^{\infty} \sigma(n) r(p^2n+4) Q_{s-1}\left(1+\frac{np^2}{2}\right). 
 \]
 In the above expression, $\sigma(n)$ is the function as defined in  \cite[Prop (3.2), Chap IV]{MR833192} with $|\sigma(n)| \leq \tau(n)$  and $\tau(n)$ the number of positive divisors of $n$, $r(n)$ is the 
 number of ideals of norm $n$ in $\Z[i]$ (resp. in $\Z[\frac{1+\sqrt{-3}}{2}]$) and $Q_{s-1}(x)$ is the Legendre function of second kind \cite[p. 238]{MR833192}. We have an estimate $r(n)=O_{\epsilon}(n^{\epsilon})$ for any $\epsilon>0$, .

 Following the same computations \cite[p. 1327]{BanerjeeBorahChaudhuri}, we have the following estimate
  \begin{align*}
      \lim_{s \to 1} \left( H(s)+\frac{4\pi}{v_{\Ga_0(p^r)}(s-1)} \right)=O_{\epsilon}(p^{n \epsilon-n}).
  \end{align*}
From the above, we get an estimate  $<c,c>_{\infty}=O_{\epsilon}(p^{n\epsilon-n})$.
  Then finally from \cite[Section 6, p. 673]{MichelUllmo}, we obtain
  \begin{align}\label{NeronTateheightp^r}
       h_m = O\left(\log p\, \left(\tau(p^r)\right)^2\right).
  \end{align}
 Now, note that for $r=3,4$ we have $\tau(p^r)=4, 5$, which completes the proof.
\end{proof}
\begin{rem}
From \eqref{NeronTateheightp^r} we can say: For the modular curve $X_0(p^r)$ with any positive integer $r$ we have $ h_m = O\left((r+1)^2\log p\right)$, whenever such $h_m$ exists. Then for such fixed $r$ Lemma \ref{NeronTateheight} is still remain valid.
\end{rem}
\begin{thm}
For $r \in \{3,4\}$, the Arakelov self intersection numbers for the modular curve $X_0(p^r)$ satisfy the following asymptotic formula
\begin{align*}
    \overline{\omega}^2_{\cX_0(p^r)}=3g_{p^r}\log (p^r)+o(g_{p^r}\log p)\,\ \text{as}\,\ p\to \infty.
\end{align*}
\end{thm}
\begin{proof}
  From Proposition \ref{Propomega}, we have
   \begin{align*}
       \overline{\omega}^2_{\cX_0(p^r)}=&-4g_{p^r}(g_{p^r}-1)  \langle H_0, H_{\infty}\rangle+\frac{1}{g_{p^r}-1}\left(g_{p^r}\langle V_0, V_{\infty}\rangle-\frac{V^2_0+V^2_\infty}{2}\right)+h.
   \end{align*}
 Using the fact that $2\langle H_0, H_{\infty}\rangle={\mathcal{G}_{\mathrm{can}}(0, \infty)}$, we have 
 \begin{align}\label{omega1}
       \overline{\omega}^2_{\cX_0(p^r)}=&-2g_{p^r}(g_{p^r}-1) \, \mathcal{G}_{\mathrm{can}}(0, \infty)+\frac{1}{g_{p^r}-1}\left(g_{p^r}\langle V_0, V_{\infty}\rangle-\frac{V^2_0+V^2_\infty}{2}\right)+h.
   \end{align}
   Now, from \cite[Theorem 5.6.2]{Majumder:Thesis} (see also \cite[Theorem 1.2]{vonPippichMajumder}), we have 
   \begin{align*}
       -2g_{p^r}(g_{p^r}-1) \, \mathcal{G}_{\mathrm{can}}(0, \infty)=2g_{p^r}\log (p^r)+o(g_{p^r}\log p)\,\ \text{as}\,\ p\to \infty.
   \end{align*}
   From Section \ref{Construction_with_sage_power3} and Section \ref{Construction_with_sage_power4}, we have
   \begin{align*}
       &\frac{1}{g_{p^r}-1}\left(g_{p^r}\langle V_0, V_{\infty}\rangle-\frac{V^2_0+V^2_\infty}{2}\right)= g_{p^r}\log (p^r)+o(g_{p^r}\log p)\,\ \text{as}\,\ p\to \infty.
   \end{align*}
   Finally, Lemma \ref{NeronTateheight} shows that $h=o(g_{p^r}\log p)\,\ \text{as}\,\ p\to \infty$. The proof of the theorem directly follows from \eqref{omega1}.
\end{proof}

\bibliographystyle{crelle}
\bibliography{Eisensteinquestion.bib}
\end{document}